\documentclass[10pt, twoside]{amsart}
\usepackage{amssymb, latexsym, amsmath}
\usepackage{graphicx,latexsym, epsfig}
\usepackage[all]{xy}
\usepackage{hyperref}

\numberwithin{equation}{section}

\theoremstyle{plain}
\newtheorem{theo}{Theorem}[section]
\newtheorem{coro}[theo]{Corollary}
\newtheorem{lemm}[theo]{Lemma}
\newtheorem{prop}[theo]{Proposition}

\theoremstyle{definition}

\newtheorem{exam}[theo]{Example}

\theoremstyle{remark}

\newcommand{\NNN}{\mathbb{N}} 
\newcommand{\ZZZ}{\mathbb{Z}} 
\newcommand{\ZP}{ \mathbb{\ZZZ}_{p}} 
\newcommand{\QQQ}{\mathbb{Q}} 
\newcommand{\QP}{\mathbb{\QQQ}_{p}} 
\newcommand{\QPB}{\overline{\QQQ}_{p}} 

\newcommand{\FFF}{\mathbb{F}} 
\newcommand{\FP}{\mathbb{F}_{p}} 
\newcommand{\FPB}{\overline{\mathbb{F}}_{p}} 

\newcommand*{\rom}[1]{\uppercase\expandafter{\romannumeral #1\relax}}

\newcommand{\maxi}[1]{\mathfrak{m}_{#1}} 
\newcommand{\roi}{\mathcal{O}_{E}}
\newcommand{\maxid}{\mathfrak{m}_{E}}
\newcommand{\GL}[1]{\mathrm{GL}_{#1}} 

\newcommand{\rnk}{\mathrm{rk}}

\newcommand{\Hom}{\mathrm{Hom}}
\newcommand{\sequiv}{\overset{\centerdot}\equiv}

\newcommand{\baseo}{\emph{e}_{1}}
\newcommand{\baset}{\emph{e}_{2}}
\newcommand{\baseth}{\emph{e}_{3}}
\newcommand{\Baseo}{\emph{E}_{1}}
\newcommand{\Baset}{\emph{E}_{2}}
\newcommand{\Baseth}{\emph{E}_{3}}
\newcommand{\hodgen}{\mathrm{t}_{H}}
\newcommand{\newtonn}{\mathrm{t}_{N}}
\newcommand{\val}[1]{v_{p}(#1)}
\newcommand{\Val}[1]{v_{p}\big{(}#1\big{)}}
\newcommand{\vphi}{\varphi}

\newcommand{\lambdat}{\eta}

\newcommand{\filo}{\mathrm{Fil}^{1}}
\newcommand{\filt}{\mathrm{Fil}^{2}}
\newcommand{\filr}{\mathrm{Fil}^{r}}

\newcommand{\fil}[1]{\mathrm{Fil}^{#1}}
\newcommand{\mfd}{\mathfrak{D}}
\newcommand{\mfl}{\mathfrak{L}}
\newcommand{\mflo}{\mathfrak{L}_{1}}
\newcommand{\mflt}{\mathfrak{L}_{2}}

\newcommand{\mfm}{\mathfrak{M}}
\newcommand{\mcm}{\mathcal{M}}
\newcommand{\mcf}{\mathcal{F}}

\newcommand{\Bdr}{\mathbf{B}_{\mathrm{dR}}}

\newcommand{\Bst}{\mathbf{B}_{\mathrm{st}}}

\newcommand{\Dst}{\mathrm{D_{st}}}

\newcommand{\ttt}{\mathbf{t}}

\hypersetup{
    colorlinks,%
    citecolor=black,%
    filecolor=black,%
    linkcolor=blue,%
    urlcolor=black
}

\begin{document}

\title{Reduction modulo $p$ of certain semi-stable representations}

\author{Chol Park}

\address{Department of Mathematics \\
		The University of Toronto}
		
\email{cpark@math.toronto.edu}

\maketitle
\begin{abstract}
Let $p>3$ be a prime number and let $G_{\QP}$ be the absolute Galois group of $\QP$. In this paper, we find Galois stable lattices in the three-dimensional irreducible semi-stable and non-crystalline representations of $G_{\QP}$ with Hodge--Tate weights $(0,1,2)$ by constructing their strongly divisible modules. We also compute the Breuil modules corresponding to the mod $p$ reductions of the strongly divisible modules, and determine which of the representations has an absolutely irreducible mod $p$ reduction.
\end{abstract}

\setcounter{tocdepth}{1}
\tableofcontents

\section{Introduction}
Let $p>3$ be a prime number and let $E$ be a finite extension of $\QP$. We write $G_{\QP}$ for the Galois group $\mathrm{Gal}(\QPB/\QP)$ and $I_{\QP}$ for the inertia subgroup of $G_{\QP}$. In this paper, we construct strongly divisible modules of the admissible filtered $(\phi,N)$-modules that correspond to the $3$-dimensional irreducible semi-stable and non-crystalline $E$-representations of $G_{\QP}$ with Hodge-Tate weights $(0,1,2)$. By a result of Liu~\cite{Liu}, this is equivalent to constructing Galois stable lattices in the semi-stable representations.  We also compute the Breuil modules corresponding to the mod $p$ reductions of the strongly divisible modules to determine which of the semi-stable representations has an absolutely irreducible mod $p$ reduction. As a consequence, if $\bar{\rho}:G_{\QP}\rightarrow\mathrm{GL}_{3}(\FPB)$ is an irreducible mod $p$ reduction of a semi-stable and non-crystalline representation with Hodge-Tate weights $(0,1,2)$, then
$\bar{\rho}|_{I_{\QP}}$ is isomorphic to either $$\omega_{3}^{2p+1}\oplus\omega_{3}^{2p^{2}+p}\oplus\omega_{3}^{2+p^{2}}\hspace{0.1cm}\mbox{  or  }\hspace{0.1cm} \omega_{3}^{p+2}\oplus\omega_{3}^{p^{2}+2p}\oplus\omega_{3}^{1+2p^{2}}$$
where $\omega_{3}$ is the fundamental character of level $3$.

This work is the second part of the project in which we construct the irreducible components of deformation spaces whose characteristic $0$ closed points are the semi-stable lifts with Hodge--Tate weights $(0,1,2)$ of a fixed irreducible representation $\bar{r} : G_{\QP} \rightarrow \mathrm{GL}_3(\overline{\mathbb{F}}_p)$ and compute the Hilbert--Samuel multiplicities of their special fibers, following the strategy in~\cite{BM,Savitt}. The existence of these deformation rings was proved by Kisin~\cite{Kisin}. The geometric structure of these local deformation spaces is described by a conjecture of Breuil--M\'ezard~\cite{BM} as well as a recent refinement due to  Emerton--Gee~\cite{EG}.  Thanks to the work of Kisin~\cite{KisinF}, this conjecture is known for $\mathrm{GL}_2$ over $\QP$.  Gee and Kisin
\cite{GK} have recently proved the Breuil--M\'ezard conjecture for $2$-dimensional potentially Barsotti-Tate representations of $G_{K}$, where $K$ is a finite unramified extension of $\QP$.  This paper begins to address the Breuil--M\'ezard conjecture for $\mathrm{GL}_{3}$.

In the paper~\cite{Park}, which is the first part of the project, all the $3$-dimensional semi-stable $E$-representations of $G_{\QP}$ with regular Hodge--Tate weights have been classified by determining the admissible filtered $(\phi,N)$-modules of Hodge type $(0,r,s)$ for $0<r<s$. There are 49 families of admissible filtered $(\phi,N)$-modules of dimension $3$. Among them, there are $26$ families with $N=0$ (i.e., the crystalline case), there are $20$ families with $\mathrm{rank} N=1$, and there are $3$ families with $\mathrm{rank}N=2$. However, if we restrict our attention to those families which contain representations that are irreducible and of Hodge type $(0,1,2)$, there are only $11$ families for $N=0$ and $7$ families for $\mathrm{rank} N=1$; there are none for $\mathrm{rank} N=2$.  Since we are concerned only with the case of absolutely irreducible residual representations and since the crystalline deformation rings are already determined by a result of Clozel--Harris--Taylor~\cite{CHT}, the $7$ families of $\mathrm{rank}N=1$
$$D_{\rnk N=1}^{4},D_{\rnk N=1}^{6},D_{\rnk N=1}^{10},D_{\rnk N=1}^{12},D_{\rnk N=1}^{17},D_{\rnk N=1}^{18},D_{\rnk N=1}^{20}$$ (defined in~\cite{Park}) are the ones we will consider in this paper.

Finding a strongly divisible module of a given admissible filtered $(\vphi,N)$-module is in general very subtle and difficult even when Hodge--Tate weights are small.  An iterative process for the construction of strongly divisible modules is given in~\cite{Breuil99}, but it is rather elaborate to execute in practice (and much more so in dimension~$3$ than in dimension~$2$). Some of the families listed above can be expected to be more difficult than others, and exhibit new features that do not occur in the $\mathrm{GL}_2$-setting.  For instance, there are families with two $\mfl$-invariants in the filtration, and our construction will produce strongly divisible modules that have coefficients defined as limits of sequences in $E$, which depend on the values of the parameters in the families of admissible filtered $(\phi,N)$-modules.

This paper is organized as follows. In the reminder of the introduction, we give a brief review of $p$-adic Hodge theory (filtered $(\phi,N)$-modules, strongly divisible modules, and Breuil modules) and introduce notation that will be used throughout the paper.  In Section~\ref{Toy Breuil modules}, we study some examples of Breuil modules of weight~$2$ which occur as mod $p$ reductions of semi-stable representations of $G_{\QP}$ with Hodge--Tate weights $(0,1,2)$. In Section \ref{sec-adm filt mod}, we glue the seven families of admissible filtered $(\phi,N)$-modules of rank$N=1$ together so that, as a consequence, there are two families $D_{[0,\frac{1}{2}]}$ and $D_{[\frac{1}{2},1]}$ that parameterize all the irreducible semi-stable and non-crystalline representations of $G_{\QP}$ with Hodge--Tate weights $(0,1,2)$. In Section~\ref{sec-strong div mod0}, we construct strongly divisible modules of $D_{[0,\frac{1}{2}]}$. We first divide the area in which the parameters of $D_{[0,\frac{1}{2}]}$ are defined into three pieces and then construct strongly divisible modules for each case. We do similar things for $D_{[\frac{1}{2},1]}$ in Section~\ref{sec-strong div mod1}. In Section~\ref{sec-Breuil mod}, we compute the Breuil modules corresponding to the mod $p$ reductions of the strongly divisible modules constructed in Sections~\ref{sec-strong div mod0} and \ref{sec-strong div mod1}, and use these Breuil modules together with the results in Section~\ref{Toy Breuil modules} to determine which admissible filtered $(\phi,N)$-modules correspond to the representations whose mod $p$ reductions are absolutely irreducible.

\subsection{Review of $p$-adic Hodge theory}
In this subsection, we quickly review filtered $(\phi,N)$-modules, strongly divisible modules, and Breuil modules. Let $K$ and $E$ be finite extensions of $\QP$ inside $\QPB$ and $K_{0}$ the maximal absolutely unramified subextension of $K$. We also let $k$ be the residue field of $K$ and $v_{p}$ be a valuation on $\QPB$ with $\val{p}=1$. We write $G_{K} = \mathrm{Gal}(\QPB/K)$ for the absolute Galois group of $K$.
\subsubsection{Filtered $(\phi,N)$-modules}
To discuss semi-stable representations, we need to understand the semi-stable period ring $\Bst$. Due to the size of the paper, we only summarize some properties of $\Bst$ below, instead of constructing the ring. (See~\cite{Fontaine} for details.) The ring $\Bst$ is a subring of $\Bdr$ (though not canonically so) and contains the maximal unramified extension $\QP^{ur}$ of $\QP$ as well as the element $\ttt\in\Bdr$. Furthermore,
\begin{itemize}
\item the action of $G_{K}$ on $\Bdr$ restricts to a continuous action on $\Bst$ with $\Bst^{G_{K}}=K_{0}$;
\item there is an injective map $\phi:\Bst\rightarrow\Bst$ that commutes with the action of $G_{K}$, is Frobenius-semilinear on $\QP^{ur}$, and satisfies $\phi(\ttt)=p\ttt$;
\item there is an $\QP^{ur}$-linear map $N:\Bst\rightarrow\Bst$ that commutes with the action of $G_{K}$ and satisfies the relation $N\phi=p\phi N$.
\end{itemize}
We let $\mathbf{B}_{\mathrm{st},K}:=K\otimes_{K_{0}}\Bst$ and endow $\mathbf{B}_{\mathrm{st},K}$ with the filtration $\fil{i}\mathbf{B}_{\mathrm{st},K}:=(K\otimes_{K_{0}}\Bst)\cap\fil{i}\Bdr$. Note that the filtration on $\mathbf{B}_{\mathrm{st},K}$ depends on the embedding $\mathbf{B}_{\mathrm{st},K}\hookrightarrow \Bdr$, which depends on the choice of a uniformizer.

A \emph{filtered $(\phi,N)$-module} (strictly speaking, a \emph{filtered $(\phi,N,K,E)$-module}) is a free $K_{0}\otimes_{\QP}E$-module $D$ of finite rank together with a triple $(\phi,N,\{\fil{i}D_{K}\}_{i\in\ZZZ})$ where
\begin{itemize}
\item the \emph{Frobenius map} $\phi:D\rightarrow D$ is a Frobenius-semilinear and $E$-linear automorphism;
\item the \emph{monodromy operator} $N:D\rightarrow D$ is a (nilpotent) $K_{0}\otimes_{\QP}E$-linear endomorphism such that $N\phi=p\phi N$;
\item the \emph{Hodge filtration} $\{\fil{i}D_{K}\}_{i\in\ZZZ}$ is a decreasing filtration on $D_{K}:=K\otimes_{K_{0}}D$ such that a $K\otimes_{\QP}E$-submodule $\fil{i}D_{K}$ is
$D_{K}$ if $i\ll 0$ and $0$ if $i\gg 0$.
\end{itemize}
The morphisms of filtered $(\phi,N)$-modules are $K_{0}\otimes_{\QP} E$-module homomorphisms that commute with $\phi$ and $N$ and that preserve the filtration.

If $D$ is a filtered $(\phi,N)$-module of dimension $n$ as a $K_{0}$-vector space, then we endow $\otimes^{n}_{K_{0}}D$ with the structure of a filtered $(\phi,N)$-module by setting
\begin{itemize}
\item $\phi:=\otimes^{n}\phi$;
\item $N:=\sum_{i=1}^{n}N_{i,1}\otimes N_{i,2}\otimes ...\otimes N_{i,n}$, where $N_{i,j}:=
                 \left\{
                   \begin{array}{ll}
                     N & \hbox{if $i=j$,} \\
                     1 & \hbox{if $i\not=j$;}
                   \end{array}
                 \right.$
\item $\fil{i}(K\otimes_{K_{0}}(\otimes^{n}_{K_{0}}D)):=\underset{i_{1}+i_{2}+\cdots
    + i_{n}=i}\sum \fil{i_{1}}D_{K}\otimes_{K}...\otimes_{K}\fil{i_{n}}D_{K}$.
\end{itemize}
Taking the image structure on $\wedge_{K_{0}}^{n}D$, we endow $\wedge_{K_{0}}^{n}D$ with a filtered $(\phi,N)$-module structure as well. Since $\dim_{K_{0}} \wedge_{K_{0}}^{n}D=1$, we define
\begin{itemize}
\item $\hodgen(D):=\mathrm{max}\{i\in\ZZZ\,|\,\fil{i}(K\otimes_{K_{0}}(\wedge_{K_{0}}^{n}D))\not=0\}$;
\item $\newtonn(D):=\val{\phi(x)/x}$ for a nonzero element $x$ in $\wedge_{K_{0}}^{n}D$.
\end{itemize}

A \emph{filtered $(\phi,N)$-submodule} is a free $(K_{0}\otimes_{\QP}E)$-submodule $D'$ of a filtered $(\phi,N)$-module $D$ that is $\phi$-stable and $N$-stable, in which case $D'$ has a Frobenius map $\phi|_{D'}$, a monodromy operator $N|_{D'}$, and the filtration $\fil{i}D'_{K}=\fil{i}D_{K}\cap D'_{K}$.  A filtered $(\phi,N)$-module $D$ is said to be \emph{admissible} if
\begin{itemize}
\item  $\hodgen(D)=\newtonn(D)$;
\item  $\hodgen(D')\leq\newtonn(D')$ for each filtered $(\phi,N)$-submodule $D'$ of $D$.
\end{itemize}

The \emph{Hodge--Tate weights} (or \emph{Hodge type}) of a filtered $(\phi,N)$-module $D$ are the integers $r$ such that $\fil{r}D_K \neq \fil{r+1}D_K$, each counted with multiplicity $\dim_{E}(\fil{r}D_{K}/\fil{r+1}D_{K}).$  If the rank of $D$ over $K_0 \otimes_{\QP} E$ is $d$, then there are precisely $d \cdot [K:\QP]$ Hodge--Tate weights, with multiplicity.  When a filtered $(\phi,N)$-module $D$ of rank $d$ has $d$ distinct Hodge--Tate weights, we say that $D$ is \emph{regular} (or that it has regular Hodge--Tate weights). We say that a filtered $(\phi,N)$-module is \emph{positive} if the lowest Hodge--Tate weight is greater than or equal to $0$.

Fix a uniformizer, thereby fixing the inclusion
$\mathbf{B}_{\mathrm{st},K}\hookrightarrow \Bdr.$  Let $V$ be a finite-dimensional $E$-vector space equipped with continuous action of $G_{K}$, and define $$\Dst(V):=(\Bst\otimes_{\QP}V)^{G_{K}}.$$  Then $\dim_{K_{0}}\Dst(V)\leq\dim_{\QP}V$.  If the equality holds,
then we say that $V$ is \emph{semi-stable}; in that case $\Dst(V)$ inherits from $\Bst$ the structure of an admissible filtered $(\phi,N)$-module. In particular, $\Dst(V)$ is a free $K_{0}\otimes_{\QP}E$-module of rank $\dim_{E}V$. More precisely,
\begin{itemize}
\item $\phi:=\phi\otimes \mathrm{Id}:\Dst(V)\rightarrow\Dst(V)$;
\item $N:=N\otimes \mathrm{Id}:\Dst(V)\rightarrow\Dst(V)$;
\item $\fil{i}\Dst(V)_{K}:=(\fil{i}\mathbf{B}_{\mathrm{st},K}\otimes_{\QP}V)^{G_{K}}$.
\end{itemize}
We say that $V$ is \emph{crystalline} if $V$ is semi-stable and the monodromy operator $N$ on $\Dst(V)$ is~$0$. Following Colmez and Fontaine~\cite{CF}, the functor $\Dst$ provides an equivalence between the category of semi-stable $E$-representations of $G_{K}$ and the category of admissible filtered $(\phi,N,K,E)$-modules. The functor $\Dst$ does depend on the choice of a uniformizer $\pi$ in $K$, but the $(\phi,N)$-module $\Dst(V)$ (forgetting the filtration) does not depend on $\pi$. $\Dst$ restricted to the category of crystalline representations does not depend on $\pi$ either.

If $V$ is semi-stable, then when we refer to the Hodge--Tate weights or the
Hodge type of $V$, we mean those of $\Dst(V)$.  Our normalizations imply that the cyclotomic character $\varepsilon : G_{\QP} \to E^{\times}$ has Hodge--Tate weight $-1$.  Twisting $V$ by a power $\varepsilon^n$ of the cyclotomic character has the effect of shifting all the Hodge--Tate weights of $V$ by $-n$; after a suitable twist, we are therefore free to assume that the lowest Hodge--Tate weight is $0$.

If $V$ is a finite dimensional vector space over $E$ equipped with a continuous action of $G_{K}$, we let $V^*$ be the dual representation of $G_{K}$. $V$ is semi-stable (resp., crystalline) if and only if so is $V^*$. If we denote $\mathrm{D_{st}^{*}}(V):=\Dst(V^*)$, then the functor $\mathrm{D_{st}^{*}}$ gives rise to an anti-equivalence between the category of semi-stable $E$-representations of $G_{K}$ and the category of admissible filtered $(\phi,N,K,E)$-modules. The quasi-inverse to $\mathrm{D^{*}_{st}}$ is given by $$\mathrm{V^{*}_{st}}(D):=\Hom_{\phi,N}(D,\Bst)\cap\Hom_{\fil{*}}(D_{K}, K\otimes_{K_{0}}\Bst).$$

\subsubsection{Strongly divisible modules}
We fix a uniformizer $\pi$ in $K$ and let $W(k)$ be the ring of Witt vectors over $k$ so that $K_{0}=W(k)[\frac{1}{p}]$. Let $E(u)\in W(k)[u]$ be the minimal polynomial of $\pi$ over $K_{0}$ and let $S$ be the $p$-adic completion of $W(k)[u,\frac{u^{ie}}{i!}]_{i\in\NNN}$, where $e$ is the absolute ramification index of $K$. We endow $S$ with the following structure:
\begin{itemize}
\item a continuous Frobenius-semilinear map $\varphi:S\rightarrow S$ with $\varphi(u)=u^{p}$;
\item a continuous $W(k)$-linear derivation $N:S\rightarrow S$ with $N(u)=-u$ and $N(u^{ie}/i!)=-ieu^{ie}/i!$;
\item a decreasing filtration $\{\fil{i}S\}_{i\in\NNN_{0}}$ where $\fil{i}S$ is the $p$-adic completion of $\sum_{j\geq i}\frac{E(u)^{j}}{j!}S$.
\end{itemize}
Note that $N\varphi=p\varphi N$ and $\varphi(\fil{i}S)\subset p^{i}S$ for $0\leq i\leq p-1$.

Let $E$ be a finite extension of $\QP$ with ring of integers $\roi$. We also let $S_{\roi}:=S\otimes_{\ZP}\roi$ and $S_{E}:=S_{\roi}\otimes_{\ZP}\QP$, and extend the definitions of $\fil{}$, $\varphi$, and $N$ to $S_{\roi}$ and $S_{E}$ by $\roi$-linearly and $E$-linearly, respectively. Let $\mathcal{MF}(\varphi,N,K,E)$ be the category whose objects are finite free $S_{E}$-modules $\mfd$ with
\begin{itemize}
\item a $\varphi$-semilinear and $E$-linear morphism $\varphi:\mfd\rightarrow\mfd$ such that the determinant of $\varphi$ with respect to some choice of $S_{\QP}$-basis is invertible in $S_{\QP}$ (which does not depend on the choice of basis);
\item a decreasing filtration of $\mfd$ by $S_{E}$-submodules $\fil{i}\mfd$, $i\in\ZZZ$, with $\fil{i}\mfd=\mfd$ for $i\leq 0$ and $\fil{i}S_{E}\cdot\fil{j}\mfd\subset\fil{i+j}\mfd
    $ for all $j$ and all $i\geq 0$;
\item a $K_{0}\otimes E$-linear map $N:\mfd
\rightarrow \mfd
$ such that
    \begin{itemize}
    \item $N(sx)=N(s)x+sN(x)$ for all $s\in S_{E}$ and $x\in\mfd$,
    \item $N\varphi=p\varphi N$,
    \item $N(\fil{i}\mfd)\subset\fil{i-1}\mfd$ for all $i$.
    \end{itemize}
\end{itemize}

For a filtered $(\phi,N)$-module $D$ with positive Hodge--Tate weights, one can associate an object $\mfd\in\mathcal{MF}(\varphi,N,K,E)$ by the following:
\begin{itemize}
\item $\mfd:=S\otimes_{W(k)}D$;
\item $\varphi:=\vphi\otimes\phi:\mfd\rightarrow\mfd$;
\item $N:=N\otimes\mathrm{Id}+\mathrm{Id}\otimes N:\mfd\rightarrow\mfd$;
\item $\fil{0}\mfd:=\mfd$ and, by induction, $$\fil{i+1}\mfd:=\{x\in\mfd|N(x)\in\fil{i}\mfd\mbox{ and }f_{\pi}(x)\in\fil{i+1}D_{K}\}$$ where $f_{\pi}:\mfd\twoheadrightarrow D_{K}$ is defined by $s(u)\otimes x\mapsto s(\pi)x$.
\end{itemize}
The functor $\mfd:D\mapsto S\otimes_{W(k)}D$ gives rise to an equivalence between the category of positive filtered $(\phi,N)$-modules and $\mathcal{MF}(\vphi,N,K,E)$, by a result of Breuil~\cite{Breuil97}.

Fix a positive integer $r\leq p-2$. The category $\mathfrak{MD}^{r}_{\roi}$ of \emph{strongly divisible modules of weight $r$} is defined to be the category of free $S_{\roi}$-modules $\mfm$ of finite rank with an $S_{\roi}$-submodule $\fil{r}\mfm$ and additive maps $\vphi,N:\mfm\rightarrow\mfm$ such that the following properties hold:
\begin{itemize}
\item $\fil{r}S_{\roi}\cdot\mfm\subset\fil{r}\mfm$;
\item $\fil{r}\mfm\cap I\mfm=I\fil{r}\mfm$ for all ideals $I$ in $\roi$;
\item $\vphi(sx)=\vphi(s)\vphi(x)$ for all $s\in S_{\roi}$ and for all $x\in\mfm$;
\item $\vphi(\fil{r}\mfm)$ is contained in $p^{r}\mfm$ and generates it over $S_{\roi}$;
\item $N(sx)=N(s)x+sN(x)$ for all $s\in S_{\roi}$ and for all $x\in\mfm$;
\item $N\vphi=p\vphi N$;
\item $E(u) N(\fil{r}\mfm)\subset\fil{r}\mfm$.
\end{itemize}
The morphisms are $S_{\roi}$-linear maps that preserve $\filr$ and commute with $\vphi$ and $N$. For a strongly divisible module $\mfm$ of weight $r$, there exists a unique admissible filtered $(\phi,N)$-module $D$ with Hodge--Tate weights lying in $[0,r]$ such that $\mfm[\frac{1}{p}]\simeq S\otimes_{W(k)}D$, so one has the following equivalent definition: let $D$ be an admissible filtered $(\phi,N)$-module such that $\fil{0}D_{K}=D_{K}$ and $\fil{r+1}D_{K}=0$. \emph{A strongly divisible module} in $\mfd:=\mfd(D)$ is an $S_{\roi}$-submodule $\mfm$ of $\mfd$ such that
\begin{itemize}
\item $\mfm$ is a free $S_{\roi}$-module of finite rank such that $\mfm[\frac{1}{p}]\simeq\mfd$;
\item $\mfm$ is stable under $\vphi$ and $N$;
\item $\vphi(\filr\mfm)\subset p^{r}\mfm$ where $\filr\mfm:=\mfm\cap\filr\mfd$.
\end{itemize}

For a strongly divisible module $\mfm$, we define an $\roi[G_{K}]$-module $\mathrm{T^{*}_{st}}(\mfm)$ as follows: the ring $\widehat{\mathbf{A}}_{\mathrm{st}}$, defined in~\cite{Breuil99}, is an $S$-algebra with a filtration $\fil{i}\widehat{\mathbf{A}}_{\mathrm{st}}$, a Frobenius map $\vphi$, and a monodromy operator $N$. Moreover, $\widehat{\mathbf{A}}_{\mathrm{st}}$ has a natural action of $G_{K}$. We put $$\mathrm{T^{*}_{st}}(\mfm):=\Hom_{S,\filr,\vphi,N}(\mfm,\widehat{\mathbf{A}}_{\mathrm{st}})$$ (that is, the homomorphisms of $S$-modules which preserve $\filr$ and commute with $\vphi$ and $N$). $\mathrm{T^{*}_{st}}(\mfm)$ inherits an $\roi$-module structure from the $\roi$-module structure on $\mfm$ and an action of $G_{K}$ from the action of $G_{K}$ on $\widehat{\mathbf{A}}_{\mathrm{st}}$. The functor $\mathrm{T^{*}_{st}}$ provides an anti-equivalence of categories between the category $\mathfrak{MD}^{r}_{\roi}$ of strongly divisible modules of weight $r$ and the category of $G_{K}$-stable $\roi$-lattices in semi-stable $E$-representations of $G_{K}$ with Hodge--Tate weights lying in $[-r,0]$, provided that $0\leq r\leq p-2$. Moreover, there is a compatibility: if $\mfm$ is a strongly divisible module in $\mfd:=\mfd(D)$ for an admissible filtered $(\phi,N)$-module $D$, then $\mathrm{T^{*}_{st}}(\mfm)$ is a Galois stable $\roi$-lattice in $\mathrm{V^{*}_{st}}(D)$. This was conjectured by Breuil and proved by Liu~\cite{Liu} in the case $E=\QP$; Emerton--Gee--Herzig~\cite{EGH} gave the (essentially formal) generalization to the case of $E$-coefficients.

\subsubsection{Breuil modules}
Let $\FFF$ be a finite extension of $\FP$, $k$ an algebraic extension of $\FP$, and $e\in\NNN$. The category $\mathrm{BrMod}^{r}_{\FFF}$ of \emph{Breuil modules of weight $r$} consists of quadruples $(\mcm,\mcm_{r},\vphi_{r},N)$ where
\begin{itemize}
\item $\mcm$ is a finitely generated $(k\otimes_{\FP}\FFF)[u]/u^{ep}$-module, free over $k[u]/u^{ep}$, (which implies that $\mcm$ is in fact a free $(k\otimes_{\FP}\FFF)[u]/u^{ep}$-module of finite rank);
\item $\mcm_{r}$ is a $(k\otimes_{\FP}\FFF)[u]/u^{ep}$-submodule of $\mcm$ containing $u^{er}\mcm$;
\item $\vphi_{r}:\mcm_{r}\rightarrow\mcm$ is $\FFF$-linear and $\vphi$-semilinear (where $\vphi:k[u]/u^{ep}\rightarrow k[u]/u^{ep}$ is the $p$-th power map) with image generating $\mcm$ as $(k\otimes_{\FP}\FFF)[u]/u^{ep}$-module;
\item $N:\mcm\rightarrow\mcm$ is $k\otimes_{\FP}\FFF$-linear and satisfies
     \begin{itemize}
     \item $N(ux)=uN(x)-ux$ for all $x\in\mcm$,
     \item $u^{e}N(\mcm_{r})\subset \mcm_{r}$, and
     \item $\vphi_{r}(u^{e}N(x))=cN(\vphi_{r}(x))$ for all $x\in\mcm_{r}$, where $c\in (k[u]/u^{ep})^{\times}$ is the image of $\frac{1}{p}\vphi(E(u))$ under the natural map $S\rightarrow k[u]/u^{ep}$.
     \end{itemize}
\end{itemize}
The morphisms are $(k\otimes_{\FP}\FFF)[u]/u^{ep}$-module homomorphisms that preserve $\mcm_{r}$ and commute with $\vphi_{r}$ and $N$.

Suppose that $k$ (resp., $\FFF$) is the residue field of $K$ (resp., of $E$). We also assume $e=[K:K_{0}]$. If $\mfm$ is an object of $\mathfrak{MD}^{r}_{\roi}$, then \begin{equation}\label{BrMod-fuctor}
\mcm:=(\mfm/\maxi{E}\mfm)\otimes_{S_{\roi}/\maxi{E}S_{\roi}}(k\otimes_{\FP}\FFF)[u]/u^{ep}
\end{equation}
is naturally an object of $\mathrm{BrMod}^{r}_{\FFF}$, where
\begin{itemize}
\item $\mcm_{r}$ is the image of $\filr\mfm$ in $\mcm$;
\item the map $\vphi_{r}$ is induced by $\frac{1}{p^{r}}\vphi|_{\filr\mfm}$;
\item $N$ is induced by the one on $\mfm$.
\end{itemize}
Note that this association gives rise to a functor from the category $\mathfrak{MD}^{r}_{\roi}$ to $\mathrm{BrMod}^{r}_{\FFF}$.

We define a functor $\mathrm{T}^{*}$ from the category $\mathrm{BrMod}^{r}_{\FFF}$ to the category of finite-dimensional $\FFF$-representations of $G_{K}$ as follows: the ring $\widehat{\mathbf{A}}$, defined in~\cite{EGH}, is a $k[u]/u^{ep}$-algebra with a filtration $\filr$, maps $\vphi_{r}$ and $N$, and an action of $G_{K}$. For $\mcm\in\mathrm{BrMod}^{r}_{\FFF}$, we define $$\mathrm{T}^{*}(\mcm):=\Hom_{k[u]/u^{ep},\filr,\vphi_{r},N}(\mcm,\widehat{\mathbf{A}}).$$  This inherits on $\mathrm{T^{*}}(\mcm)$ an action of $G_{K}$ from the action of $G_{K}$ on $\widehat{\mathbf{A}}$. The functor $\mathrm{T}$ is faithful with $\dim_{\FFF}\mathrm{T}^{*}(\mcm)=\mathrm{rank}_{(k\otimes_{\FP}\FFF)[u]/u^{ep}}\mcm$. Moreover, there is a compatibility, that is, if $\mfm\in\mathfrak{MD}^{r}_{\roi}$ and $\mcm:=(\mfm/\maxi{E}\mfm)\otimes_{S_{\roi}/\maxi{E}S_{\roi}}(k\otimes_{\FP}\FFF)[u]/u^{ep}$ denotes the Breuil module corresponding to the reduction of $\mfm$, then $\mathrm{T^{*}_{st}}(\mfm)\otimes_{\roi}\FFF\simeq\mathrm{T}^{*}(\mcm)$.

We say a morphism of Breuil modules $f:\mcm\rightarrow\mcm'$ is a \emph{quotient map} if $f(\mcm_{r})=\mcm_{r}'$. If $f:\mcm\twoheadrightarrow\mcm'$ is a quotient map of Breuil modules, then it is clear that $\mathrm{T}^{*}(f):\mathrm{T}^{*}(\mcm')\hookrightarrow\mathrm{T}^{*}(\mcm)$, i.e., $\mathrm{T}^{*}(\mcm')$ is a subrepresentation of $\mathrm{T}^{*}(\mcm)$.
Moreover, the converse is also true, due to the Proposition~3.2.6 in~\cite{EGH}: if $\mcm\in\mathrm{BrMod}^{r}_{\FFF}$ and if $T'$ is a subrepresentation of $\mathrm{T}^{*}(\mcm)$, then there is a unique quotient map $f:\mcm\twoheadrightarrow\mcm'$ in $\mathrm{BrMod}^{r}_{\FFF}$ such that $\mathrm{T}^{*}(f)$ is identified with the inclusion $T'\hookrightarrow \mathrm{T}^{*}(\mcm)$.

\subsection{Notation}
We let $S$ be the $p$-adic completion of $\ZP[\frac{u^{i}}{i!}]_{i\in\NNN}$ since we are concerned only with representations of $G_{\QP}$ in this paper, and we fix a prime number $p$ to fix an embedding $\Bst\hookrightarrow\Bdr$, and we let $E(u):=u-p\in S$. We assume that $p>3$ since we are concerned with strongly divisible modules of weight $2$. We let $E$ be a finite extension of $\QP$ with ring of integers $\roi$, maximal ideal $\maxi{E}$, and residue field $\FFF$: the field $E$ is the coefficients of our semi-stable representations, $S_{\roi}:=S\otimes \roi$ is the coefficient of our strongly divisible modules, and $\mcf:=\FFF[u]/u^{p}$ is the coefficient of our Breuil modules. We also let $S_{E}:=S_{\roi}\otimes_{\ZP}\QP$. We write $\overline{a}\in\FFF$ for the image of $a\in\roi$ under the fixed quotient map $\roi\rightarrow\FFF$, and let $v_{p}$ be the valuation on $\QPB$ with $\val{p}=1$. For $a,b,c\in E$, we denote $a\equiv b$ modulo $(c)$ if $\val{a-b}\geq\val{c}$ and $a\sequiv b$ modulo $(c)$ if $\val{a-b}>\val{c}$. We let $\gamma:=\frac{(u-p)^{p}}{p}\in S$. It is easy to check that $\vphi(\gamma)\in p^{p-1}S$ and $\vphi(\frac{u-p}{p})=\frac{u^{p}-p}{p}\equiv\gamma-1$ modulo $pS$. It is also straightforward to check $N(\gamma)=-p[\gamma+(u-p)^{p-1}]$.

In this paper, there are eight sequences denoted either by $G_{m}$ or by $H_{m}$, but they are all different. It will be clear from the context which sequences are being used.

\section{Examples of Breuil modules of weight $2$}\label{Toy Breuil modules}
In this section, we provide some examples of Breuil modules which occur as mod $p$ reductions of semi-stable representations of $G_{\QP}$ with Hodge-Tate weights $(0,1,2)$.

\subsection{Simple Breuil modules}
The Breuil modules introduced in this subsection correspond to absolutely irreducible mod $p$ representations of $G_{\QP}$. We prove it by showing the restriction of the corresponding representations to the inertia subgroup $I_{\QP}$ of $G_{\QP}$ is of niveau $3$.
\begin{exam} \label{simple Breuil module exam}
Let $s:=(1,2,3)$ be a cycle of length $3$ in the symmetric group $S_{3}$. For $i=1,2$ and for $a,b,c$ in $\FFF^{\times}$, the Breuil module $\mathcal{M}(s^{i},a,b,c)$ is defined as follows:
     \begin{itemize}
     \item $\mcm:=\mcf(\Baseo,\Baset,\Baseth)$;
     \item $\mcm_{2}:=\mcf(u^{2}\Baseo,u\Baset,\Baseth)$;
     \item $\vphi_{2}:\mcm_{2}\rightarrow \mcm$ is induced by
     $\left\{
       \begin{array}{ll}
         u^{2}\Baseo\mapsto a E_{s^{i}(1)}  & \hbox{} \\
         u\Baset\mapsto bE_{s^{i}(2)} & \hbox{} \\
         \Baseth\mapsto cE_{s^{i}(3)}; & \hbox{}
       \end{array}
     \right.$
     \item $N:\mcm\rightarrow\mcm$ is induced by $N(\Baseo)=N(\Baset)=N(\Baseth)=0$.
     \end{itemize}
\end{exam}

\begin{lemm}\label{simple Breuil module lemma1}
$\mcm(s^{i},a,b,c)$ is isomorphic to $\mcm(s^{j},\alpha,\beta,\gamma)$ if and only if $i=j$ and $abc=\alpha\beta\gamma$.
\end{lemm}

\begin{proof}
It is easy to check that if $i\not=j$, then the only morphism between $\mcm(s^{i},a,b,c)$ and $\mcm(s^{j},\alpha,\beta,\gamma)$ is the trivial map from the commutativity with $\vphi_{2}$.
If $i=j$, then the commutativity with $\vphi_{2}$ also implies that the morphism is of the form $\Baseo\mapsto x\Baseo, \Baset\mapsto y\Baset, \Baseth\mapsto z\Baseth$ for $x,y,z\in\FFF$. If $i=j=1$, then, from the commutativity with $\vphi_{2}$ again, we have equations $\alpha x=ay, \beta y=bz, \gamma z=c x$, which implies $\alpha\beta\gamma=abc$ if we assume that the morphism is an isomorphism. It is easy to check that the morphism commutes with $N$ since $N(\Baseo)=N(\Baset)=N(\Baseth)=0$ and $x,y,z\in \FFF^{\times}$. Similarly, one can also get the same result when $i=j=2$.

Conversely, assume that $abc=\alpha\beta\gamma$. If $i=j=1$, then the association $\Baseo\mapsto\Baseo, \Baset\mapsto\frac{\alpha}{a}\Baset, \Baseth\mapsto\frac{\alpha\beta}{ab}\Baseth$ gives rise to an isomorphism from $\mcm(s^{i},a,b,c)$ to $\mcm(s^{j},\alpha,\beta,\gamma)$, and if $i=j=2$, then the association $\Baseo\mapsto\Baseo, \Baset\mapsto\frac{b}{\beta}\Baset, \Baseth\mapsto\frac{bc}{\beta\gamma}\Baseth$ does so.
\end{proof}

We use Theorem 5.2.2 in~\cite{Caruso} to prove that the Breuil modules in Example~\ref{simple Breuil module exam} correspond to absolutely irreducible mod $p$ representations of $G_{\QP}$. To use the theorem, we need a little preparation. Let $n=[\FFF:\FP]$ and $\sigma$ be the absolute arithmetic Frobenius on $\FFF$, so that the Galois group $\mathrm{Gal}(\FFF/\FP)$ consists of $\sigma^{i}$ for $i=1,2,...,n$. The association $k\otimes e\mapsto(k\cdot\sigma^{i}(e))_{i}$ gives rise to an isomorphism $\FPB\otimes_{\FP}\FFF\underset{\sim}\longrightarrow\oplus_{\sigma^{i}:\FFF\hookrightarrow \FPB}\FPB$ as $i$ ranges over the integers from $1$ to $n$. Note that $\vphi_{r}$ acts on $\FPB$ Frobenius-semilinearly and on $\FFF$ linearly for the Breuil modules over $(\FPB\otimes_{\FP}\FFF)[u]/u^{p}$.

\begin{equation*}
\xymatrix{
\underset{\underset{\vphi_{r}}\circlearrowleft}{(\FPB\otimes_{\FP}\FFF)[u]/u^{p}}&\underset{\underset{\vphi_{r}}\circlearrowleft}{\FPB\otimes_{\FP}\FFF}\ar[r]_{\sim}\ar@{_{(}->}[l] &\underset{\underset{\vphi_{r}}\circlearrowleft}{\oplus_{\sigma^{i}:\FFF\hookrightarrow \FPB}\FPB}}
\end{equation*}
We first investigate the action of $\vphi_{r}$ on $\oplus_{\sigma^{i}:\FFF\hookrightarrow \FPB}\FPB$ under the isomorphism above.

\begin{lemm} \label{simple Breuil module lemma2}
If $(x_{i})_{i}\in\oplus_{\sigma^{i}:\FFF\hookrightarrow \FPB}\FPB$, then $\vphi_{r}((x_{i})_{i})=(x_{i-1}^{p})_{i}$.
\end{lemm}
\begin{proof}
If $x\otimes y$ be in $\FPB\otimes_{\FP}\FFF$, then $\vphi_{r}(x\otimes y)=x^{p}\otimes y$. Hence, we have $\vphi_{r}((x\sigma^{i}(y))_{i})=(x^{p}\sigma^{i}(y))_{i}=((x\sigma^{i-1}(y))^{p})_{i}$. The fact that field $\FPB$ has characteristic $p$ completes the proof.
\end{proof}

We let $\FPB\mcm(s^{i},a,b,c):=\FPB\otimes_{\FP}\mcm(s^{i},a,b,c)$ and extend $\vphi_{2}$ $\vphi$-semilinearly and $N$ linearly on $\FPB$. Then $\FPB\mcm(s^{i},a,b,c)$ is a Breuil module with $(\FPB\otimes_{\FP}\FFF)[u]/u^{p}$-coefficients.

\begin{lemm}\label{simple Breuil module lemma3}
For all $a,b,c\in\FFF^{\times}$, $\FPB\mcm(s^{i},a,b,c)$ is isomorphic to $\FPB\mcm(s^{i},1,1,1)$ if $[\FFF:\FP]=3m$.
\end{lemm}
In the lemma, we don't actually need the assumption $3|[\FFF:\FP]$. But it is good enough for our purpose and gives a shorter proof.

\begin{proof}
By lemma \ref{simple Breuil module lemma1}, it is enough to show that $\FPB\mcm(s^{i},1,1,\alpha)$ is isomorphic to $\FPB\mcm(s^{i},1,1,1)$
We only prove the case $i=1$, and the case $i=2$ is similar.  We let $f$ be a morphism from $\FPB\mcm(s,1,1,\alpha)$ to $\FPB\mcm(s,1,1,1)$ denoted by $$\Baseo\mapsto(x_{i})_{i}\Baseo,\Baset\mapsto(y_{i})_{i}\Baset,\Baseth\mapsto(z_{i})_{i}\Baseth$$ for $x_{i},y_{i},z_{i}\in\FPB$. Then, using the action in Lemma \ref{simple Breuil module lemma2}, one can check that $f$ commutes with $\vphi_{2}$ if and only if $x_{i},y_{i},z_{i}$ satisfy the equations $x_{i-1}^{p}=y_{i}$, $y_{i-1}^{p}=z_{i}$, and $z_{i-1}^{p}=\alpha^{p^{i}}x_{i}$ for $i\in\ZZZ/n\ZZZ$. But it is easy to check that this system of equations have solutions if and only if $x_{1}, x_{2}, x_{3}$ satisfying the equations
$$
x_{1}^{p^{3m}}=\alpha^{p^{7m-3+\frac{3m(m-1)}{2}}}x_{1},\hspace{0.1cm}
x_{2}^{p^{3m}}=\alpha^{p^{8m-3+\frac{3m(m-1)}{2}}}x_{2},\hspace{0.1cm}
x_{3}^{p^{3m}}=\alpha^{p^{9m-3+\frac{3m(m-1)}{2}}}x_{3}
$$
using $[\FFF:\FP]=3m$. It is also easy to check that the map $f$ commutes with $N$ since $N(\Baseo)=N(\Baset)=N(\Baseth)=0$.
\end{proof}

\begin{prop}\label{simple Breuil module prop}
Let $\bar{\rho}:=\mathrm{T}^{*}\mcm(s^{i},a,b,c)\otimes_{\FFF}\FPB$. Then
$$\bar{\rho}|_{I_{\QP}}\simeq
\left\{
  \begin{array}{ll}
    \omega_{3}^{2p+1}\oplus\omega_{3}^{2p^{2}+p}\oplus\omega_{3}^{2+p^{2}}, & \hbox{if $i=1$;} \\
    \omega_{3}^{p+2}\oplus\omega_{3}^{p^{2}+2p}\oplus\omega_{3}^{1+2p^{2}}, & \hbox{if $i=2$,}
  \end{array}
\right.
$$
where $\omega_{3}$ is the fundamental character of level $3$. In particular, $\bar{\rho}$ is absolutely irreducible.
\end{prop}

\begin{proof}
Without loss of generality, we may assume that $3|[\FFF:\FP]$. By lemma \ref{simple Breuil module lemma3}, it is enough to show that the representation of $I_{\QP}$ corresponding to $\FPB\mcm(s^{i},1,1,1)$ is of niveau $3$. But we may regard $\FPB\mcm(s^{i},1,1,1)$ as a Breuil module with $(\FPB\otimes_{\FP}\FP)[u]/u^{p}$-coefficients. By Theorem 4.3.2 in~\cite{Caruso}, this is a simple module and, by Theorem 5.2.2 in~\cite{Caruso}, the corresponding representation is $\omega_{3}^{p+2p^{2}}:I_{\QP}\rightarrow\FFF_{p^{3}}^{\times}$ if $i=1$ and $\omega_{3}^{1+2p^{2}}:I_{\QP}\rightarrow\FFF_{p^{3}}^{\times}$ if $i=2$. Taking $\mathrm{Ind}_{G_{\mathbb{Q}_{p^{3}}}}^{G_{\QP}}$ on these characters completes the proof.
\end{proof}

\subsection{Non-simple Breuil modules}
In this subsection, we introduce a few examples of Breuil modules that correspond to reducible representations of $G_{\QP}$.
\begin{exam}
For a $3\times3$ invertible matrix $(a_{i,j})$ over $\FFF$, the Breuil module
$\mathcal{M}(a_{i,j})$ is defined as follows:
     \begin{itemize}
     \item $\mcm:=\mcf(\Baseo,\Baset,\Baseth)$;
     \item $\mcm_{2}:=\mcf(u\Baseo,u\Baset,u\Baseth)$;
     \item $\vphi_{2}:\mcm_{2}\rightarrow \mcm$ is induced by
     $\left\{
       \begin{array}{ll}
         u\Baseo\mapsto a_{1,1}\Baseo+a_{1,2}\Baset+a_{1,3}\Baseth  & \hbox{} \\
         u\Baset\mapsto a_{2,1}\Baseo+a_{2,2}\Baset+a_{2,3}\Baseth & \hbox{} \\
         u\Baseth\mapsto a_{3,1}\Baseo+a_{3,2}\Baset+a_{3,3}\Baseth; & \hbox{}
       \end{array}
     \right.$
     \item $N:\mcm\rightarrow\mcm$ is induced by $N(\Baseo)=N(\Baset)=N(\Baseth)=0$.
     \end{itemize}
\end{exam}

\begin{prop}\label{non-simple Breuil module prop1}
The corresponding representations to $\mcm(a_{i,j})$ are reducible.
\end{prop}

\begin{proof}
Assume that $\FFF$ is big enough so that the characteristic equation of $(a_{i,j})$ has a solution $d$ in $\FFF$. Note that $d\not=0$ since $(a_{i,j})$ is invertible. We define Breuil modules $\mcm':=\mcf(\Baseo)$ of rank $1$ as follows:
\begin{itemize}
\item $\mcm'_{2}:=\mcf(u\Baseo)$;
\item $\vphi_{2}:\mcm'_{2}\rightarrow \mcm'$ is induced by
     $u\Baseo\mapsto d\Baseo$;
\item $N:\mcm'\rightarrow\mcm'$ is induced by $N(E_{1})=0$.
\end{itemize}
Let $(a,b,c)$ be an eigenvector associated with the eigenvalue $d$. Then the association $\Baseo\mapsto a\Baseo$, $\Baset\mapsto b\Baseo$, $\Baseth\mapsto c\Baseo$ induces a quotient map from $\mcm$ to $\mcm'$. Hence, the corresponding representations are reducible.
\end{proof}

The following two examples also correspond to reducible mod $p$ representations of $G_{\QP}$. We prove it by constructing a non-trivial morphism between these two modules.
\begin{exam}
For $a,b,c,d$ in $\FFF^{\times}$, the Breuil module
$\mathcal{M}(a,b,c,d)$ is defined as follows:
\begin{itemize}
\item $\mcm:=\mcf(\Baseo,\Baset,\Baseth)$;
\item $\mcm_{2}:=\mcf(u^{2}\Baseo,\Baset+d\Baseth,u\Baseth)$;
\item $\vphi_{2}:\mcm_{2}\rightarrow \mcm$ is induced by
     $\left\{
       \begin{array}{ll}
         u^{2}\Baseo\mapsto a\Baseth & \hbox{} \\
         \Baset+ d\Baseth\mapsto b\Baseo & \hbox{} \\
         u\Baseth\mapsto c\Baset; & \hbox{}
       \end{array}
     \right.$
\item $N:\mcm\rightarrow\mcm$ is induced by $N(\Baseo)=N(\Baset)=N(\Baseth)=0$.
\end{itemize}
\end{exam}

\begin{exam}
For $a,b,c,d$ in $\FFF^{\times}$, the Breuil module
$\mathcal{M}'(a,b,c,d)$ is defined as follows:
\begin{itemize}
\item $\mcm':=\mcf(\Baseo,\Baset,\Baseth)$;
\item $\mcm_{2}':=\mcf(u\Baseo+d u\Baset, u^{2}\Baset,\Baseth)$;
\item $\vphi_{2}:\mcm_{2}'\rightarrow \mcm'$ is induced by
     $\left\{
       \begin{array}{ll}
         u\Baseo+d u\Baset\mapsto a\Baset & \hbox{} \\
         u^{2}\Baset\mapsto b\Baseth & \hbox{} \\
         \Baseth\mapsto c\Baseo; & \hbox{}
       \end{array}
     \right.$
\item $N:\mcm'\rightarrow\mcm'$ is induced by $N(\Baseo)=N(\Baset)=N(\Baseth)=0$.
\end{itemize}
\end{exam}

\begin{prop} \label{non-simple Breuil module prop2}
There is a non-trivial morphism between $\mcm(a,b,c,d)$ and $\mcm'(x,y,z,w)$ if $x=-cdw$. Hence, in particular, both of them correspond to reducible representations.
\end{prop}

\begin{proof}
It is routine to check that the association $\Baseo\mapsto 0$, $\Baset\mapsto\Baset$, $\Baseth\mapsto0$ gives rise to a morphism from $\mcm'(x,y,z,w)$ to $\mcm(a,b,c,d)$ if $x=-cdw$.
\end{proof}

\section{Semi-stable representations with Hodge--Tate weights $(0,1,2)$}\label{sec-adm filt mod}
In the paper~\cite{Park}, we have classified all the $3$-dimensional semi-stable representations of $G_{\QP}$ with regular Hodge--Tate weights. The families $D_{\rnk N=1}^{4}$, $D_{\rnk N=1}^{6}$, $D_{\rnk N=1}^{10}$, $D_{\rnk N=1}^{12}$, $D_{\rnk N=1}^{17}$, $D_{\rnk N=1}^{18}$, and $D_{\rnk N=1}^{20}$ in~\cite{Park} are the only ones that contain the irreducible semi-stable and non-crystalline representations of $G_{\QP}$ with Hodge-Tate weights $(0,1,2)$. In this section, we glue these families together. The following two families $D_{[0,\frac{1}{2}]}$ and $D_{[\frac{1}{2},1]}$ parameterize all the filtered modules listed above.

\begin{exam}\label{example D[0,1/2]}
For $\lambda,\lambdat\in \roi$ and $\mflo,\mflt\in E$, we define the admissible filtered $(\phi,N)$-modules $D_{[0,\frac{1}{2}]}=D_{[0,\frac{1}{2}]}(\lambda,\lambdat,\mflo,\mflt)$ as follows:
\begin{itemize}
\item $\fil{i}D=\left\{
                  \begin{array}{ll}
                    D=E(\baseo,\baset,\baseth) & \hbox{if $i\leq0$,} \\
                    E(\baseo+\mflo\baseth,\baset+\mflt\baseth) & \hbox{if $i=1$,} \\
                    E(\baseo+\mflo\baseth) & \hbox{if $i=2$,} \\
                    0 & \hbox{if $i\geq3$;}
                  \end{array}
                \right.
$
\item
$[\phi]=
\begin{small}\left(
  \begin{array}{ccc}
    p\lambda & 0  & 0 \\
    1 & \lambdat & 0 \\
    0 & 0 & \lambda \\
  \end{array}
\right)
\end{small}
$ and
$[N]=
\begin{small}\left(
  \begin{array}{ccc}
    0 & 0  & 0 \\
    0 & 0  & 0 \\
    1 & 0  & 0 \\
  \end{array}
\right);
\end{small}
$
\item $0\leq\val{\lambda}\leq\frac{1}{2}$ and $2\val\lambda+\val\lambdat=2$.
\end{itemize}
\end{exam}
Note that $\lambdat\not=\lambda$ since $0\leq\val\lambda\leq\frac{1}{2}<1\leq\val\lambdat\leq2$.

\begin{prop}\label{Prop-ad-filt-module[0,1/2]}
$D_{[0,\frac{1}{2}]}$ parameterizes $D_{\rnk N=1}^{10}$, $D_{\rnk N=1}^{12}$, $D_{\rnk N=1}^{18}$, and $D_{\rnk N=1}^{20}$ for $0\leq\val{\lambda}\leq\frac{1}{2}$ with Hodge--Tate weights $(0,1,2)$. Moreover, $D_{[0,\frac{1}{2}]}(\lambda,\lambdat,\mflo,\mflt)$ has a submodule if and only if either $\val\lambda=0$ or $\val\lambda=\frac{1}{2}$ and $\mflt=0$.
\end{prop}

\begin{proof}
It is immediate that the identity map gives rise to an isomorphism from $D_{\rnk N=1}^{10}$ and from $D_{\rnk N=1}^{12}$ to $D_{[0,\frac{1}{2}]}$, and so $D_{\rnk N=1}^{10}$ and $D_{\rnk N=1}^{12}$ are sitting on the $p\lambda=\lambdat$ part of $D_{[0,\frac{1}{2}]}$.
The association $\baseo\mapsto (p\lambda-\lambdat)\baseo+\baset$, $\baset\mapsto -\baset$, $\baseth\mapsto(p\lambda-\lambdat)\baseth$ gives rise to an isomorphism from $D_{\rnk N=1}^{18}$ and from $D_{\rnk N=1}^{20}$ to $D_{[0,\frac{1}{2}]}$. Then, by replacing $\mflo$ and $-(p\lambda-\lambdat)\mflt$ with $\mflo$ and $\mflt$ respectively, we see that $D_{\rnk N=1}^{18}$ and $D_{\rnk N=1}^{20}$ cover exactly the $p\lambda\not=\lambdat$ part of $D_{[0,\frac{1}{2}]}$.

For the second part, we know that only $D_{\rnk N=1}^{18}$ and $D_{\rnk N=1}^{20}$ contain reducible semi-stable representations by a result of~\cite{Park}. Since $D_{\rnk N=1}^{18}$ (resp., $D_{\rnk N=1}^{20}$ for $0\leq\val\lambda\leq\frac{1}{2}$) has a submodule if and only if either $\val{\lambda}=0$ or $\val\lambda=\frac{1}{2}$ (resp., $\val\lambda=0$), the statement is now clear from the association above.
\end{proof}

\begin{exam}\label{example D[1/2,1]}
For $\lambda,\lambdat\in \roi$ and $\mflo,\mflt\in E$, we define the admissible filtered $(\phi,N)$-modules $D_{[\frac{1}{2},1]}=D_{[\frac{1}{2},1]}(\lambda,\lambdat,\mflo,\mflt)$ as follows:
\begin{itemize}
\item $\fil{i}D=\left\{
                  \begin{array}{ll}
                    D=E(\baseo,\baset,\baseth) & \hbox{if $i\leq0$,} \\
                    E(\baseo+\mflo\baset+\mflt\baseth,\baset) & \hbox{if $i=1$,} \\
                    E(\baseo+\mflo\baset+\mflt\baseth) & \hbox{if $i=2$,} \\
                    0 & \hbox{if $i\geq3$;}
                  \end{array}
                \right.
$
\item
$
[\phi]=
\begin{small}\left(
  \begin{array}{ccc}
    p\lambda & 0  & 0 \\
    0 & \lambdat & 0 \\
    0 & 1 & \lambda \\
  \end{array}
\right)
\end{small}
$ and
$
[N]=
\begin{small}\left(
  \begin{array}{ccc}
    0 & 0  & 0 \\
    0 & 0  & 0 \\
    1 & 0  & 0 \\
  \end{array}
\right);
\end{small}
$
\item $\frac{1}{2}\leq\val\lambda\leq1$ and $2\val{\lambda}+\val\lambdat=2$.
\end{itemize}
\end{exam}
Note that $p\lambda\not=\lambdat$ since $0\leq\val\lambdat\leq1<\frac{3}{2}\leq\val{p\lambda}\leq2$.

\begin{prop}\label{Prop-ad-filt-module[1/2,1]}
$D_{[\frac{1}{2},1]}$ parameterizes $D_{\rnk N=1}^{4}$, $D_{\rnk N=1}^{6}$, $D_{\rnk N=1}^{17}$, and $D_{\rnk N=1}^{20}$ for $\frac{1}{2}\leq\val{\lambda}\leq1$ with Hodge--Tate weights $(0,1,2)$. Moreover, $D_{[\frac{1}{2},1]}(\lambda,\lambdat,\mflo,\mflt)$ has a submodule if and only if either $\val\lambda=1$ or $\val\lambda=\frac{1}{2}$ and $\mflo=0$.
\end{prop}

\begin{proof}
The identity map gives rise to an isomorphism from $D_{\rnk N=1}^{4}$ to $D_{[\frac{1}{2},1]}$, and so $D_{\rnk N=1}^{4}$ covers the $\lambdat=\lambda$ and $\mflo=0$ part of $D_{[\frac{1}{2},1]}$. The association $\baseo\mapsto\baseo$, $\baset\mapsto\baset-\mflt\baseth$, $\baseth\mapsto\baseth$ gives an isomorphism from $D_{\rnk N=1}^{6}$ to $D_{[\frac{1}{2},1]}$, and so we see that $D_{\rnk N=1}^{6}$ covers the $\lambdat=\lambda$ and $\mflo\not=0$ part of $D_{[\frac{1}{2},1]}$ by replacing $\mflo$ and $-\mflo\mflt$ with $\mflo$ and $\mflt$ respectively. The association $\baseo\mapsto-\baseo$, $\baset\mapsto(\lambdat-\lambda)\baset+\baseth$, $\baseth\mapsto-\baseth$ gives an isomorphism from $D_{\rnk N=1}^{17}$ to $D_{[\frac{1}{2},1]}$, so that we see that $D_{\rnk N=1}^{17}$ covers the $\lambdat\not=\lambda$ and $\mflo=0$ part of $D_{[\frac{1}{2},1]}$. The association $\baseo\mapsto-\baseo$, $\baset\mapsto(\lambdat-\lambda)\mflt\baset+\mflt\baseth$, $\baseth\mapsto-\baseth$ gives an isomorphism from $D_{\rnk N=1}^{20}$ to $D_{[\frac{1}{2},1]}$, and so we see that $D_{\rnk N=1}^{20}$ covers the $\lambdat\not=\lambda$ and $\mflo\not=0$ part of $D_{[\frac{1}{2},1]}$ by replacing $(\lambda-\lambdat)\mflt$ and $\mflo-\mflt$ with $\mflo$ and $\mflt$ respectively.

For the second part, we know that only $D_{\rnk N=1}^{17}$ and $D_{\rnk N=1}^{20}$ contain reducible semi-stable representations by a result of~\cite{Park}. Since $D_{\rnk N=1}^{17}$ (resp., $D_{\rnk N=1}^{20}$ for $\frac{1}{2}\leq\val{\lambda}\leq1$) has a submodule if and only if either $\val{\lambda}=\frac{1}{2}$ or $\val\lambda=1$ (resp., $\val\lambda=1$), the statement is now clear.
\end{proof}

Note that there are no isomorphisms between the modules $D_{[0,\frac{1}{2}]}$ and between the modules $D_{[\frac{1}{2},1]}$ for different values of the parameters $\lambda,\lambdat,\mflo,\mflt$. (See the first part of the proof of the proposition below for the reason.) But there are isomorphisms between $D_{[0,\frac{1}{2}]}$ and $D_{[\frac{1}{2},1]}$ and this happens between the irreducible parts of $D_{[0,\frac{1}{2}]}$ and $D_{[\frac{1}{2},1]}$ only when $\val{\lambda}=\frac{1}{2}$.

\begin{prop}\label{Prop-ad-filt-module}
$D_{[0,\frac{1}{2}]}(\lambda,\lambdat,\mflo,\mflt)$ are isomorphic to $D_{[\frac{1}{2},1]}(\lambda',\lambdat',\mflo',\mflt')$ if and only if
\begin{itemize}
\item $\lambda=\lambda'$ and $\lambdat=\lambdat'$ $(\mbox{and so }\val\lambda=\frac{1}{2}\mbox{ and }\val\lambdat=1)$;
\item $(\lambda-\lambdat)\mflt=(\lambdat-p\lambda)\mflo'$ and $(\lambdat-p\lambda)(\mflo-\mflt')=\mflt$;
\item $\mflt\not=0\not=\mflo'$.
\end{itemize}
\end{prop}

\begin{proof}
We start the proof noting that there are no isomorphisms between $D_{\rnk N=1}^{4}$, $D_{\rnk N=1}^{6}$, $D_{\rnk N=1}^{10}$, $D_{\rnk N=1}^{12}$, $D_{\rnk N=1}^{17}$, $D_{\rnk N=1}^{18}$, and $D_{\rnk N=1}^{20}$ by a result in~\cite{Park}. Hence, the isomorphism between $D_{[0,\frac{1}{2}]}$ and $D_{[\frac{1}{2},1]}$ occurs only between the parts of $D_{[0,\frac{1}{2}]}$ and $D_{[\frac{1}{2},1]}$ on which $D_{\rnk N=1}^{20}$ is sitting, since the $\val\lambda=\frac{1}{2}$ part of $D_{\rnk N=1}^{20}$ is embedded both into $D_{[0,\frac{1}{2}]}$ and into $D_{[\frac{1}{2},1]}$. Hence, the isomorphism only occurs when $\val\lambda=\frac{1}{2}=\val{\lambda'}$ and $\mflt\not=0\not=\mflo'$.

Let $T$ be an isomorphism from $D_{[0,\frac{1}{2}]}(\lambda,\lambdat,\mflo,\mflt)$ to $D_{[\frac{1}{2},1]}(\lambda',\lambdat',\mflo',\mflt')$. Then $T$ preserves the Jordan form of the Frobenius maps and, in particular, their eigenvalues. Hence, $\lambda=\lambda'$ and $\lambdat=\lambdat'$. The commutativity with the monodromy operator $N$ forces that $T$ be of the form $T(\baseo)=a\baseo+b\baset+c\baseth$, $T(\baset)=d\baset+e\baseth$, and $T(\baseth)=a\baseth$. Then the commutativity with the Frobenius maps forces that $T$ be of the form $T(\baseo)=x\baseo+(\lambdat-\lambda)y\baset+y\baseth$, $T(\baset)=(\lambdat-p\lambda)(\lambdat-\lambda)y\baset+(\lambdat-p\lambda)y\baseth$, and $T(\baseth)=x\baseth$. Since $T$ preserves the filtration, we have $(\lambdat-p\lambda)y+\mflt x=0$, and $x\mflo'=(\lambdat-\lambda)y$, and $y+x\mflo =x\mflt'$, which forces $(\lambda-\lambdat)\mflt=(\lambdat-p\lambda)\mflo'$ and $(\lambdat-p\lambda)(\mflo-\mflt')=\mflt$.

It is easy to check that the converse holds. The association
\begin{equation}\label{association from 0 to 1}
\left\{
  \begin{array}{ll}
    \baseo\mapsto\baseo+\mflo'\baset-(\mflo-\mflt')\baseth & \hbox{} \\
    \baset\mapsto(\lambda-\lambdat)\mflt\baset-\mflt\baseth & \hbox{} \\
    \baseth\mapsto\baseth & \hbox{}
  \end{array}
\right.
\end{equation}
gives rise to an isomorphism from $D_{[0,\frac{1}{2}]}(\lambda,\lambdat,\mflo,\mflt)$ to $D_{[\frac{1}{2},1]}(\lambda,\lambdat,\mflo',\mflt')$.
\end{proof}

\begin{coro}
$D_{[0,\frac{1}{2}]}$ and $D_{[\frac{1}{2},1]}$ contain all of the $3$-dimensional irreducible semi-stable and non-crystalline representations of $G_{\QP}$ with Hodge--Tate weights $(0,1,2)$.
\end{coro}

\begin{proof}
By a result of~\cite{Park}, the families $D_{\rnk N=1}^{4}$, $D_{\rnk N=1}^{6}$, $D_{\rnk N=1}^{10}$, $D_{\rnk N=1}^{12}$, $D_{\rnk N=1}^{17}$, $D_{\rnk N=1}^{18}$, and $D_{\rnk N=1}^{20}$ contain all the irreducible semi-stable and non-crystalline representations of $G_{\QP}$ with Hodge-Tate weights $(0,1,2)$. By Propositions~\ref{Prop-ad-filt-module[0,1/2]} and~\ref{Prop-ad-filt-module[1/2,1]}, the two families $D_{[0,\frac{1}{2}]}$ and $D_{[\frac{1}{2},1]}$ parameterize all of the $7$ families above.
\end{proof}

We end this section by noting that there are also a few families of admissible filtered $(\phi,N)$-modules containing only reducible semi-stable and non-crystalline representations of $G_{\QP}$ with Hodge--Tate weights $(0,1,2)$. See~\cite{Park} for details.

\section{Galois stable lattices of $D_{[0,\frac{1}{2}]}$}\label{sec-strong div mod0}
In this section, we construct strongly divisible modules of the modules $D_{[0,\frac{1}{2}]}$ in Example \ref{example D[0,1/2]}. We let $\mfd_{[0,\frac{1}{2}]}:=S\otimes_{\ZP}D_{[0,\frac{1}{2}]}$. In this section, we write $D$ and $\mfd$ for $D_{[0,\frac{1}{2}]}$ and $\mfd_{[0,\frac{1}{2}]}$ respectively for brevity. Since we are concerned only with absolutely irreducible mod $p$ reductions, we may assume $0<\val\lambda\leq\frac{1}{2}$ due to Proposition \ref{Prop-ad-filt-module[0,1/2]}.

It is easy to check that
$$\filo\mfd=S_{E}(\baseo+\mflo\baseth,\baset+\mflt\baseth)+\filo S_{E}\cdot\mfd;$$
$$\filt\mfd=S_{E}\big(\baseo+\mflo\baseth+\frac{u-p}{p}\baseth\big)+\filo S_{E}(\baset+\mflt\baseth)+\filt S_{E}\cdot\mfd.$$ (We omit their proofs.) So every element in $\filt\mfd$ is of the form $C_{0}(\baseo+\mflo\baseth+\frac{u-p}{p}\baseth)+C_{1}(u-p)(\baseo+\mflo\baseth) +C_{2}(u-p)(\baset+\mflt\baseth)+A\baseo+B\baset+C\baseth$, where $C_{0},C_{1},C_{2}$ are in $E$ and $A,B,C$ are in $\filt S_{E}$. We let $$\mathfrak{X}_{0}:=C_{0}\big(\baseo+\mflo\baseth+\frac{u-p}{p}\baseth\big)+(u-p)\big(C_{1}(\baseo+\mflo\baseth) +C_{2}(\baset+\mflt\baseth)\big),$$ which rearranges to $C_{0}\big(\baseo+\mflo\baseth\big)+(u-p)\big(C_{1}\baseo +C_{2}\baset+\frac{C_{0}+p\mflo C_{1}+p\mflt C_{2}}{p}\baseth\big)$.

We divide the area in which the parameters of $D$ are defined into $3$ pieces as follows: for $\lambda,\lambdat\in\roi$ with $0<\val{\lambda}\leq\frac{1}{2}$ and $2\val\lambda+\val\lambdat=2$ and for $\mflo,\mflt\in E$,
\begin{description}
\item[H(0,1)] $\val{\mflo-1}\geq 1-\val\lambda$ and $\val{\mflt+p\lambda}\geq\val{\lambda\lambdat}$;
\item[H(0,2)] $\val{\mflt+p\lambda}\geq\val{p(\mflo-1)}$ and $\val{\mflo-1}< 1-\val\lambda$;
\item[H(0,3)] $\val{\mflt+p\lambda}\leq\val{p(\mflo-1)}$ and $\val{\mflt+p\lambda}<\min\{\val{\lambdat(\mflo-1)},\val{\lambda\lambdat}\}$.
\end{description}
Note that the conditions $\mathbf{H(0,2)}$ and $\mathbf{H(0,3)}$ intersect in $\val{\mflt+p\lambda}=\val{p(\mflo-1)}$ and $\val{\mflo-1}<1-\val\lambda$ if $0<\val\lambda<\frac{1}{2}$. We also note that the condition $\val{\mflt+p\lambda}<\val{\lambdat(\mflo-1)}$ in $\mathbf{H(0,3)}$ matters only when $\val\lambda=\frac{1}{2}$, since $\val{\mflt+p\lambda}\leq\val{p(\mflo-1)}$ implies $\val{\mflt+p\lambda}<\val{\lambdat(\mflo-1)}$ if $0<\val{\lambda}<\frac{1}{2}$. We construct strongly divisible modules for each case in the following three subsections.

In this section, we let $\mathfrak{U}_{0}:=p\baseo+\frac{1}{\lambda}\baset+(\gamma+\mflo-1)\baseth$ for brevity.

\subsection{Strongly divisible modules: the first case}
In this subsection, we construct strongly divisible modules in $\mfd_{[0,\frac{1}{2}]}$ under the assumption $\mathbf{H(0,1)}$ as at the beginning of this section.

\begin{prop}\label{prop-D0first}
Keep the assumption $\mathbf{H(0,1)}$. Then $\mfm_{[0,\frac{1}{2}]}:=S_{\roi}(E_{1},E_{2},E_{3})$ is a strongly divisible module in $\mfd_{[0,\frac{1}{2}]}$, where
    $$\left\{
        \begin{array}{ll}
          E_{1}=\mathfrak{U}_{0} & \hbox{} \\
          E_{2}=\frac{1}{p\lambda}(\lambdat\baset+\lambda\mflt\baseth)-\lambda(\gamma-1)\baseth & \hbox{} \\
          E_{3}=\frac{p}{\lambda}\baseth. & \hbox{}
        \end{array}
      \right.$$
\end{prop}

\begin{proof}
During the proof, we write $\mfm$ for $\mfm_{[0,\frac{1}{2}]}$ for brevity. It is routine to check that $\vphi(E_{1})\equiv\vphi(E_{2})\equiv\vphi(E_{3})\equiv N(\Baseo)\equiv N(\Baset)\equiv N(\Baseth)\equiv 0$ modulo $\maxi{E}\mfm$, using the fact $\vphi(\gamma)\in p^{p-1}S$ and $N(\gamma)=-p[\gamma+(u-p)^{p-1}]$. Hence, $\mfm$ is stable under $\vphi$ and $N$.

We check that $\vphi(\filt \mfm)\subset p^{2}\mfm$. It is easy to check that $\filt S\cdot\mfm\subset\filt\mfm$, and so it is enough to check $\vphi(\mathfrak{X}_{0})\in p^{2}\mfm$, since $\vphi(\filt S\cdot\mfm)\subset p^{2}\mfm$. We first compute  $\filt \mfm$ modulo $\filt S\cdot\mfm$, by finding out the conditions that $\mathfrak{X}_{0}$ be in $\mfm$. $\mathfrak{X}_{0}\equiv C_{0}\Big{(}\frac{1}{p}E_{1}-\frac{1}{\lambdat}\Baset+\frac{\lambda(\mflt+p\lambda)}{p^{2}\lambdat}\Baseth -\frac{\lambda(\mfl-1)}{p^{2}}\Baseth+\frac{\lambda\mflo}{p}\Baseth \Big{)}+ (u-p)\Big{(}\frac{C_{1}}{p}E_{1}- \frac{C_{1}-p\lambda C_{2}}{\lambdat}E_{2}+ \frac{V}{W}E_{3}\Big{)}$ modulo $\filt S\cdot\mfm$, where $W=p^{2}\lambdat$ and
\begin{equation*}
V=\lambda\lambdat(C_{0}+p\mflo C_{1}+p\mflt C_{2})+ \lambda(\mflt+p\lambda)C_{1}- \lambda\lambdat(\mflo-1)C_{1}-p\lambda^{2}(\mflt+p\lambda)C_{2}.
\end{equation*}
Hence, if $\mathfrak{X}_{0}\in\filt \mfm$, then we get $\val{C_{0}}\geq\val{\lambdat}$, $\val{C_{1}}\geq1$,
\begin{equation}\label{D0-first}
\val{C_{1}-p\lambda C_{2}}\geq\val{\lambdat},
\end{equation}
and
\begin{equation}\label{D0-second}
\val{V}\geq \val{W}.
\end{equation}
The inequality (\ref{D0-first}) with $\val{C_{1}}\geq1$ implies $$\val{C_{2}}\geq\min\{\val{C_{1}},\val{\lambdat}\}-\val{p\lambda}\geq-\val\lambda,$$
and one can easily check that $V\equiv\lambda\lambdat C_{0}+\lambda(\mflt+p\lambda)(C_{1}-p\lambda C_{2})-\lambda\lambdat(\mflo-1)C_{1}$ modulo $(W)$. Then the inequality (\ref{D0-first}) with $\val{C_{1}}\geq1$ implies $V\equiv \lambda\lambdat C_{0}$ modulo $(W)$. Hence, we get $$\val{C_{0}}\geq2-\val\lambda$$ from the inequality (\ref{D0-second}).

Finally, we check $\vphi(\mathfrak{X}_{0})\in p^{2}\mfm$. Using the fact $\vphi(\frac{u-p}{p})\equiv\gamma-1$ modulo $pS$, $\vphi(\mathfrak{X}_{0})\equiv\lambda C_{0}\Baseo+p\lambda C_{1}(\gamma-1)\Big(\Baseo-\frac{\lambda(\gamma-1)}{p}\Baseth\Big)+pC_{2}(\gamma-1) \Big(p\lambda\Baset+\lambda^{3}(\gamma-1)\Baseth\Big)$ modulo $p^{2}\mfm$. Since $\val{C_{0}}\geq2-\val\lambda$, $\val{C_{1}}\geq1$, and $\val{C_{2}}\geq-\val\lambda$, $\vphi(\mathfrak{X}_{0})\equiv-\lambda^{2}(C_{1}-p\lambda C_{2})(\gamma-1)^{2}\Baseth$ modulo $p^{2}\mfm$. Then $\vphi(\mathfrak{X}_{0})\equiv0$ modulo $p^{2}\mfm$ by the inequality (\ref{D0-first}). Thus $\vphi(\filt \mfm)\subset p^{2}\mfm$.
\end{proof}

\subsection{Strongly divisible modules: the second case}\label{subsec-D0second}
In this subsection, we construct strongly divisible modules in $\mfd_{[0,\frac{1}{2}]}$ under the assumption $\mathbf{H(0,2)}$ as at the beginning of this section.

We first define two sequences $G_{m}$ and $H_{m}$ for $m\geq0$ recursively as follows: $G_{0}=H_{0}=1$ and
$$G_{m+1}=(\mflo-1)[(\mflo-1)H_{m}-\frac{\lambdat}{\lambda}G_{m}]^{2};$$ $$H_{m+1}=G_{m+1}-\frac{1}{\lambda}\big([(\mflt+p\lambda)-\lambdat(\mflo-1)][(\mflo-1)H_{m}-\frac{\lambdat}{\lambda}G_{m}]+p\lambdat G_{m}\big)G_{m}.$$
We prove that the sequence $G_{m}/H_{m}$ converses in $1+\maxi{E}$. The limit appears in the coefficients of the strongly divisible modules in Proposition \ref{prop-D0second}.

\begin{lemm}
Keep the assumption $\mathbf{H(0,2)}$. Then, for $m\geq0$,
\begin{enumerate}
\item $\val{G_{m}-H_{m}}-\val{G_{m}}\geq\min\big\{\Val{\frac{1}{\lambda}[(\mflt+p\lambda)-\lambdat(\mflo-1)]}-2\val{\mflo-1}, 3[\val{\frac{p}{\lambda}}-\val{\mflo-1}]\big\}>0$;
\item $\val{(\mflo-1)H_{m}-\frac{\lambdat}{\lambda}G_{m}}=\val{\mflo-1}+\val{H_{m}}$;
\item $\val{G_{m}H_{m+1}-H_{m}G_{m+1}}-\val{H_{m}H_{m+1}}\geq(m+1)\min \big\{\Val{\frac{1}{\lambda}[(\mflt+p\lambda)-\lambdat(\mflo-1)]}-2\val{\mflo-1},3[\val{\frac{p}{\lambda}}-\val{\mflo-1}]\big\}$.
\end{enumerate}
\end{lemm}

\begin{proof}
(2) is immediate from (1) since (1) implies $\val{G_{m}}=\val{H_{m}}$ for all $m\geq0$ and since $\val{\mflo-1}<\val{\frac{p}{\lambda}}\leq\val{\frac{\lambdat}{\lambda}}$.

We prove (1) by induction. For $m=0$, it is trivial. Assume that (1) is true for $m$. Then $\val{G_{m}}=\val{H_{m}}$ and $\val{(\mflo-1)H_{m}-\frac{\lambdat}{\lambda}G_{m}}=\val{(\mflo-1)H_{m}}$. Hence, we have $\val{G_{m+1}}=3\val{\mflo-1}+2\val{H_{m}}$ and
$$
\begin{aligned}
\val{&G_{m+1}-H_{m+1}}-\val{G_{m+1}}\\
=&\Val{\frac{1}{\lambda}\big([(\mflt+p\lambda)-\lambdat(\mflo-1)][(\mflo-1)H_{m}-\frac{\lambdat}{\lambda}G_{m}]+p\lambdat G_{m}\big)G_{m}}-\val{G_{m+1}}\\
\geq&\min\big\{\Val{\frac{1}{\lambda}[(\mflt+p\lambda)-\lambdat(\mflo-1)](\mflo-1)H_{m}G_{m}},\val{\frac{p\lambdat}{\lambda} G_{m}^{2}}\big\}-\val{G_{m+1}}\\
=&\min\big\{\Val{\frac{1}{\lambda}[(\mflt+p\lambda)-\lambdat(\mflo-1)]}-2\val{\mflo-1}, \val{\frac{p\lambdat}{\lambda}}-3\val{\mflo-1}\big\}>0.
\end{aligned}
$$
Hence, (1) holds by induction.

For (3), We induct on $m$ as well. If $m=0$, then $G_{0}H_{1}-H_{0}G_{1}=H_{1}-G_{1}=- \frac{1}{\lambda}\big([(\mflt+p\lambda)-\lambdat(\mflo-1)][(\mflo-1)-\frac{\lambdat}{\lambda}]+ p\lambdat\big)$. So it works for $m=0$.

We claim the following identity: for $m\geq1$,
$$G_{m}H_{m+1}-H_{m}G_{m+1}=(W_{1}+W_{2})(G_{m-1}H_{m}-H_{m-1}G_{m}),$$
where $W_{1}=\frac{1}{\lambda}[(\mflt+p\lambda)-\lambdat(\mflo-1)](\mflo-1)^{2}[(\mflo-1)H_{m-1}-\frac{\lambdat}{\lambda}G_{m-1}]
 [(\mflo-1)H_{m}-\frac{\lambdat}{\lambda}G_{m}]$ and
$W_{2}=\frac{p\lambdat}{\lambda}(\mflo-1)^{2}\big([(\mflo-1)H_{m-1}-\frac{\lambdat}{\lambda}G_{m-1}]G_{m}+ G_{m-1}[(\mflo-1)H_{m}-\frac{\lambdat}{\lambda}G_{m}]\big)$.

Indeed, $G_{m}H_{m+1}-H_{m}G_{m+1}=G_{m}(H_{m+1}-G_{m+1})+(G_{m}-H_{m})G_{m+1}= - \frac{1}{\lambda}\big([(\mflt+p\lambda)-\lambdat(\mflo-1)][(\mflo-1)H_{m}-\frac{\lambdat}{\lambda}G_{m}]+p\lambdat G_{m}\big)G_{m}^{2}+ \frac{1}{\lambda}\big([(\mflt+p\lambda)-\lambdat(\mflo-1)][(\mflo-1)H_{m-1}-\frac{\lambdat}{\lambda}G_{m-1}]+p\lambdat G_{m-1}\big)G_{m-1}G_{m+1}=  \frac{1}{\lambda}[(\mflt+p\lambda)-\lambdat(\mflo-1)]\big([(\mflo-1)H_{m-1}-\frac{\lambdat}{\lambda}G_{m-1}]G_{m-1}G_{m+1}- [(\mflo-1)H_{m}-\frac{\lambdat}{\lambda}G_{m}]G_{m}^{2}\big)+ \frac{p\lambdat}{\lambda}(G_{m-1}^{2}G_{m+1}-G_{m}^{3})$. However, $[(\mflo-1)H_{m-1}-\frac{\lambdat}{\lambda}G_{m-1}]G_{m-1}G_{m+1}- [(\mflo-1)H_{m}-\frac{\lambdat}{\lambda}G_{m}]G_{m}^{2}= (\mflo-1)[(\mflo-1)H_{m-1}-\frac{\lambdat}{\lambda}G_{m-1}][(\mflo-1)H_{m}-\frac{\lambdat}{\lambda}G_{m}] \big(G_{m-1}[(\mflo-1)H_{m}-\frac{\lambdat}{\lambda}G_{m}]-G_{m}[(\mflo-1)H_{m-1}-\frac{\lambdat}{\lambda}G_{m-1}]\big)= (\mflo-1)^{2}[(\mflo-1)H_{m-1}-\frac{\lambdat}{\lambda}G_{m-1}][(\mflo-1)H_{m}-\frac{\lambdat}{\lambda}G_{m}] (G_{m-1}H_{m}-H_{m-1}G_{m})$ and $G_{m-1}^{2}G_{m+1}-G_{m}^{3}=(\mflo-1)\big(G_{m-1}^{2}[(\mflo-1)H_{m}-\frac{\lambdat}{\lambda}G_{m}]^{2}- G_{m}^{2}[(\mflo-1)H_{m-1}-\frac{\lambdat}{\lambda}G_{m-1}]^{2}\big)= (\mflo-1)\big(G_{m-1}[(\mflo-1)H_{m}-\frac{\lambdat}{\lambda}G_{m}]+ [(\mflo-1)H_{m-1}-\frac{\lambdat}{\lambda}G_{m-1}]G_{m}\big) \big(G_{m-1}[(\mflo-1)H_{m}-\frac{\lambdat}{\lambda}G_{m}]- [(\mflo-1)H_{m-1}-\frac{\lambdat}{\lambda}G_{m-1}]G_{m}\big)=(\mflo-1)^{2}\big(G_{m-1}[(\mflo-1)H_{m}- \frac{\lambdat}{\lambda}G_{m}]+ [(\mflo-1)H_{m-1}-\frac{\lambdat}{\lambda}G_{m-1}]G_{m}\big)(G_{m-1}H_{m}-H_{m-1}G_{m})$. Hence, we proved the identity.

From the identity, we have  $\val{G_{m}H_{m+1}-H_{m}G_{m+1}}\geq\min\{\val{W_{1}},\val{W_{2}}\}+\val{G_{m-1}H_{m}-H_{m-1}G_{m}}$. By parts (1) and (2), $\val{W_{1}}=\Val{\frac{1}{\lambda}[(\mflt+p\lambda)-\lambdat(\mflo-1)]}+4\val{\mflo-1}+\val{H_{m-1}H_{m}}$ and $\val{W_{2}}=\val{\frac{p\lambdat}{\lambda}}+3\val{\mflo-1}+\val{H_{m-1}H_{m}}$. Thus, we have $\val{G_{m}H_{m+1}-H_{m}G_{m+1}}\geq \min\big\{\Val{\frac{1}{\lambda}[(\mflt+p\lambda)-\lambdat(\mflo-1)]}-2\val{\mflo-1}, 3[\val{\frac{p}{\lambda}}-\val{\mflo-1}]\big\}+ 6\val{\mflo-1}+\val{H_{m-1}H_{m}}+ m\min \big\{\Val{\frac{1}{\lambda}[(\mflt+p\lambda)-\lambdat(\mflo-1)]}-2\val{\mflo-1},3[\val{\frac{p}{\lambda}}-\val{\mflo-1}]\big\} +\val{H_{m-1}H_{m}}$ by induction hypothesis. Hence, it holds by induction.
\end{proof}

The assumption $\mathbf{H(0,2)}$ implies that the quantities in the set of the part (3) of the lemma above are strictly positive. So the part (3) of the lemma says that $\val{G_{m}H_{m+1}-H_{m}G_{m+1}}-\val{H_{m}H_{m+1}}$ approaches $\infty$ as $m$ goes to $\infty$. That is, the sequence $G_{m}/H_{m}$ is Cauchy. We let $$\Delta_{\tiny{\ref{subsec-D0second}}}:=\lim_{m\mapsto\infty}\frac{G_{m}}{H_{m}}$$ Note that $\Delta_{\tiny{\ref{subsec-D0second}}}$ depends on the values of the parameters $\lambda,\lambdat,\mflo,\mflt$.

The part (1) of the lemma implies that $\val{G_{m}}=\val{H_{m}}$ for all $m\geq0$ and
\begin{multline*}
\val{1-\Delta_{\tiny{\ref{subsec-D0second}}}}\geq\min\big\{3[\val{\frac{p}{\lambda}}-\val{\mflo-1}],\\ \Val{\frac{1}{\lambda}[(\mflt+p\lambda)-\lambdat(\mflo-1)]}-2\val{\mflo-1}\big\}>0,
\end{multline*}
which immediately implies that $$\frac{\lambda(\mflo-1)-\lambdat\Delta_{\tiny{\ref{subsec-D0second}}}}{\lambda(\mflo-1)}\in 1+\maxi{E}.$$ In particular, $\Val{\lambda(\mflo-1)-\lambdat\Delta_{\tiny{\ref{subsec-D0second}}}}=\Val{\lambda(\mflo-1)}$.

It is also easy to check that $\Delta_{\tiny{\ref{subsec-D0second}}}$ satisfies the equation
\begin{multline} \label{D0-third}
(\mflo-1)[\lambda(\mflo-1)-\lambdat\Delta_{\tiny{\ref{subsec-D0second}}}]^{2}(1-\Delta_{\tiny{\ref{subsec-D0second}}})+\\
\big([(\mflt+p\lambda)-\lambdat(\mflo-1)][\lambda(\mflo-1)-\lambdat\Delta_{\tiny{\ref{subsec-D0second}}}] +p\lambda\lambdat\Delta_{\tiny{\ref{subsec-D0second}}}\big)\Delta_{\tiny{\ref{subsec-D0second}}}^{2}=0,
\end{multline}
by taking the limits of $G_{m+1}/H_{m+1}=(\mflo-1)[(\mflo-1)H_{m}-\frac{\lambdat}{\lambda}G_{m}]^{2}/ \big[(\mflo-1)[(\mflo-1)H_{m}-\frac{\lambdat}{\lambda}G_{m}]^{2}- \frac{1}{\lambda}\big([(\mflt+p\lambda)-\lambdat(\mflo-1)][(\mflo-1)H_{m}-\frac{\lambdat}{\lambda}G_{m}]+p\lambdat G_{m}\big)G_{m}\big]$. The equation plays a crucial role in the proof of the following proposition.

\begin{prop}\label{prop-D0second}
Keep the assumption $\mathbf{H(0,2)}$. Then $\mfm_{[0,\frac{1}{2}]}:=S_{\roi}(E_{1},E_{2},E_{3})$ is a strongly divisible module in $\mfd_{[0,\frac{1}{2}]}$, where
    $$\left\{
        \begin{array}{ll}
          E_{1}=\mathfrak{U}_{0}+\frac{p \Delta_{\tiny{\ref{subsec-D0second}}}(\gamma-1)}{\lambda[\lambda(\mflo-1)-\lambdat\Delta_{\tiny{\ref{subsec-D0second}}}]} \left(\frac{\lambda(\mflo-1)-\lambdat\Delta_{\tiny{\ref{subsec-D0second}}}}{p\lambda(\mflo-1)} (\lambdat\baset+\lambda\mflt\baseth)-\lambda^{2}(\gamma-1)\baseth\right)& \hbox{} \\
          E_{2}=\frac{\lambda(\mflo-1)-\lambdat\Delta_{\tiny{\ref{subsec-D0second}}}}{p\lambda(\mflo-1)} (\lambdat\baset+\lambda\mflt\baseth)-\lambda^{2}(\gamma-1)\baseth & \hbox{} \\
          E_{3}=[\lambda(\mflo-1)-\lambdat\Delta_{\tiny{\ref{subsec-D0second}}}]\baseth. & \hbox{}
        \end{array}
      \right.$$
\end{prop}
Note that $S_{\roi}(\Baseo,\Baset,\Baseth)=S_{\roi}(\Baseo,\Baset,\lambda(\mflo-1)\baseth)$.

\begin{proof}
During the proof, we let $\Delta:=\Delta_{\tiny{\ref{subsec-D0second}}}$ and $\mfm:=\mfm_{[0,\frac{1}{2}]}$ for brevity. It is routine to check that $\vphi(\Baseo)\equiv\Baseth$ and $\vphi(\Baset)\equiv\vphi(\Baseth)\equiv N(\Baseo)\equiv N(\Baset)\equiv N(\Baseth)\equiv 0$ modulo $\maxi{E}\mfm$. Hence, $\mfm$ is stable under $\vphi$ and $N$.

We check that $\vphi(\filt \mfm)\subset p^{2}\mfm$. It is easy to check that $\filt S\cdot\mfm\subset\filt\mfm$, and so it is enough to check $\vphi(\mathfrak{X}_{0})\in p^{2}\mfm$, since $\vphi(\filt S\cdot\mfm)\subset p^{2}\mfm$. We first compute  $\filt \mfm$ modulo $\filt S\cdot\mfm$, by finding out the conditions that $\mathfrak{X}_{0}$ be in $\mfm$. $\mathfrak{X}_{0}\equiv C_{0}\Big{(}\frac{1}{p}E_{1}-\frac{1}{\lambda\lambdat}\Baset +\frac{(\mflt+p\lambda)[\lambda(\mflo-1)-\lambdat\Delta]+ p\lambda\lambdat\Delta}{p\lambdat[\lambda(\mflo-1)-\lambdat\Delta]^{2}}\Baseth- \frac{(\mflo-1)-p\mflo}{p[\lambda(\mflo-1)-\lambdat\Delta]}\Baseth\Big{)}+ (u-p)\Big{(}\frac{C_{1}}{p}E_{1}- \frac{[\lambda(\mflo-1)-\lambdat\Delta]C_{1}- p\lambda^{2}(\mflo-1)C_{2}}{\lambda\lambdat[\lambda(\mflo-1)-\lambdat\Delta]}E_{2}+ \frac{V}{W}E_{3}\Big{)}$ modulo $\filt S\cdot\mfm$, where $W=p\lambdat[\lambda(\mflo-1)-\lambdat\Delta]^{2}$ and
\begin{multline*}
V=\lambdat[\lambda(\mflo-1)-\lambdat\Delta](C_{0}+p\mflo C_{1}+p\mflt C_{2})+ p\lambda^{2}(\mflo-1)C_{1}+\\ \mflt[\lambda(\mflo-1)-\lambdat\Delta] C_{1}-\lambdat(\mflo-1)[\lambda(\mflo-1)-\lambdat\Delta]C_{1}-\\ p^{2}\lambda^{3}(\mflo-1)C_{2}- p\lambda\mflt[\lambda(\mflo-1)-\lambdat\Delta]C_{2}.
\end{multline*}
Hence, if $\mathfrak{X}_{0}\in\filt \mfm$, then we get $\val{C_{0}}\geq\val{\lambda\lambdat}=2-\val\lambda$, $\val{C_{1}}\geq1$,
\begin{equation}\label{D0-fourth}
\Val{[\lambda(\mflo-1)-\lambdat\Delta]C_{1}- p\lambda^{2}(\mflo-1)C_{2}}\geq\Val{\lambda\lambdat[\lambda(\mflo-1)-\lambdat\Delta]},
\end{equation}
and
\begin{equation}\label{D0-fifth}
\val{V}\geq \val{W}.
\end{equation}
The inequality (\ref{D0-fourth}) with $\val{C_{1}}\geq1$ implies $$\val{C_{2}}\geq-\val\lambda.$$ Using the inequalities $\val{C_{1}}\geq1$ and $\val{C_{2}}\geq-\val\lambda$, one can readily check that
\begin{multline*}
V\equiv\lambdat[\lambda(\mflo-1)-\lambdat\Delta]C_{0}+ (\mflt+p\lambda)[\lambda(\mflo-1)-\lambdat\Delta]C_{1}+\\ p\lambda\lambdat\Delta C_{1} -\lambdat(\mflo-1)[\lambda(\mflo-1)-\lambdat\Delta]C_{1}- \\ p\lambda(\mflt+p\lambda)[\lambda(\mflo-1)-\lambdat\Delta]C_{2}-p^{2}\lambda^{2}\lambdat\Delta C_{2}
\end{multline*}
 modulo $(W)$. Thus
\begin{equation*}
\Delta[\lambda(\mflo-1)-\lambdat\Delta]\cdot V\equiv\lambdat [\lambda(\mflo-1)-\lambdat\Delta]^{2} \big(\Delta C_{0}-p\lambda(\mflo-1)C_{2}\big)+X
\end{equation*}
modulo $([\lambda(\mflo-1)-\lambdat\Delta]W)$, where
\begin{multline*}
X= p\lambda\lambdat(\mflo-1)[\lambda(\mflo-1)-\lambdat\Delta]^{2}C_{2}+ (\mflt+p\lambda)[\lambda(\mflo-1)-\lambdat\Delta]^{2}\Delta C_{1}-\\ p\lambda(\mflt+p\lambda)[\lambda(\mflo-1)-\lambdat\Delta]^{2}\Delta C_{2}+ p\lambda\lambdat[\lambda(\mflo-1)-\lambdat\Delta]\Delta^{2}C_{1}- \\ p^{2}\lambda^{2}\lambdat[\lambda(\mflo-1)-\lambdat\Delta]\Delta^{2}C_{2}- \lambdat(\mflo-1)[\lambda(\mflo-1)-\lambdat\Delta]^{2}\Delta C_{1}.
\end{multline*}
By the inequality (\ref{D0-fourth}),
\begin{multline*}
X\equiv p\lambda\lambdat(\mflo-1)[\lambda(\mflo-1)-\lambdat\Delta]^{2}C_{2}+ p\lambda^{2}(\mflt+p\lambda)(\mflo-1)[\lambda(\mflo-1)-\lambdat\Delta]\Delta C_{2}-\\ p\lambda(\mflt+p\lambda)[\lambda(\mflo-1)-\lambdat\Delta]^{2}\Delta C_{2}+ p^{2}\lambda^{3}\lambdat(\mflo-1)\Delta^{2}C_{2}- \\ p^{2}\lambda^{2}\lambdat[\lambda(\mflo-1)-\lambdat\Delta]\Delta^{2}C_{2}- p\lambda^{2}\lambdat(\mflo-1)^{2}[\lambda(\mflo-1)-\lambdat\Delta]\Delta C_{2}
\end{multline*}
modulo $([\lambda(\mflo-1)-\lambdat\Delta]W)$, which rearranges to
\begin{multline*}
p\lambda\lambdat\big[(\mflo-1)[\lambda(\mflo-1)-\lambdat\Delta]^{2}- \lambda(\mflo-1)^{2}[\lambda(\mflo-1)-\lambdat\Delta]\Delta+\\
\big((\mflt+p\lambda)[\lambda(\mflo-1)-\lambdat\Delta]+p\lambda\lambdat\Delta\big)\Delta^{2}\big] C_{2}.
\end{multline*}
But this is $0$ by the equation (\ref{D0-third}). Hence, from the inequality (\ref{D0-fifth}), we have
\begin{equation}\label{D0-sixth}
\Val{\Delta C_{0}-p\lambda(\mflo-1)C_{2}}\geq\Val{p[\lambda(\mflo-1)-\lambdat\Delta]}.
\end{equation}

Finally, we check $\vphi(\mathfrak{X}_{0})\in p^{2}\mfm$. It is easy to check that $\vphi(\mathfrak{X}_{0})\equiv\lambda C_{0}\Big(\Baseo-\frac{p\Delta(\gamma-1)}{\lambda[\lambda(\mflo-1)-\lambdat\Delta]}\Baset\Big)+ p\lambda C_{1}(\gamma-1)\Big(\Baseo-\frac{\gamma-1}{[\lambda(\mflo-1)-\lambdat\Delta]}\Baseth- \frac{p\Delta(\gamma-1)}{\lambda[\lambda(\mflo-1)-\lambdat\Delta]}\Baset\Big)+ pC_{2}(\gamma-1)\Big(\frac{p\lambda(\mflo-1)}{[\lambda(\mflo-1)-\lambdat\Delta]}\Baset+ \frac{p\lambda^{3}(\mflo-1)(\gamma-1)}{[\lambda(\mflo-1)-\lambdat\Delta]^{2}}\Baseth\Big)$ modulo $p^{2}\mfm$, using the fact $\vphi(\frac{u-p}{p})\equiv\gamma-1$ modulo $pS$. Since $\val{C_{0}}\geq\val{\lambda\lambdat}$ and $\val{C_{1}}\geq1$,
$\vphi(\mathfrak{X}_{0})\equiv-p\frac{\Delta C_{0}-p\lambda(\mflo-1)C_{2}}{\lambda(\mflo-1)-\lambdat\Delta} (\gamma-1)\Baset - p\lambda\frac{[\lambda(\mflo-1)-\lambdat\Delta]C_{1}-p\lambda^{2} (\mflo-1)C_{2}}{[\lambda(\mflo-1)-\lambdat\Delta]^{2}}(\gamma-1)^{2}\Baseth$ modulo $p^{2}\mfm$. Then $\vphi(\mathfrak{X}_{0})\equiv0$ modulo $p^{2}\mfm$ by the inequalities (\ref{D0-fourth}) and (\ref{D0-sixth}). Thus $\vphi(\filt \mfm)\subset p^{2}\mfm$.
\end{proof}

\subsection{Strongly divisible modules: the third case}\label{subsec-D0third}
In this subsection, we construct strongly divisible modules in $\mfd_{[0,\frac{1}{2}]}$ under the assumption $\mathbf{H(0,3)}$ as at the beginning of this section.

We first define two sequences $G_{m}$ and $H_{m}$ for $m\geq0$ recursively as follows: $G_{0}=H_{0}=1$ and
$$G_{m+1}=(\mflt H_{m}+p\lambda G_{m})^{2};$$ $$H_{m+1}=G_{m+1}-\lambdat[(\mflo-1)(\mflt H_{m}+p\lambda G_{m})+p\lambdat G_{m}]H_{m}.$$
We prove that the sequence $G_{m}/H_{m}$ converses in $1+\maxi{E}$. The limit appears in the coefficients of the strongly divisible modules in Proposition \ref{prop-D0third}.

\begin{lemm}
Keep the assumption $\mathbf{H(0,3)}$. Then, for $m\geq0$,
\begin{enumerate}
\item $\val{G_{m}-H_{m}}-\val{G_{m}}\geq\min\big\{\Val{\lambdat(\mflo-1)}-\val{\mflt+p\lambda}, \val{p\lambdat^{2}}-2\val{\mflt+p\lambda}\big\}>0$;
\item $\val{\mflt H_{m}+p\lambda G_{m}}=\val{\mflt+p\lambda}+\val{H_{m}}$;
\item $\val{G_{m}H_{m+1}-H_{m}G_{m+1}}-\val{H_{m}H_{m+1}}\geq \min\big\{\Val{\lambdat(\mflo-1)}-\val{\mflt+p\lambda},\val{p\lambdat^{2}}-2\val{\mflt+p\lambda}\big\}+ m\min\big\{\Val{p\lambda\lambdat(\mflo-1)}-2\val{\mflt+p\lambda},3[\val{\lambda\lambdat}-\val{\mflt+p\lambda}], \val{p\lambdat^{2}}-2\val{\mflt+p\lambda}\big\}$.
\end{enumerate}
\end{lemm}

\begin{proof}
(2) is immediate from (1) since (1) implies $\val{G_{m}}=\val{H_{m}}$ for all $m\geq0$ and $\val{\mflt H_{m}+p\lambda G_{m}}=\Val{(\mflt+p\lambda)H_{m}+p\lambda(G_{m}-H_{m})}=\val{\mflt+p\lambda}+\val{H_{m}}$.

We prove (1) by induction. For $m=0$, it is trivial. Assume that (1) is true for $m$. Then $\val{G_{m}}=\val{H_{m}}$ and
$\val{\mflt H_{m}+p\lambda G_{m}}=\val{\mflt+p\lambda}+\val{H_{m}}$. Hence, we have $\val{G_{m+1}}=2\val{\mflt+p\lambda}+2\val{H_{m}}$ and
$$
\begin{aligned}
\val{&G_{m+1}-H_{m+1}}- \val{G_{m+1}}\\
&=\Val{\lambdat[(\mflo-1)(\mflt H_{m}+p\lambda G_{m})+p\lambdat G_{m}]H_{m}}- \val{G_{m+1}}\\
&\geq\min\big\{\Val{\lambdat(\mflo-1)(\mflt+p\lambda)H_{m}^{2}},\Val{p\lambdat^{2}G_{m}H_{m}}\big\}- \val{G_{m+1}}\\
&=\min\big\{\Val{\lambdat(\mflo-1)}-\val{\mflt+p\lambda},\val{p\lambdat^{2}}-2\val{\mflt+p\lambda}\big\}>0.
\end{aligned}
$$
Hence, (1) holds by induction.

For (3), we induct on $m$ as well. For $m=0$, $G_{0}H_{1}-H_{0}G_{1}=H_{1}-G_{1}= -\lambdat[(\mflo-1)(\mflt+p\lambda)+p\lambdat]$. Hence, it holds for $m=0$.

We claim the following identity: for $m\geq1$, $$G_{m}H_{m+1}-H_{m}G_{m+1}=-(W_{1}+W_{2}-W_{3})(G_{m-1}H_{m}-H_{m-1}G_{m}),$$
where $W_{1}=p\lambda\lambdat(\mflo-1)(\mflt H_{m-1}+p\lambda G_{m-1})(\mflt H_{m}+p\lambda G_{m})$, $W_{2}=p^{2}\lambda\lambdat^{2}[(\mflt H_{m}+p\lambda G_{m})G_{m-1}+ (\mflt H_{m-1}+p\lambda G_{m-1})H_{m}]$, and
$W_{3}=p\lambdat^{2}(\mflt+p\lambda)(\mflt H_{m-1}+p\lambda G_{m-1})H_{m}$.

Indeed, $G_{m}H_{m+1}-H_{m}G_{m+1}=G_{m}(H_{m+1}-G_{m+1})+(G_{m}-H_{m})G_{m+1}= -\lambdat [(\mflo-1)(\mflt H_{m}+p\lambda G_{m})+p\lambdat G_{m}]G_{m}H_{m}+ \lambdat[(\mflo-1)(\mflt H_{m-1}+p\lambda G_{m-1})+p\lambdat G_{m-1}]H_{m-1}G_{m+1}= \lambdat(\mflo-1)[H_{m-1}G_{m+1}(\mflt H_{m-1}+p\lambda G_{m-1})-(\mflt H_{m}+p\lambda G_{m})G_{m}H_{m}]+p\lambdat^{2}(G_{m-1}H_{m-1}G_{m+1}-G_{m}^{2}H_{m})$. But $H_{m-1}G_{m+1}(\mflt H_{m-1}+p\lambda G_{m-1})- (\mflt H_{m}+p\lambda G_{m})G_{m}H_{m}= (\mflt H_{m-1}+p\lambda G_{m-1})(\mflt H_{m}+p\lambda G_{m})[(\mflt H_{m}+p\lambda G_{m})H_{m-1}-(\mflt H_{m-1}+p\lambda G_{m-1})H_{m}]= p\lambda(H_{m-1}G_{m}-G_{m-1}H_{m})(\mflt H_{m-1}+p\lambda G_{m-1})(\mflt H_{m}+p\lambda G_{m})$ and $G_{m-1}H_{m-1}G_{m+1}-G_{m}^{2}H_{m}= G_{m-1}H_{m-1}(\mflt H_{m}+p\lambda G_{m})^{2}-(\mflt H_{m-1}+p\lambda G_{m-1})^{2}G_{m}H_{m}= G_{m-1}(\mflt H_{m}+p\lambda G_{m})[H_{m-1}(\mflt H_{m}+p\lambda G_{m})- (\mflt H_{m-1}+p\lambda G_{m-1})H_{m}]+(\mflt H_{m-1}+p\lambda G_{m-1})H_{m}[G_{m-1}(\mflt H_{m}+p\lambda G_{m})- (\mflt H_{m-1}+p\lambda G_{m-1})G_{m}]= \mflt (\mflt H_{m-1}+p\lambda G_{m-1})H_{m}(G_{m-1}H_{m}-H_{m-1}G_{m})+p\lambda G_{m-1}(\mflt H_{m}+p\lambda G_{m})(H_{m-1}G_{m}-G_{m-1}H_{m})$. Hence, we proved the identity.

The identity implies $\val{G_{m}H_{m+1}-H_{m}G_{m+1}}\geq\min\{\val{W_{1}},\val{W_{2}},\val{W_{3}}\}+\val{G_{m-1}H_{m}-H_{m-1}G_{m}}$. By parts (1) and (2), $\val{W_{1}}=\val{p\lambda\lambdat}+\val{\mflo-1}+2\val{\mflt+p\lambda}+\val{H_{m-1}H_{m}}$, $\val{W_{2}}=\val{p^{2}\lambda\lambdat^{2}}+\val{\mflt+p\lambda}+\val{H_{m-1}H_{m}}$, and $\val{W_{3}}=\val{p\lambdat^{2}}+2\val{\mflt+p\lambda}+\val{H_{m-1}H_{m}}$. Thus, $\val{G_{m}H_{m+1}-H_{m}G_{m+1}}\geq \min\big\{\Val{p\lambda\lambdat(\mflo-1)}-2\val{\mflt+p\lambda},3[\val{\lambda\lambdat}-\val{\mflt+p\lambda}], \val{p\lambdat^{2}}-2\val{\mflt+p\lambda}\big\}+4\val{\mflt+p\lambda}+\val{H_{m-1}H_{m}} +\min\big\{\Val{\lambdat(\mflo-1)}-\val{\mflt+p\lambda},\val{p\lambdat^{2}}-2\val{\mflt+p\lambda}\big\}+ (m-1)\min\big\{\Val{p\lambda\lambdat(\mflo-1)}-2\val{\mflt+p\lambda},3[\val{\lambda\lambdat}-\val{\mflt+p\lambda}], \val{p\lambdat^{2}}-2\val{\mflt+p\lambda}\big\}+\val{H_{m-1}H_{m}}$ by induction hypothesis. Hence, it holds by induction.
\end{proof}

The assumption $\mathbf{H(0,3)}$ implies that the quantities in the sets of the part (3) of the lemma above are strictly positive. So the part (3) of the lemma says that $\val{G_{m}H_{m+1}-H_{m}G_{m+1}}-\val{H_{m}H_{m+1}}$ approaches $\infty$ as $m$ goes to $\infty$. That is, the sequence $G_{m}/H_{m}$ is Cauchy. We let $$\Delta_{\tiny{\ref{subsec-D0third}}}:=\lim_{m\mapsto\infty}\frac{G_{m}}{H_{m}}.$$
Note that $\Delta_{\tiny{\ref{subsec-D0third}}}$ depends on the values of the parameters $\lambda,\lambdat,\mflo,\mflt$.

The part (1) of the lemma implies that $\val{G_{m}}=\val{H_{m}}$ for all $m\geq0$ and
\begin{equation*}
\val{1-\Delta_{\tiny{\ref{subsec-D0third}}}}\geq
\min\big\{\Val{\lambdat(\mflo-1)}-\val{\mflt+p\lambda}, \val{p\lambdat^{2}}-2\val{\mflt +p\lambda}\big\}>0,
\end{equation*}
which immediately implies $$\frac{\mflt+p\lambda\Delta_{\tiny{\ref{subsec-D0third}}}}{\mflt+p\lambda}\in1+\maxid.$$
In particular, $\val{\mflt+p\lambda\Delta_{\tiny{\ref{subsec-D0third}}}}=\val{\mflt+p\lambda}$.

It is also easy to check that $\Delta_{\tiny{\ref{subsec-D0third}}}$ satisfies the equation
\begin{multline}\label{D0-seventh}
(\mflt+ p\lambda \Delta_{\tiny{\ref{subsec-D0third}}})^{2}(1-\Delta_{\tiny{\ref{subsec-D0third}}})+
\lambdat\big((\mflo-1)(\mflt +p\lambda \Delta_{\tiny{\ref{subsec-D0third}}})+p\lambdat\Delta_{\tiny{\ref{subsec-D0third}}}\big) \Delta_{\tiny{\ref{subsec-D0third}}}=0,
\end{multline}
by taking the limits of $G_{m+1}/H_{m+1}=(\mflt H_{m}+p\lambda G_{m})^{2}/\big((\mflt H_{m}+p\lambda G_{m})^{2}-\lambdat[(\mflo-1)(\mflt H_{m}+p\lambda G_{m})+p\lambdat G_{m}]H_{m}\big)$. The equation plays a crucial role in the proof of the following proposition.

\begin{prop}\label{prop-D0third}
Keep the assumption $\mathbf{H(0,3)}$. Then $\mfm_{[0,\frac{1}{2}]}:=S_{\roi}(E_{1},E_{2},E_{3})$ is a strongly divisible module in $\mfd_{[0,\frac{1}{2}]}$, where
    $$\left\{
        \begin{array}{ll}
          E_{1}=\mathfrak{U}_{0}+\frac{\lambdat(\gamma-1)}{\lambda^{2}(\mflt +p\lambda \Delta_{\tiny{\ref{subsec-D0third}}})}(\lambdat\baset+\lambda\mflt\baseth)  & \hbox{} \\
          E_{2}=\frac{1}{p\lambda} (\lambdat\baset+\lambda\mflt\baseth)-\lambda \Delta_{\tiny{\ref{subsec-D0third}}}(\gamma-1)\baseth & \hbox{} \\
          E_{3}=\frac{\lambda(\mflt+p\lambda\Delta_{\tiny{\ref{subsec-D0third}}})}{p}\baseth. & \hbox{}
        \end{array}
      \right.$$
\end{prop}
Note that $S_{\roi}(\Baseo,\Baset,\Baseth)=S_{\roi}(\Baseo,\Baset,\frac{\lambda(\mflt+p\lambda)}{p}\baseth)$.

\begin{proof}
During the proof, we let $\Delta:=\Delta_{\tiny{\ref{subsec-D0third}}}$ and $\mfm:=\mfm_{[0,\frac{1}{2}]}$ for brevity. It is routine to check that $\vphi(\Baseo)\equiv\frac{p(\mflo-1)}{\mflt+p\lambda}\Baseth$, $\vphi(\Baset)\equiv\Baseth$, and $\vphi(\Baseth)\equiv N(\Baseo)\equiv N(\Baset)\equiv N(\Baseth)\equiv 0$ modulo $\maxi{E}\mfm$. Hence, $\mfm$ is stable under $\vphi$ and $N$.

We check that $\vphi(\filt \mfm)\subset p^{2}\mfm$. It is easy to check that $\filt S\cdot\mfm\subset\filt\mfm$, and so it is enough to check $\vphi(\mathfrak{X}_{0})\in p^{2}\mfm$, since $\vphi(\filt S\cdot\mfm)\subset p^{2}\mfm$. We first compute  $\filt \mfm$ modulo $\filt S\cdot\mfm$, by finding out the conditions that $\mathfrak{X}_{0}$ be in $\mfm$. $\mathfrak{X}_{0}\equiv C_{0}\Big{(}\frac{1}{p}E_{1}-\frac{1}{\lambdat}\Baset+\frac{\lambdat}{\lambda(\mflt+p\lambda \Delta)}\Baset+ \frac{1}{\lambda\lambdat }\Baseth-\frac{\mflo-1}{\lambda(\mflt+p\lambda \Delta)}\Baseth- \frac{p\lambdat \Delta}{\lambda(\mflt+p\lambda \Delta)^{2}}\Baseth+ \frac{p\mflo}{\lambda(\mflt+p\lambda \Delta)}\Baseth\Big{)}+ (u-p)\Big(\frac{C_{1}}{p}E_{1}- \frac{\lambda(\mflt+p\lambda)(C_{1}-p\lambda C_{2})-\lambdat^{2}C_{1}}{\lambda\lambdat(\mflt+p\lambda\Delta)}E_{2}+ \frac{V}{W}E_{3}\Big)$ modulo $\filt S\cdot\mfm$, where $W=\lambda\lambdat(\mflt+p\lambda \Delta)^{2}$ and
\begin{multline*}
V=\lambdat(\mflt+p\lambda \Delta)(C_{0}+p\mflo C_{1}+p\mflt C_{2})+ (\mflt+p\lambda \Delta)^{2}C_{1}- \\ \lambdat(\mflo-1)(\mflt+p\lambda \Delta)C_{1}- p\lambdat^{2}\Delta C_{1}-p\lambda(\mflt+p\lambda \Delta)^{2}C_{2}.
\end{multline*}
Hence, if $\mathfrak{X}_{0}\in\filt \mfm$, then we get $\val{C_{0}}\geq\val{\lambda\lambdat}=2-\val\lambda$, $\val{C_{1}}\geq1$,
\begin{equation}\label{D0-eightth}
\Val{\lambda(\mflt+p\lambda\Delta)(C_{1}-p\lambda C_{2})-\lambdat^{2}C_{1}}\geq\Val{\lambda\lambdat(\mflt+p\lambda\Delta)},
\end{equation}
and
\begin{equation}\label{D0-nineth}
\val{V}\geq \val{W}.
\end{equation}
The inequality (\ref{D0-eightth}) is equivalent to
\begin{equation}\label{D0-tenth}
\val{C_{1}-p\lambda C_{2}}\geq\val{\lambdat},
\end{equation}
which implies $\val{C_{2}}\geq-\val\lambda$ since $\val{C_{1}}\geq1$. Using the inequalities $\val{C_{1}}\geq1$ and $\val{C_{2}}\geq-\val\lambda$, one can readily check that
\begin{multline*}
V\equiv\lambdat(\mflt+p\lambda \Delta)C_{0}+(\mflt+p\lambda \Delta)^{2}C_{1}-\\ \lambdat(\mflo-1)(\mflt+p\lambda \Delta)C_{1}-
p\lambdat^{2}\Delta C_{1}-p\lambda(\mflt+p\lambda \Delta)^{2}C_{2}
\end{multline*}
 modulo $(W)$. Thus
\begin{equation*}
\Delta\cdot V\equiv(\mflt+p\lambda \Delta)\big(\lambdat\Delta C_{0}+ (\mflt+p\lambda \Delta)(C_{1}-p\lambda\Delta C_{2})\big)+X
\end{equation*}
modulo $(W)$, where $$X=\big[(\mflt+p\lambda \Delta)^{2}(\Delta-1)-\lambdat\big((\mflo-1)(\mflt+p\lambda \Delta)+ p\lambdat\Delta\big)\Delta\big]C_{1}.$$ But, by the equation (\ref{D0-seventh}), $X=0$. Hence, from the inequality (\ref{D0-nineth}), we have
\begin{equation}\label{D0-eleventh}
\Val{\lambdat\Delta C_{0}+(\mflt+p\lambda \Delta)(C_{1}-p\lambda\Delta C_{2})}\\
\geq\Val{\lambda\lambdat(\mflt+p\lambda \Delta)}.
\end{equation}

Finally, we check $\vphi(\mathfrak{X}_{0})\in p^{2}\mfm$. Using the fact $\vphi(\frac{u-p}{p})\equiv\gamma-1$ modulo $pS$, $\vphi(\mathfrak{X}_{0})\equiv\lambda C_{0}\Big(\Baseo-\frac{p\lambdat(\gamma-1)}{\lambda(\mflt+p\lambda\Delta)}\Baset- \frac{p^{2}\lambdat\Delta(\gamma-1)^{2}}{\lambda(\mflt+p\lambda\Delta)^{2}}\Baseth\Big)+p\lambda C_{1}(\gamma-1)\Big(\Baseo-\frac{p(\gamma-1)}{\lambda(\mflt+p\lambda\Delta)}\Baseth -\frac{p\lambdat(\gamma-1)}{\lambda(\mflt+p\lambda\Delta)}\Baset- \frac{p^{2}\lambdat\Delta(\gamma-1)^{2}}{\lambda(\mflt+p\lambda\Delta)^{2}}\Baseth\Big)+ pC_{2}(\gamma-1)\Big(p\lambda\Baset+ \frac{p^{2}\lambda\Delta(\gamma-1)}{\mflt+p\lambda\Delta}\Baseth\Big)$ modulo $p^{2}\mfm$. Then $\vphi(\mathfrak{X}_{0})\equiv- p^{2}\frac{\lambdat\Delta C_{0}+(\mflt+p\lambda \Delta)(C_{1}-p\lambda\Delta C_{2})}{(\mflt+p\lambda\Delta)^{2}}(\gamma-1)^{2}\Baseth$ modulo $p^{2}\mfm$ since $\val{C_{0}}\geq2-\val\lambda$, $\val{C_{1}}\geq1$, and $\val{C_{2}}\geq-\val\lambda$, and so $\vphi(\mathfrak{X}_{0})\equiv0$ modulo $p^{2}\mfm$ by the inequality (\ref{D0-eleventh}).  Thus $\vphi(\filt \mfm)\subset p^{2}\mfm$.
\end{proof}

\section{Galois stable lattices of $D_{[\frac{1}{2},1]}$}\label{sec-strong div mod1}
In this section, we construct strongly divisible modules of the modules $D_{[\frac{1}{2},1]}$ in Example \ref{example D[1/2,1]}. We let $\mfd_{[\frac{1}{2},1]}:=S\otimes_{\ZP}D_{[\frac{1}{2},1]}$. In this subsection, we write $D$ and $\mfd$ for $D_{[\frac{1}{2},1]}$ and $\mfd_{[\frac{1}{2},1]}$ respectively for brevity. Since we are concerned only with absolutely irreducible mod $p$ reductions, we may assume $\frac{1}{2}\leq\val\lambda<1$ due to Proposition \ref{Prop-ad-filt-module[1/2,1]}.

It is easy to check that $$\filo\mfd=S_{E}(\baseo+\mflo\baset+\mflt\baseth,\baset)+\filo S_{E}\cdot\mfd;$$
$$\filt\mfd=S_{E}\big(\baseo+\mflo\baset+\mflt\baseth+\frac{u-p}{p}\baseth\big)+\filo S_{E}(\baset)+\filt S_{E}\cdot\mfd.$$ (We omit their proofs.) So every element in $\filt\mfd$ is of the form $C_{0}(\baseo+\mflo\baset+\mflt\baseth+\frac{u-p}{p}\baseth)+C_{1}(u-p)(\baseo+\mflo\baset+\mflt\baseth) +C_{2}(u-p)\baset+A\baseo+B\baset+C\baseth$, where $C_{0},C_{1},C_{2}$ are in $E$ and $A,B,C$ are in $\filt S_{E}$. We let $$\mathfrak{X}_{1}:=C_{0}\big(\baseo+\mflo\baset+\mflt\baseth+\frac{u-p}{p}\baseth\big)+(u-p)\big(C_{1} (\baseo+\mflo\baset+\mflt\baseth) +C_{2}\baset\big),$$ which rearranges to $C_{0}\big(\baseo+\mflo\baset+\mflt\baseth\big)+(u-p)\big(C_{1} \baseo+(\mflo C_{1}+C_{2})\baset+ \frac{C_{0}+p\mflt C_{1}}{p}\baseth\big)$.

We divide the area in which the parameters of $D$ are defined into $3$ pieces as follows: for $\lambda,\lambdat\in\roi$ with $\frac{1}{2}\leq\val{\lambda}<1$ and $2\val\lambda+\val\lambdat=2$ and for $\mflo,\mflt\in E$,
\begin{description}
\item[H(1,1)] $\val{\mflo-\lambdat}\geq1$ and $\val{1-\mflt}\geq\val\lambda$;
\item[H(1,2)] $\Val{\lambda(\mflo-\lambdat)}\geq\Val{p(1-\mflt)}$ and $\val{1-\mflt}<\val\lambda$;
\item[H(1,3)] $\Val{\lambda(\mflo-\lambdat)}\leq\Val{p(1-\mflt)}$ and $\val{\mflo-\lambdat}<\min\{\Val{\lambda(1-\mflt)},1\}$.
\end{description}
Note that the conditions $\mathbf{H(1,2)}$ and $\mathbf{H(1,3)}$ intersect in $\val{\lambda(\mflo-\lambdat)}=\val{p(1-\mflt)}$ and $\val{1-\mflt}<\val\lambda$ if $\frac{1}{2}<\val\lambda<1$. We also note that the condition $\val{\mflo-\lambdat}<\val{\lambda(1-\mflt)}$ in $\mathbf{H(1,3)}$ matters only when $\val\lambda=\frac{1}{2}$, since $\val{\lambda(\mflo-\lambdat)}\leq\val{p(1-\mflt)}$ implies $\val{\mflo-\lambdat}<\val{\lambda(1-\mflt)}$ if $\frac{1}{2}<\val\lambda<1$. We construct strongly divisible modules for each case in the following three subsections.

In this section, we let $\mathfrak{U}_{1}:=p\baseo+\frac{\mflo}{\lambda}(\lambdat\baset+\baseth)+(\gamma+\mflt-1)\baseth$ for brevity.

\subsection{Strongly divisible modules: the first case}
In this subsection, we construct strongly divisible modules in $\mfd_{[\frac{1}{2},1]}$ under the assumption $\mathbf{H(1,1)}$ as at the beginning of this section.

\begin{prop}\label{prop-D1first}
Keep the assumption $\mathbf{H(1,1)}$. Then $\mfm_{[\frac{1}{2},1]}:=S_{\roi}(E_{1},E_{2},E_{3})$ is a strongly divisible module in $\mfd_{[\frac{1}{2},1]}$, where
    $$\left\{
        \begin{array}{ll}
          E_{1}=\mathfrak{U}_{1}+\frac{\lambdat(\gamma-1)}{\lambda}(\lambdat\baset+\baseth) & \hbox{} \\
          E_{2}=\frac{p}{\lambda}(\lambdat\baset+\baseth) & \hbox{} \\
          E_{3}=\lambda\baseth. & \hbox{}
        \end{array}
      \right.$$
\end{prop}

\begin{proof}
We note that $\val{\mflo-\lambdat}\geq 1$ implies $\val{\mflo}\geq \val\lambdat$. During the proof, we let $\mfm:=\mfm_{[\frac{1}{2},1]}$ for brevity. It is routine to check that $\vphi(E_{1})\equiv\vphi(E_{2})\equiv\vphi(E_{3})\equiv N(\Baseo)\equiv N(\Baset)\equiv N(\Baseth)\equiv 0$ modulo $\maxi{E}\mfm$, using the fact $\vphi(\gamma)\in p^{p-1}S$ and $N(\gamma)=-p[\gamma+(u-p)^{p-1}]$. Hence, $\mfm$ is stable under $\vphi$ and $N$.

We check that $\vphi(\filt \mfm)\subset p^{2}\mfm$. It is easy to check that $\filt S\cdot\mfm\subset\filt\mfm$, and so it is enough to check $\vphi(\mathfrak{X}_{1})\in p^{2}\mfm$, since $\vphi(\filt S\cdot\mfm)\subset p^{2}\mfm$. We first compute  $\filt \mfm$ modulo $\filt S\cdot\mfm$, by finding out the conditions that $\mathfrak{X}_{1}$ be in $\mfm$. $\mathfrak{X}_{1}\equiv C_{0}\Big(\frac{1}{p}E_{1}- \frac{\mflo-\lambdat}{p^{2}}\Baset+\frac{\lambda(\mflo-\lambdat)+\lambda\lambdat}{p\lambdat}\Baset +\frac{1-\mflt}{p\lambda}\Baseth-\frac{(\mflo-\lambdat)+\lambdat}{\lambda\lambdat}E_{3}- \frac{(1-\mflt)-1}{\lambda}\Baseth\Big)+ (u-p)\Big(\frac{C_{1}}{p}E_{1}+ \frac{p\lambda\mflo C_{1}+p\lambda C_{2}-\lambdat(\mflo-\lambdat)C_{1}}{p^{2}\lambdat}E_{2}+ \frac{V}{W}E_{3}\Big)$ modulo $\filt S\cdot\mfm$, where $W=p\lambda\lambdat$ and $$V=\lambdat (C_{0}+p\mflt C_{1})+\lambdat(1-\mflt)C_{1}-p(\mflo C_{1}+C_{2}).$$
Hence, if $\mathfrak{X}_{1}\in\filt \mfm$, then we get $\val{C_{0}}\geq1$, $\val{C_{1}}\geq1$,
\begin{equation}\label{D1-first}
\Val{p\lambda\mflo C_{1}+p\lambda C_{2}-\lambdat(\mflo-\lambdat)C_{1}}\geq\val{p^{2}\lambdat},
\end{equation}
and
\begin{equation}\label{D1-second}
\val{V}\geq \val{W}.
\end{equation}
The inequality (\ref{D1-first}) is equivalent to $$\val{C_{2}}\geq\val{p\lambdat}-\val\lambda$$ since $\val{C_{1}}\geq1$ and $\val{\mflo-\lambdat}\geq1$. The inequality (\ref{D1-second}) is equivalent to
\begin{equation*}
\val{\lambdat C_{0}-pC_{2}}\geq\val{p\lambda\lambdat},
\end{equation*}
which implies $$\val{C_{0}}\geq2-\val\lambda$$ since $\val{C_{2}}\geq\val{p\lambdat}-\val\lambda$.

Finally, we check $\vphi(\mathfrak{X}_{1})\in p^{2}\mfm$. Using the fact $\vphi(\frac{u-p}{p})\equiv\gamma-1$ modulo $pS$, $\vphi(\mathfrak{X}_{1})\equiv\lambda C_{0}\Big(\Baseo-\frac{\lambdat(\gamma-1)}{p}\Baset\Big)+ p\lambda C_{1}(\gamma-1)\Big(\Baseo-\frac{\gamma-1}{\lambda}\Baseth-\frac{\lambdat(\gamma-1)}{p}\Baset\Big)+ \lambda C_{2}(\gamma-1)\Baset$ modulo $p^{2}\mfm$. Then $\vphi(\mathfrak{X}_{1})\equiv-\lambda\frac{\lambdat C_{0}-pC_{2}}{p}(\gamma-1)\Baset$ modulo $p^{2}\mfm$ since $\val{C_{0}}\geq2-\val\lambda$ and $\val{C_{1}}\geq1$, and so $\vphi(\mathfrak{X}_{1})\equiv0$ modulo $p^{2}\mfm$ since $\val{\lambdat C_{0}-pC_{2}}\geq\val{p\lambda\lambdat}$.  Thus $\vphi(\filt \mfm)\subset p^{2}\mfm$.
\end{proof}

\subsection{Strongly divisible modules: the second case}\label{subsec-D1second}
In this subsection, we construct strongly divisible modules in $\mfd_{[\frac{1}{2},1]}$ under the assumption $\mathbf{H(1,2)}$ as at the beginning of this section.

We first define two sequences $G_{m}$ and $H_{m}$ for $m\geq0$ recursively as follows: $G_{0}=H_{0}=1$ and
$$G_{m+1}=(1-\mflt)[(1-\mflt)H_{m}+\frac{p\lambda}{\lambdat}G_{m}]^{2};$$
$$H_{m+1}=G_{m+1}+\frac{p}{\lambdat}\big([(\mflo-\lambdat)-\lambda(1-\mflt)][(1-\mflt)H_{m}+\frac{p\lambda}{\lambdat}G_{m}]+ p\lambda G_{m}\big)G_{m}.$$
We prove that the sequence $G_{m}/H_{m}$ converses in $1+\maxi{E}$. The limit appears in the coefficients of the strongly divisible modules in Proposition \ref{prop-D1second}.

\begin{lemm}
Keep the assumption $\mathbf{H(1,2)}$. Then, for $m\geq0$,
\begin{enumerate}
\item $\val{G_{m}-H_{m}}-\val{G_{m}}\geq\min\big\{\Val{\frac{p}{\lambdat}[(\mflo-\lambdat)-\lambda(1-\mflt)]}-2\val{1-\mflt}, 3[\val\lambda-\val{1-\mflt}]\big\}>0$;
\item $\Val{(1-\mflt)H_{m}+\frac{p\lambda}{\lambdat}G_{m}}=\Val{1-\mflt}+\val{H_{m}}$
\item $\val{G_{m}H_{m+1}-H_{m}G_{m+1}}-\val{H_{m}H_{m+1}}\geq(m+1)\min\big\{\Val{\frac{p}{\lambdat} [(\mflo-\lambdat)-\lambda(1-\mflt)]}-2\val{1-\mflt}, 3[\val{\lambda}-\val{1-\mflt}]\big\}$.
\end{enumerate}
\end{lemm}

\begin{proof}
(2) is immediate from (1) since (1) implies $\val{G_{m}}=\val{H_{m}}$ for all $m\geq0$ and since $\val{1-\mflt}<\val{\lambda}\leq\val{\frac{p\lambda}{\lambdat}}$.

We prove (1) by induction. For $m=0$, it is trivial. Assume that (1) is true for $m$. Then $\val{G_{m}}=\val{H_{m}}$ and $\val{(1-\mflt)H_{m}+\frac{p\lambda}{\lambdat}G_{m}}=\val{(1-\mflt)H_{m}}$. Hence, we have $\val{G_{m+1}}=3\val{1-\mflt}+2\val{H_{m}}$ and
$$
\begin{aligned}
\val{&G_{m+1}-H_{m+1}}-\val{G_{m+1}}\\
=&\Val{\frac{p}{\lambdat}\big([(\mflo-\lambdat)-\lambda(1-\mflt)][(1-\mflt)H_{m}+\frac{p\lambda}{\lambdat}G_{m}]+p\lambda G_{m}\big)G_{m}}-\val{G_{m+1}}\\
\geq&\min\big\{\Val{\frac{p}{\lambdat}[(\mflo-\lambdat)-\lambda(1-\mflt)](1-\mflt)H_{m}G_{m}},\val{\frac{p^{2}\lambda}{\lambdat} G_{m}^{2}}\big\}-\val{G_{m+1}}\\
=&\min\big\{\Val{\frac{p}{\lambdat}[(\mflo-\lambdat)-\lambda(1-\mflt)]}-2\val{1-\mflt},3[\val\lambda-\val{1-\mflt}]\big\}>0.
\end{aligned}
$$
Hence, (1) holds by induction.

For (3), we induct on $m$ as well. If $m=0$, then $G_{0}H_{1}-H_{0}G_{1}= H_{1}-G_{1}= \frac{p}{\lambdat}\big([(\mflo-\lambdat)-\lambda(1-\mflt)][(1-\mflt)+\frac{p\lambda}{\lambdat}]+p\lambda\big)$. Hence, it holds for $m=0$.

We prove the following identity: for $m\geq1$,
$$G_{m}H_{m+1}-H_{m}G_{m+1}=-(W_{1}+W_{2})(G_{m-1}H_{m}-H_{m-1}G_{m}),$$
where $W_{1}=\frac{p}{\lambdat}[(\mflo-\lambdat)-\lambda(1-\mflt)](1-\mflt)^{2}[(1-\mflt)H_{m-1}+ \frac{p\lambda}{\lambdat}G_{m-1}][(1-\mflt)H_{m}+\frac{p\lambda}{\lambdat}G_{m}]$ and  $W_{2}=\frac{p^{2}\lambda}{\lambdat}(1-\mflt)^{2} \big([(1-\mflt)H_{m-1}+\frac{p\lambda}{\lambdat}G_{m-1}]G_{m}+ G_{m-1}[(1-\mflt)H_{m}+\frac{p\lambda}{\lambdat}G_{m}]\big)$.

Indeed, $G_{m}H_{m+1}-H_{m}G_{m+1}=G_{m}(H_{m+1}-G_{m+1})+(G_{m}-H_{m})G_{m+1}= \frac{p}{\lambdat}\big([(\mflo-\lambdat)-\lambda(1-\mflt)][(1-\mflt)H_{m}+\frac{p\lambda}{\lambdat}G_{m}]+p\lambda G_{m}\big)G_{m}^{2}- \frac{p}{\lambdat}\big([(\mflo-\lambdat)-\lambda(1-\mflt)][(1-\mflt)H_{m-1}+\frac{p\lambda}{\lambdat}G_{m-1}]+p\lambda G_{m-1}\big)G_{m-1}G_{m+1}=  \frac{p}{\lambdat}[(\mflo-\lambdat)-\lambda(1-\mflt)] \big(G_{m}^{2}[(1-\mflt)H_{m}+\frac{p\lambda}{\lambdat}G_{m}]- G_{m-1}[(1-\mflt)H_{m-1}+\frac{p\lambda}{\lambdat}G_{m-1}]G_{m+1}\big)+ \frac{p^{2}\lambda}{\lambdat}(G_{m}^{3}-G_{m-1}^{2}G_{m+1})$. However, $G_{m}^{2}[(1-\mflt)H_{m}+\frac{p\lambda}{\lambdat}G_{m}]-G_{m-1}[(1-\mflt)H_{m-1}+\frac{p\lambda}{\lambdat}G_{m-1}]G_{m+1}= (1-\mflt)[(1-\mflt)H_{m-1}+\frac{p\lambda}{\lambdat}G_{m-1}][(1-\mflt)H_{m}+\frac{p\lambda}{\lambdat}G_{m}] \big([(1-\mflt)H_{m-1}+\frac{p\lambda}{\lambdat}G_{m-1}]G_{m}-G_{m-1}[(1-\mflt)H_{m}+\frac{p\lambda}{\lambdat}G_{m}]\big) =-(1-\mflt)^{2} [(1-\mflt)H_{m-1}+\frac{p\lambda}{\lambdat}G_{m-1}][(1-\mflt)H_{m}+\frac{p\lambda}{\lambdat}G_{m}] (G_{m-1}H_{m}-H_{m-1}G_{m})$ and $G_{m}^{3}-G_{m-1}^{2}G_{m+1}= (1-\mflt)\big([(1-\mflt)H_{m-1}+\frac{p\lambda}{\lambdat}G_{m-1}]^{2}G_{m}^{2}- G_{m-1}^{2}[(1-\mflt)H_{m}+\frac{p\lambda}{\lambdat}G_{m}]^{2}\big)= (1-\mflt)\big(G_{m-1}[(1-\mflt)H_{m}+\frac{p\lambda}{\lambdat}G_{m}]+ [(1-\mflt)H_{m-1}+\frac{p\lambda}{\lambdat}G_{m-1}]G_{m}\big)(1-\mflt)(H_{m-1}G_{m}-G_{m-1}H_{m})= -(1-\mflt)^{2}\big([(1-\mflt)H_{m-1}+\frac{p\lambda}{\lambdat}G_{m-1}]G_{m}+ G_{m-1}[(1-\mflt)H_{m}+\frac{p\lambda}{\lambdat}G_{m}]\big)(G_{m-1}H_{m}-H_{m-1}G_{m})$. Hence, we proved the identity.

The identity implies that $\val{G_{m}H_{m+1}-H_{m}G_{m+1}}\geq\min\{\val{W_{1}},\val{W_{2}}\}+ \val{G_{m-1}H_{m}-H_{m-1}G_{m}}$. By parts (1) and (2), $\val{W_{1}}=\Val{\frac{p}{\lambdat}[(\mflo-\lambdat)-\lambda(1-\mflt)]}+4\val{1-\mflt}+\val{H_{m-1}H_{m}}$ and $\val{W_{2}}=3\val\lambda+3\val{1-\mflt}+\val{H_{m-1}H_{m}}$. Thus, $\val{G_{m}H_{m+1}-H_{m}G_{m+1}}\geq \min\big\{\Val{\frac{p}{\lambdat} [(\mflo-\lambdat)-\lambda(1-\mflt)]}-2\val{1-\mflt}, 3[\val{\lambda}-\val{1-\mflt}]\big\}+6\val{1-\mflt}+\val{H_{m-1}H_{m}}+ m\min\big\{\Val{\frac{p}{\lambdat} [(\mflo-\lambdat)-\lambda(1-\mflt)]}-2\val{1-\mflt}, 3[\val{\lambda}-\val{1-\mflt}]\big\}+\val{H_{m-1}H_{m}}$ by induction hypothesis. Hence, it holds by induction.
\end{proof}

The assumption $\mathbf{H(1,2)}$ implies that the quantities in the set of the part (3) of the lemma above are strictly positive. So the part (3) of the lemma above says that $\val{G_{m}H_{m+1}-H_{m}G_{m+1}}-\val{H_{m}H_{m+1}}$ approaches $\infty$ as $m$ goes to $\infty$. That is, the sequence $G_{m}/H_{m}$ is Cauchy. We let $$\Delta_{\tiny{\ref{subsec-D1second}}}:=\lim_{m\mapsto\infty}\frac{G_{m}}{H_{m}}.$$
Note that $\Delta_{\tiny{\ref{subsec-D1second}}}$ depends on the values of the parameters $\lambda,\lambdat,\mflo,\mflt$.

The part (1) of the lemma implies that $\val{G_{m}}=\val{H_{m}}$ for all $m\geq0$ and
\begin{multline*}
\val{1-\Delta_{\tiny{\ref{subsec-D1second}}}}\geq\min\big\{3[\val{\lambda}-\val{1-\mflt}],\\
\Val{\frac{p}{\lambdat} [(\mflo-\lambdat)-\lambda(1-\mflt)]}-2\val{1-\mflt} \big\}>0,
\end{multline*}
which immediately implies that $$\frac{\lambdat(1-\mflt)+p\lambda\Delta_{\tiny{\ref{subsec-D1second}}}}{\lambdat(1-\mflt)}\in1+\maxid.$$
In particular, $\Val{\lambdat(1-\mflt)+p\lambda\Delta_{\tiny{\ref{subsec-D1second}}}}=\Val{\lambdat(1-\mflt)}$.

It is also easy to check that $\Delta_{\tiny{\ref{subsec-D1second}}}$ satisfies the equation
\begin{multline}\label{D1-third}
(1-\mflt)[\lambdat(1-\mflt)+p\lambda\Delta_{\tiny{\ref{subsec-D1second}}}]^{2} (1-\Delta_{\tiny{\ref{subsec-D1second}}})-\\
 p\big([(\mflo-\lambdat)-\lambda(1-\mflt)][\lambdat(1-\mflt)+p\lambda\Delta_{\tiny{\ref{subsec-D1second}}}]+p\lambda\lambdat \Delta_{\tiny{\ref{subsec-D1second}}}\big)\Delta_{\tiny{\ref{subsec-D1second}}}^{2}=0,
\end{multline}
by taking the limits of $G_{m+1}/H_{m+1}= (1-\mflt)[(1-\mflt)H_{m}+\frac{p\lambda}{\lambdat}G_{m}]^{2}/\big[(1-\mflt)[(1-\mflt)H_{m}+ \frac{p\lambda}{\lambdat}G_{m}]^{2}+\frac{p}{\lambdat}\big([(\mflo-\lambdat)-\lambda(1-\mflt)][(1-\mflt)H_{m}+\frac{p\lambda}{\lambdat}G_{m}]+ p\lambda G_{m}\big)G_{m}\big]$. The equation plays a crucial role in the proof of the following proposition.

\begin{prop}\label{prop-D1second}
Keep the assumption $\mathbf{H(1,2)}$. Then $\mfm_{[\frac{1}{2},1]}:=S_{\roi}(E_{1},E_{2},E_{3})$ is a strongly divisible module in $\mfd_{[\frac{1}{2},1]}$, where
    $$\left\{
        \begin{array}{ll}
          E_{1}=\mathfrak{U}_{1}+\frac{\lambdat^{2} (1-\mflt)(\gamma-1)}{\lambda [\lambdat(1-\mflt)+p\lambda\Delta_{\tiny{\ref{subsec-D1second}}}]}(\lambdat\baset+\baseth) & \hbox{} \\
          E_{2}=\frac{\lambdat(1-\mflt)}{p}(\lambdat\baset+\baseth)- \lambda\Delta_{\tiny{\ref{subsec-D1second}}}(\gamma-1)\baseth & \hbox{} \\
          E_{3}=\frac{\lambda[\lambdat(1-\mflt)+p\lambda\Delta_{\tiny{\ref{subsec-D1second}}}]}{p}\baseth. & \hbox{}
        \end{array}
      \right.$$
\end{prop}
Note that $S_{\roi}(\Baseo,\Baset,\Baseth)=S_{\roi}(\Baseo,\Baset,\frac{p(1-\mflt)}{\lambda}\baseth)$.

\begin{proof}
During the proof, we let $\Delta:=\Delta_{\tiny{\ref{subsec-D1second}}}$ and $\mfm:=\mfm_{[\frac{1}{2},1]}$ for brevity. It is routine to check that $\vphi(\Baseo)\equiv-\frac{p\lambda(1-\mflt)-p(\mflo-\lambdat)}{\lambda\lambdat(1-\mflt)}\Baseth$, $\vphi(\Baset)\equiv\Baseth$, and $\vphi(\Baseth)\equiv N(\Baseo)\equiv N(\Baset)\equiv N(\Baseth)\equiv 0$ modulo $\maxi{E}\mfm$. Hence, $\mfm$ is stable under $\vphi$ and $N$.

We check that $\vphi(\filt \mfm)\subset p^{2}\mfm$. It is easy to check that $\filt S\cdot\mfm\subset\filt\mfm$, and so it is enough to check $\vphi(\mathfrak{X}_{1})\in p^{2}\mfm$, since $\vphi(\filt S\cdot\mfm)\subset p^{2}\mfm$. We compute  $\filt \mfm$.
$\mathfrak{X}_{1}\equiv C_{0}\Big(\frac{1}{p}\Baseo-\frac{(\mflo-\lambdat)[\lambdat(1-\mflt)+p\lambda\Delta]+ p\lambda\lambdat\Delta}{\lambda\lambdat(1-\mflt)[\lambdat(1-\mflt)+p\lambda\Delta]}\Baset+ \frac{p(\mflo-\lambdat)+p\lambdat}{\lambdat^{2}(1-\mflt)}\Baset+ \frac{p(\mflo-\lambdat)[\lambdat(1-\mflt)+p\lambda\Delta]\Delta+ p^{2}\lambda\lambdat\Delta^{2}}{\lambda\lambdat(1-\mflt)[\lambdat(1-\mflt)+p\lambda\Delta]^{2}}\Baseth+ \frac{(1-\mflt)+p\mflt}{\lambda[\lambdat(1-\mflt)+p\lambda\Delta]}\Baseth- \frac{p(\mflo-\lambdat)+p\lambdat}{\lambda\lambdat^{2}(1-\mflt)}\Baseth\Big)+ (u-p)\Big(\frac{C_{1}}{p}\Baseo+ \frac{p\lambda[\lambdat(1-\mflt)+p\lambda\Delta](\mflo C_{1}+C_{2})-\lambdat\mflo[\lambdat(1-\mflt)+p\lambda\Delta]C_{1}+\lambdat^{3}(1-\mflt)C_{1}}{\lambda\lambdat^{2}(1-\mflt) [\lambdat(1-\mflt)+p\lambda\Delta]}\Baset+\frac{V}{W}\Baseth\Big)$ modulo $\filt S\cdot\mfm$, where $W=\lambda\lambdat^{2}(1-\mflt)[\lambdat(1-\mflt)+p\lambda\Delta]^{2}$ and
\begin{multline*}
V=\lambdat^{2}(1-\mflt)[\lambdat(1-\mflt)+p\lambda\Delta] (C_{0}+p\mflt C_{1})+\\ p\lambdat\mflo [\lambdat(1-\mflt)+p\lambda\Delta]\Delta C_{1}- p\lambdat^{3}(1-\mflt)\Delta C_{1}+\\ \lambdat^{2}(1-\mflt)^{2}[\lambdat(1-\mflt)+p\lambda\Delta]C_{1}- p[\lambdat(1-\mflt)+p\lambda\Delta]^{2}(\mflo C_{1}+C_{2}).
\end{multline*}
Hence, if $\mathfrak{X}_{1}\in\filt \mfm$, then we get $\val{C_{0}}\geq\val{\lambda\lambdat}=2-\val\lambda$, $\val{C_{1}}\geq1$,
\begin{multline}\label{D1-fourth}
\Val{p\lambda[\lambdat(1-\mflt)+p\lambda\Delta](\mflo C_{1}+C_{2})-\lambdat\mflo[\lambdat(1-\mflt)+p\lambda\Delta]C_{1}+\\
\lambdat^{3}(1-\mflt)C_{1}}\geq \Val{\lambda\lambdat^{2}(1-\mflt) [\lambdat(1-\mflt)+p\lambda\Delta]},
\end{multline}
and
\begin{equation}\label{D1-fifth}
\val{V}\geq \val{W}.
\end{equation}
Using the inequality $\val{C_{1}}\geq1$, one can readily check that
\begin{multline*}
V\equiv\lambdat^{2}(1-\mflt)[\lambdat(1-\mflt)+p\lambda\Delta] C_{0}+ p\lambdat(\mflo-\lambdat)[\lambdat(1-\mflt)+p\lambda\Delta]\Delta C_{1}+\\ p^{2}\lambda\lambdat^{2}\Delta^{2}C_{1}+ \lambdat^{2}(1-\mflt)^{2}[\lambdat(1-\mflt)+p\lambda\Delta]C_{1}-p[\lambdat(1-\mflt)+p\lambda\Delta]^{2}C_{2}
\end{multline*}
modulo $(W)$. Thus
\begin{multline*}
\Delta\cdot V\equiv [\lambdat(1-\mflt)+p\lambda\Delta]\big(\lambdat^{2}(1-\mflt)\Delta C_{0}+\\
\lambdat(1-\mflt)[\lambdat(1-\mflt)+p\lambda\Delta] C_{1}- p[\lambdat(1-\mflt)+p\lambda\Delta]\Delta C_{2}\big)+X
\end{multline*} modulo $(W)$,
where
\begin{multline*}
X=\lambdat\big[\lambdat(1-\mflt)^{2}[\lambdat(1-\mflt)+p\lambda\Delta]\Delta -(1-\mflt)[\lambdat(1-\mflt)+p\lambda\Delta]^{2}+\\
p\big((\mflo-\lambdat) [\lambdat(1-\mflt)+p\lambda\Delta]+p\lambda\lambdat\Delta\big)\Delta^{2}\big]C_{1}.
\end{multline*}
But $X=0$ by the equation (\ref{D1-third}). Hence, from the inequality (\ref{D1-fifth}), we have
\begin{multline}\label{D1-sixth}
\Val{\lambdat^{2}(1-\mflt)\Delta C_{0}+
\lambdat(1-\mflt)[\lambdat(1-\mflt)+p\lambda\Delta] C_{1}-\\
p[\lambdat(1-\mflt)+p\lambda\Delta]\Delta C_{2}}\geq\Val{\lambda\lambdat^{2}(1-\mflt)[\lambdat(1-\mflt)+p\lambda\Delta]}
\end{multline}
and it is immediate that the inequality (\ref{D1-sixth}) implies
\begin{multline}\label{D1-seventh}
\Val{\lambdat^{2}(1-\mflt) C_{0}-p[\lambdat(1-\mflt)+p\lambda\Delta]C_{2}} \\ \geq\Val{p\lambdat(1-\mflt)[\lambdat(1-\mflt)+p\lambda\Delta]}
\end{multline}
since $\val{C_{1}}\geq1$.

Finally, we check $\vphi(\mathfrak{X}_{1})\in p^{2}\mfm$. Using the fact that $\vphi(\frac{u-p}{p})\equiv\gamma-1$ modulo $pS$, it is easy to check that $\vphi(\mathfrak{X}_{1})\equiv\lambda C_{0}\Big(\Baseo-\frac{p\lambdat(\gamma-1)}{\lambda[\lambdat(1-\mflt)+p\lambda\Delta]}\Baset- \frac{p^{2}\lambdat\Delta(\gamma-1)^{2}}{\lambda[\lambdat(1-\mflt)+p\lambda\Delta]^{2}}\Baseth\Big)+ p\lambda C_{1}(\gamma-1)\Big(\Baseo-\frac{p(\gamma-1)}{\lambda[\lambdat(1-\mflt)+p\lambda\Delta]}\Baseth -\frac{p\lambdat(\gamma-1)}{\lambda[\lambdat(1-\mflt)+p\lambda\Delta]}\Baset- \frac{p^{2}\lambdat\Delta(\gamma-1)^{2}}{\lambda[\lambdat(1-\mflt)+p\lambda\Delta]^{2}}\Baseth\Big)+ pC_{2}(\gamma-1)\Big(\frac{p}{\lambdat(1-\mflt)}\Baset+ \frac{p^{2}\Delta(\gamma-1)}{\lambdat(1-\mflt)[\lambdat(1-\mflt)+p\lambda\Delta]}\Baseth\Big)$ modulo $p^{2}\mfm$. Since $\val{C_{0}}\geq2-\val\lambda$ and $\val{C_{1}}\geq1$, $\vphi(\mathfrak{X}_{1})\equiv -p\frac{\lambdat^{2}(1-\mflt)C_{0}-p[\lambdat(1-\mflt)+p\lambda\Delta]C_{2}}{\lambdat(1-\mflt) [\lambdat(1-\mflt)+p\lambda\Delta]}(\gamma-1)\Baset- p^{2}\frac{\lambdat^{2}(1-\mflt)\Delta C_{0}+ \lambdat(1-\mflt)[\lambdat(1-\mflt)+p\lambda\Delta]C_{1}- p[\lambdat(1-\mflt)+p\lambda\Delta] \Delta C_{2}}{\lambdat(1-\mflt)[\lambdat(1-\mflt)+p\lambda\Delta]^{2}}(\gamma-1)^{2}\Baseth$ modulo $p^{2}\mfm$.
Then $\vphi(\mathfrak{X}_{1})\equiv0$ modulo $p^{2}\mfm$ by the inequalities (\ref{D1-sixth}) and (\ref{D1-seventh}). Thus $\vphi(\filt \mfm)\subset p^{2}\mfm$.
\end{proof}

Note that the inequality (\ref{D1-seventh}) implies $\val{C_{2}}\geq\min\{\val{\frac{p\lambdat}{\lambda}},\val{\lambdat(1-\mflt)}\}$ since $\val{C_{0}}\geq2-\val\lambda$, and so one can readily check that $p\lambda[\lambdat(1-\mflt)+p\lambda\Delta](\mflo C_{1}+C_{2})-\lambdat\mflo[\lambdat(1-\mflt)+p\lambda\Delta]C_{1}+
\lambdat^{3}(1-\mflt)C_{1}\equiv0$ modulo $(\lambda\lambdat^{2}(1-\mflt) [\lambdat(1-\mflt)+p\lambda\Delta])$, that is, the inequality (\ref{D1-fourth}) does not give any stronger condition.

\subsection{Strongly divisible modules: the third case}\label{subsec-D1third}
In this subsection, we construct strongly divisible modules in $\mfd_{[\frac{1}{2},1]}$ under the assumption $\mathbf{H(1,3)}$ as at the beginning of this section.

We first define two sequences $G_{m}$ and $H_{m}$ for $m\geq0$ recursively as follows: $G_{0}=H_{0}=1$ and
$$G_{m+1}=(\mflo H_{m}-\lambdat G_{m})^{2};$$ $$H_{m+1}=G_{m+1}-\lambda[(1-\mflt)(\mflo H_{m}-\lambdat G_{m})+p\lambda G_{m}]H_{m}.$$
We prove that the sequence $G_{m}/H_{m}$ converses in $1+\maxi{E}$. The limit appears in the coefficients of the strongly divisible modules in Proposition \ref{prop-D1third}.

\begin{lemm}
Keep the assumption $\mathbf{H(1,3)}$. Then, for $m\geq0$,
\begin{enumerate}
\item $\val{G_{m}-H_{m}}-\val{G_{m}}\geq\min\big\{\val{\lambda(1-\mflt)}-\val{\mflo-\lambdat}, \val{p\lambda^{2}}-2\val{\mflo-\lambdat}\big\}>0$;
\item $\val{\mflo H_{m}-\lambdat G_{m}}=\val{\mflo-\lambdat}+\val{H_{m}}$;
\item $\val{G_{m}H_{m+1}-H_{m}G_{m+1}}-\val{H_{m}H_{m+1}}\geq\min\big\{\Val{\lambda(1-\mflt)}-\val{\mflo-\lambdat}, \val{p\lambda^{2}}-2\val{\mflo-\lambdat}\big\}+m\min\big\{\Val{\lambda\lambdat(1-\mflt)}-2\val{\mflo-\lambdat}, \val{p\lambda^{2}}-2\val{\mflo-\lambdat},3[1-\val{\mflo-\lambdat}]\big\}$.
\end{enumerate}
\end{lemm}

\begin{proof}
(2) is immediate from (1) since (1) implies $\val{G_{m}}=\val{H_{m}}$ for all $m\geq0$ and $\val{\mflo H_{m}-\lambdat G_{m}}=\Val{(\mflo-\lambdat)H_{m}-\lambdat(G_{m}-H_{m})}=\val{\mflo-\lambdat}+\val{H_{m}}$.

We prove (1) by induction. For $m=0$, it is trivial. Assume that (1) is true for $m$. Then $\val{G_{m}}=\val{H_{m}}$ and
$\val{\mflo H_{m}-\lambdat G_{m}}=\val{\mflo-\lambdat}+\val{H_{m}}$. Hence, we have $\val{G_{m+1}}=2\val{\mflo-\lambdat}+2\val{H_{m}}$ and
$$
\begin{aligned}
\val{&G_{m+1}-H_{m+1}}-\val{G_{m+1}}\\
&=\Val{\lambda[(1-\mflt)(\mflo H_{m}-\lambdat G_{m})+p\lambda G_{m}]H_{m}}-\val{G_{m+1}}\\
&\geq\min\big\{\Val{\lambda(1-\mflt)(\mflo H_{m}-\lambdat G_{m})H_{m}},\val{p\lambda^{2} G_{m}H_{m}}\big\}-\val{G_{m+1}}\\
&=\min\big\{\Val{\lambda(1-\mflt)}-\val{\mflo-\lambdat},\val{p\lambda^{2}}-2\val{\mflo-\lambdat}\big\}>0.
\end{aligned}
$$
Hence, (1) holds by induction.

For (3), we induct on $m$ as well. If $m=0$, then $G_{0}H_{1}-H_{0}G_{1}=H_{1}-G_{1}= -\lambda[(1-\mflt)(\mflo-\lambdat)+p\lambda]$. So it holds for $m=0$.

We claim the following identity: for $m\geq1$, $$G_{m}H_{m+1}-H_{m}G_{m+1}=(W_{1}+W_{2}+W_{3})(G_{m-1}H_{m}-H_{m-1}G_{m}),$$ where
$W_{1}=\lambda\lambdat(1-\mflt)(\mflo H_{m-1}-\lambdat G_{m-1})(\mflo H_{m}-\lambdat G_{m})$,
$W_{2}=p\lambda^{2}(\mflo H_{m-1}-\lambdat G_{m-1})(\mflo H_{m}-\lambdat G_{m})$, and
$W_{3}=p\lambda^{2}\lambdat[(\mflo H_{m-1}-\lambdat G_{m-1})G_{m}+G_{m-1}(\mflo H_{m}-\lambdat G_{m})]$.

Indeed, $G_{m}H_{m+1}-H_{m}G_{m+1}=G_{m}(H_{m+1}-G_{m+1})+(G_{m}-H_{m})G_{m+1}= -\lambda G_{m}[(1-\mflt)(\mflo H_{m}-\lambdat G_{m})+p\lambda G_{m}]H_{m}+ \lambda[(1-\mflt)(\mflo H_{m-1}-\lambdat G_{m-1})+p\lambda G_{m-1}]H_{m-1}G_{m+1}= \lambda(1-\mflt)[(\mflo H_{m-1}-\lambdat G_{m-1})H_{m-1}G_{m+1}-G_{m}H_{m}(\mflo H_{m}-\lambdat G_{m})]+p\lambda^{2}(G_{m-1}H_{m-1}G_{m+1}-G_{m}^{2}H_{m})$. But $(\mflo H_{m-1}-\lambdat G_{m-1})H_{m-1}G_{m+1}-(\mflo H_{m}-\lambdat G_{m})G_{m}H_{m}= (\mflo H_{m-1}-\lambdat G_{m-1})(\mflo H_{m}-\lambdat G_{m})[H_{m-1}(\mflo H_{m}-\lambdat G_{m})- (\mflo H_{m-1}-\lambdat G_{m-1})H_{m}]= \lambdat(\mflo H_{m-1}-\lambdat G_{m-1})(\mflo H_{m}-\lambdat G_{m})(G_{m-1}H_{m}-H_{m-1}G_{m})$ and $G_{m-1}H_{m-1}G_{m+1}-G_{m}^{2}H_{m}=(\mflo H_{m}-\lambdat G_{m})^{2}G_{m-1}H_{m-1}- (\mflo H_{m-1}-\lambdat G_{m-1})^{2}G_{m}H_{m}=(\mflo H_{m}-\lambdat G_{m})G_{m-1}[H_{m-1}(\mflo H_{m}-\lambdat G_{m})- H_{m}(\mflo H_{m-1}-\lambdat G_{m-1})]+ H_{m} (\mflo H_{m-1}-\lambdat G_{m-1})[G_{m-1}(\mflo H_{m}-\lambdat G_{m})- (\mflo H_{m-1}-\lambdat G_{m-1})G_{m}]  =\lambdat G_{m-1}(\mflo H_{m}-\lambdat G_{m})(G_{m-1}H_{m}-H_{m-1}G_{m})+ \mflo H_{m}(\mflo H_{m-1}-\lambdat G_{m-1})(G_{m-1}H_{m}-H_{m-1}G_{m})$. Hence, we proved the identity.

The identity implies  $\val{G_{m}H_{m+1}-H_{m}G_{m+1}}\geq\min\{\val{W_{1}},\val{W_{2}},\val{W_{3}}\}+ \val{G_{m-1}H_{m}-H_{m-1}G_{m}}$. By parts (1) and (2), $\val{W_{1}}=\val{\lambda\lambdat}+\val{1-\mflt}+ 2\val{\mflo-\lambdat}+\val{H_{m-1}H_{m}}$, $\val{W_{2}}=\val{p\lambda^{2}}+2\val{\mflo-\lambdat}+\val{H_{m-1}H_{m}}$, and $\val{W_{3}}=\val{p\lambda^{2}\lambdat}+\val{\mflo-\lambdat}+\val{H_{m-1}H_{m}}$. Thus, $\val{G_{m}H_{m+1}-H_{m}G_{m+1}}\geq \min\big\{\Val{\lambda\lambdat(1-\mflt)}-2\val{\mflo-\lambdat}, \val{p\lambda^{2}}-2\val{\mflo-\lambdat},3[1-\val{\mflo-\lambdat}]\big\}+4\val{\mflo-\lambdat}+\val{H_{m-1}H_{m}}+ \min\big\{\Val{\lambda(1-\mflt)}-\val{\mflo-\lambdat}, \val{p\lambda^{2}}-2\val{\mflo-\lambdat}\big\}+(m-1)\min\big\{\Val{\lambda\lambdat(1-\mflt)}-2\val{\mflo-\lambdat}, \val{p\lambda^{2}}-2\val{\mflo-\lambdat},3[1-\val{\mflo-\lambdat}]\big\}+\val{H_{m-1}H_{m}}$ by induction hypothesis. Hence, it holds by induction.
\end{proof}

The assumption $\mathbf{H(1,3)}$ implies that the quantities in the sets of the part (3) of the lemma above are strictly positive. So the part (3) of the lemma says that $\val{G_{m}H_{m+1}-H_{m}G_{m+1}}-\val{H_{m}H_{m+1}}$ approaches $\infty$ as $m$ goes to $\infty$. That is, the sequence $G_{m}/H_{m}$ is Cauchy. We let $$\Delta_{\tiny{\ref{subsec-D1third}}}:=\lim_{m\mapsto\infty}\frac{G_{m}}{H_{m}}.$$
Note that $\Delta_{\tiny{\ref{subsec-D1third}}}$ depends on the values of the parameters $\lambda,\lambdat,\mflo,\mflt$.

The part (1) of the lemma implies that $\val{G_{m}}=\val{H_{m}}$ for all $m\geq0$ and
\begin{equation*}
\val{1-\Delta_{\tiny{\ref{subsec-D1third}}}}\geq
\min\big\{\Val{\lambda(1-\mflt)}-\val{\mflo-\lambdat},\val{p\lambda^{2}}-2\val{\mflo-\lambdat}\big\}>0,
\end{equation*}
which immediately implies $$\frac{\mflo-\lambdat\Delta_{\tiny{\ref{subsec-D1third}}}}{\mflo-\lambdat}\in1+\maxid.$$
In particular, $\val{\mflo-\lambdat\Delta_{\tiny{\ref{subsec-D1third}}}}=\val{\mflo-\lambdat}$.

It is also easy to check that $\Delta_{\tiny{\ref{subsec-D1third}}}$ satisfies the equation
\begin{equation}\label{D1-eighth}
(\mflo-\lambdat \Delta_{\tiny{\ref{subsec-D1third}}})^{2}(1-\Delta_{\tiny{\ref{subsec-D1third}}})+\lambda[(1-\mflt)(\mflo-\lambdat \Delta_{\tiny{\ref{subsec-D1third}}})+p\lambda\Delta_{\tiny{\ref{subsec-D1third}}}] \Delta_{\tiny{\ref{subsec-D1third}}}=0,
\end{equation}
by taking the limits of $G_{m+1}/H_{m+1}=(\mflo H_{m}-\lambdat G_{m})^{2}/\big[(\mflo H_{m}-\lambda G_{m})^{2}-\lambda\big((1-\mflt)(\mflo H_{m}-\lambdat G_{m})+p\lambda G_{m}\big)H_{m}\big]$. The equation plays a crucial role in the proof of the following proposition.

\begin{prop}\label{prop-D1third}
Keep the assumption $\mathbf{H(1,3)}$. Then $\mfm_{[\frac{1}{2},1]}:=S_{\roi}(E_{1},E_{2},E_{3})$ is a strongly divisible module in $\mfd_{[\frac{1}{2},1]}$, where
    $$\left\{
        \begin{array}{ll}
          E_{1}=\mathfrak{U}_{1}+\frac{p\Delta_{\tiny{\ref{subsec-D1third}}} (\gamma-1)}{\lambda(\mflo-\lambdat\Delta_{\tiny{\ref{subsec-D1third}}})}\left( \frac{\lambdat(\mflo-\lambdat\Delta_{\tiny{\ref{subsec-D1third}}})}{p} (\lambdat\baset+\baseth)-\lambda^{2}(\gamma-1)\baseth\right) & \hbox{} \\
          E_{2}=\frac{\lambdat(\mflo-\lambdat\Delta_{\tiny{\ref{subsec-D1third}}})}{p} (\lambdat\baset+\baseth)-\lambda^{2}(\gamma-1)\baseth   & \hbox{} \\
          E_{3}=(\mflo-\lambdat\Delta_{\tiny{\ref{subsec-D1third}}})\baseth. & \hbox{}
        \end{array}
      \right.$$
\end{prop}
Note that $S_{\roi}(\Baseo,\Baset,\Baseth)=S_{\roi}\big(\Baseo,\frac{\lambdat(\mflo-\lambdat)}{p} (\lambdat\baset+\baseth),(\mflo-\lambdat)\baseth\big)$ in this case since $\val{\lambda^{2}}\geq1>\val{\mflo-\lambdat\Delta_{\tiny{\ref{subsec-D1third}}}}=\val{\mflo-\lambdat}$.

\begin{proof}
During the proof, we let $\Delta:=\Delta_{\tiny{\ref{subsec-D1third}}}$ and $\mfm:=\mfm_{[\frac{1}{2},1]}$ for brevity.
It is routine to check that $\vphi(\Baseo)\equiv\Baseth$ and $\vphi(\Baset)\equiv\vphi(\Baseth)\equiv N(\Baseo)\equiv N(\Baset)\equiv N(\Baseth)\equiv 0$ modulo $\maxi{E}\mfm$. Hence, $\mfm$ is stable under $\vphi$ and $N$.

We check that $\vphi(\filt \mfm)\subset p^{2}\mfm$. It is easy to check that $\filt S\cdot\mfm\subset\filt\mfm$, and so it is enough to check $\vphi(\mathfrak{X}_{1})\in p^{2}\mfm$, since $\vphi(\filt S\cdot\mfm)\subset p^{2}\mfm$. We first compute  $\filt \mfm$ modulo $\filt S\cdot\mfm$.
$\mathfrak{X}_{1}\equiv C_{0}\big(\frac{1}{p}\Baseo-\frac{1}{\lambda\lambdat}\Baset+ \frac{p(\mflo-\lambdat\Delta)+p\lambdat\Delta}{\lambdat^{2}(\mflo-\lambdat\Delta)}\Baset+ \frac{\lambda(\mflo-\lambdat\Delta)+\lambda\lambdat\Delta}{\lambdat(\mflo-\lambdat\Delta)^{2}}\Baseth+ \frac{(1-\mflt)}{p(\mflo-\lambdat\Delta)}\Baseth- \frac{p\lambda^{2}(\mflo-\lambdat\Delta)+p\lambda^{2}\lambdat\Delta}{\lambdat^{2}(\mflo-\lambdat\Delta)^{2}}\Baseth- \frac{(\mflo-\lambdat)+\lambdat(1-\mflt)}{\lambdat(\mflo-\lambdat\Delta)}\Baseth\big)+ (u-p)\big(\frac{C_{1}}{p}\Baseo+ \frac{p\lambda(\mflo C_{1}+C_{2})-\lambdat(\mflo-\lambdat\Delta)C_{1}}{\lambda\lambdat^{2}(\mflo-\lambdat\Delta)}\Baset+\frac{V}{W}\Baseth\big)$ modulo $\filt S\cdot\mfm$, where $W=p\lambdat^{2}(\mflo-\lambdat\Delta)^{2}$ and
\begin{multline*}
V=\lambdat^{2}(\mflo-\lambdat\Delta)(C_{0}+p\mflt C_{1})+p\lambda\lambdat\mflo C_{1}+ \lambdat^{2}(1-\mflt)(\mflo-\lambdat\Delta)C_{1}-\\
 p^{2}\lambda^{2}(\mflo C_{1}+C_{2})-p\lambdat(\mflo-\lambdat\Delta)(\mflo C_{1}+C_{2}).
\end{multline*}
Hence, if $\mathfrak{X}_{1}\in\filt \mfm$, then we get $\val{C_{0}}\geq\val{\lambda\lambdat}=2-\val\lambda$, $\val{C_{1}}\geq1$,
\begin{equation}\label{D1-nineth}
\Val{p\lambda(\mflo C_{1}+C_{2})-\lambdat(\mflo-\lambdat\Delta)C_{1}}\geq\Val{\lambda\lambdat^{2}(\mflo-\lambdat\Delta)},
\end{equation}
and
\begin{equation}\label{D1-tenth}
\val{V}\geq \val{W}.
\end{equation}
The inequality (\ref{D1-nineth}) is equivalent to
\begin{equation} \label{D1-eleventh}
\Val{\lambdat(\mflo-\lambdat\Delta)C_{1}-p\lambda C_{2}}\geq\Val{\lambda\lambdat^{2}(\mflo-\lambdat\Delta)},
\end{equation}
which implies that $\val{C_{2}}\geq\Val{\frac{\lambdat}{\lambda}(\mflo-\lambdat\Delta)}$ since $\val{C_{1}}\geq1$. Using the inequalities $\val{C_{1}}\geq1$ and $\val{C_{2}}\geq\Val{\frac{\lambdat}{\lambda}(\mflo-\lambdat\Delta)}$, one can readily check that
\begin{multline*}
V\equiv\lambdat^{2}(\mflo-\lambdat\Delta)C_{0}+p\lambda\lambdat^{2}\Delta C_{1}+\lambdat^{2}(1-\mflt)(\mflo-\lambdat\Delta)C_{1}-\\ p\lambdat(\mflo-\lambdat\Delta)C_{2}+p\lambda[\lambdat(\mflo-\lambdat\Delta)C_{1}-p\lambda C_{2}]
\end{multline*}
modulo $(W)$. By the inequality (\ref{D1-eleventh}), $$V\equiv\lambdat^{2}(\mflo-\lambdat\Delta)C_{0}+p\lambda\lambdat^{2}\Delta C_{1}+\lambdat^{2}(1-\mflt)(\mflo-\lambdat\Delta)C_{1}- p\lambdat(\mflo-\lambdat\Delta)C_{2}$$ modulo $(W)$. Then
\begin{equation*}
\Delta(\mflo-\lambdat\Delta)\cdot V\equiv\lambdat(\mflo-\lambdat\Delta)^{2}(\lambdat\Delta C_{0}-p C_{2})+X
\end{equation*} modulo $((\mflo-\lambdat\Delta) W)$,
where
$$X=p\lambdat(\mflo-\lambdat\Delta)^{2}(1-\Delta)C_{2}+ p\lambda\lambdat^{2}(\mflo-\lambdat\Delta)\Delta^{2}C_{1}+ \lambdat^{2}(1-\mflt)(\mflo-\lambdat\Delta)^{2}\Delta C_{1}.$$ By the inequality (\ref{D1-eleventh}), $$X\equiv p\lambdat\big((\mflo-\lambdat\Delta)^{2}(1-\Delta)+ \lambda[(1-\mflt)(\mflo-\lambdat\Delta)+p\lambda\Delta]\Delta\big) C_{2}$$ modulo $((\mflo-\lambdat\Delta)W)$. But this is $0$ by the equation (\ref{D1-eighth}).
Hence, from the inequality (\ref{D1-tenth}), we have
\begin{equation}\label{D1-twelfth}
\val{\lambdat \Delta C_{0}-p C_{2}} \geq\val{p\lambdat(\mflo-\lambdat\Delta)}.
\end{equation}

Finally, we check $\vphi(\mathfrak{X}_{1})\in p^{2}\mfm$. Using the fact $\vphi(\frac{u-p}{p})\equiv\gamma-1$ modulo $pS$, $\vphi(\mathfrak{X}_{1})\equiv\lambda C_{0}\Big(\Baseo-\frac{p\Delta(\gamma-1)}{\lambda(\mflo-\lambdat\Delta)}\Baset\Big)+ p\lambda C_{1}(\gamma-1)\Big(\Baseo-\frac{\gamma-1}{\mflo-\lambdat\Delta}\Baseth -\frac{p\Delta(\gamma-1)}{\lambda(\mflo-\lambdat\Delta)}\Baset\Big)+ pC_{2}(\gamma-1)\Big(\frac{p}{\lambdat(\mflo-\lambdat\Delta)}\Baset+ \frac{p\lambda^{2}(\gamma-1)}{\lambdat(\mflo-\lambdat\Delta)^{2}}\Baseth\Big)$ modulo $p^{2}\mfm$. Since $\val{C_{0}}\geq2-\val\lambda$ and $\val{C_{1}}\geq1$, $\vphi(\mathfrak{X}_{1})\equiv -p\frac{\lambdat\Delta C_{0}-p C_{2}}{\lambdat (\mflo-\lambdat\Delta)}(\gamma-1)\Baset- p\lambda\frac{\lambdat(\mflo-\lambdat\Delta)C_{1}-p\lambda C_{2}}{\lambdat(\mflo-\lambdat\Delta)^{2}} (\gamma-1)^{2}\Baseth$ modulo $p^{2}\mfm$. Then $\vphi(\mathfrak{X}_{1})\equiv0$ modulo $p^{2}\mfm$ by the inequalities (\ref{D1-eleventh}) and (\ref{D1-twelfth}). Thus $\vphi(\filt \mfm)\subset p^{2}\mfm$.
\end{proof}

\section{Reduction modulo $p$}\label{sec-Breuil mod}
In this section, we study mod $p$ reductions of the irreducible semi-stable and non-crystalline representations of $G_{\QP}$ with Hodge--Tate weights $(0,1,2)$, by computing the Breuil modules corresponding to the mod $p$ reductions of the strongly divisible modules constructed in Sections~\ref{sec-strong div mod0} and~\ref{sec-strong div mod1}. We determine which of the representations has an absolutely irreducible mod $p$ reduction. We write $\mcf$ for $\FFF[u]/u^{p}$ in this section.
\subsection{Breuil modules of $D_{[0,\frac{1}{2}]}$}
The Breuil modules $\mcm_{[0,\frac{1}{2}]}:=\mcf(\Baseo,\Baset,\Baseth)$ corresponding to the mod $p$ reductions of the strongly divisible modules $\mfm_{[0,\frac{1}{2}]}$ (constructed in Section~\ref{sec-strong div mod0}) are computed as follows:

\noindent\textbf{(a):} the mod $p$ reductions of the modules in Proposition \ref{prop-D0first} correspond to
     \begin{itemize}
     \item $\mcm_{2}:=\mcf\big(u\Baseo,u\Baset,u\Baseth\big)$;
     \item $\vphi_{2}:\mcm_{2}\rightarrow \mcm$ is induced by
     $\left\{
       \begin{array}{ll}
         u\Baseo\mapsto\overline{\frac{\lambda(\mflo-1)}{p}}\Baseo-\Baset & \hbox{} \\
         u\Baset\mapsto \overline{\frac{\lambda(\mflt+p\lambda)}{p^{2}}}\Baseo-\overline{\frac{\lambdat}{p}}\Baset+ \overline{\frac{\lambda^{2}\lambdat}{p^{2}}}\Baseth & \hbox{} \\
         u\Baseth\mapsto \Baseo; & \hbox{}
       \end{array}
     \right.$
     \item $N:\mcm\rightarrow\mcm$ is induced by $N(\Baseo)=N(\Baset)=N(\Baseth)=0$.
     \end{itemize}
\noindent\textbf{(b):} the mod $p$ reductions of the modules in Proposition \ref{prop-D0second} correspond to
     \begin{itemize}
     \item $\mcm_{2}:=\mcf\big(u^{2}\Baseo, \Baset-\overline{\frac{(\mflt+p\lambda)-\lambdat(\mflo-1)}{p(\mflo-1)}}\Baseth,u\Baseth\big)$;
     \item $\vphi_{2}:\mcm_{2}\rightarrow \mcm$ is induced by
     $\left\{
       \begin{array}{ll}
         u^{2}\Baseo\mapsto\Baseth & \hbox{} \\
         \Baset-\overline{\frac{(\mflt+p\lambda)-\lambdat(\mflo-1)}{p(\mflo-1)}}\Baseth \mapsto -\overline{\frac{\lambda^{2}\lambdat}{p^{2}}}\Baseo & \hbox{} \\
         u\Baseth\mapsto \Baset; & \hbox{}
       \end{array}
     \right.$
     \item $N:\mcm\rightarrow\mcm$ is induced by $N(\Baseo)=N(\Baset)=N(\Baseth)=0$.
     \end{itemize}
\noindent\textbf{(c):} the mod $p$ reductions of the modules in Proposition \ref{prop-D0third} correspond to
         \begin{itemize}
     \item $\mcm_{2}:=\mcf\big(u\Baseo-\overline{\frac{p(\mflo-1)}{\mflt+p\lambda}}u\Baset, u^{2}\Baset, \Baseth\big)$;
     \item $\vphi_{2}:\mcm_{2}\rightarrow \mcm$ is induced by
     $\left\{
       \begin{array}{ll}
         u\Baseo-\overline{\frac{p(\mflo-1)}{\mflt+p\lambda}}u\Baset\mapsto-\Baset & \hbox{} \\
         u^{2}\Baset \mapsto \Baseth & \hbox{} \\
         \Baseth\mapsto \overline{\frac{\lambda^{2}\lambdat}{p^{2}}}\Baseo; & \hbox{}
       \end{array}
     \right.$
     \item $N:\mcm\rightarrow\mcm$ is induced by $N(\Baseo)=N(\Baset)=N(\Baseth)=0$.
     \end{itemize}

We check our computation of the Breuil Modules $\mcm_{[0,\frac{1}{2}]}$ below. We first prove that the Breuil modules in \textbf{(a)} correspond to the mod $p$ reductions of the strongly divisible modules in Proposition \ref{prop-D0first}. We keep the notation as in the proof of Proposition \ref{prop-D0first}. We may let $C_{1}-p\lambda C_{2}=\lambdat\alpha$ and $V=W\beta$ for some $\alpha,\beta\in\roi^{\times}$ by the inequalities (\ref{D0-first}) and (\ref{D0-second}). Recall that we have $\val{C_{0}}\geq2-\val\lambda=\val{\lambda\lambdat}$, $\val{C_{1}}\geq1$, and $\val{C_{2}}\geq-\val\lambda$. Using these inequalities, it is easy to check that $V\sequiv\lambda\lambdat C_{0}+\lambda(\mflt+p\lambda)(C_{1}-p\lambda C_{2})-\lambda\lambdat(\mflo-1)C_{1}=\lambda\lambdat C_{0}+\lambda\lambdat(\mflt+p\lambda)\alpha-\lambda\lambdat(\mflo-1)C_{1}$ modulo $(W)$. By the definition of the functor (\ref{BrMod-fuctor}), $u\Big(\overline{\frac{C_{1}}{p}}\Baseo-\overline{\alpha}\Baset+\overline{\beta}\Baseth\Big)\overset{\vphi_{2}}\mapsto \overline{\frac{p^{2}\lambdat\beta -\lambda\lambdat(\mflt+p\lambda)\alpha+\lambda\lambdat(\mflo-1)C_{1}}{p^{2}\lambdat}}\Baseo- \overline{\frac{C_{1}-\lambdat\alpha}{p}}\Baset-\overline{\frac{\lambda^{2}\lambdat\alpha}{p^{2}}}\Baseth$. Thus, $u\Baseo\overset{\vphi_{2}}\mapsto\overline{\frac{\lambda(\mflo-1)}{p}}\Baseo-\Baset$, $u\Baset\overset{\vphi_{2}}\mapsto \overline{\frac{\lambda(\mflt+p\lambda)}{p^{2}}}\Baseo-\overline{\frac{\lambdat}{p}}\Baset+ \overline{\frac{\lambda^{2}\lambdat}{p^{2}}}\Baseth$, and $u\Baseth\overset{\vphi_{2}}\mapsto\Baseo$. $N$ is immediate from the proof of Proposition \ref{prop-D0first}.

We also check that the Breuil modules in \textbf{(b)} correspond to the mod $p$ reductions of the strongly divisible modules in Proposition \ref{prop-D0second}. We keep the notation as in the proof of Proposition \ref{prop-D0second}. The inequality (\ref{D0-sixth}) with $\val{C_{0}}\geq2-\val\lambda$ implies that $\val{C_{2}}\geq\min\{0,1-\val{\lambda^{2}(\mflo-1)}\}>-\val\lambda$, and so we have $\val{C_{1}}\geq\min\{2-\val{\lambda(\mflo-1)},1+\val\lambda\}>1$ by the inequality (\ref{D0-fourth}). We may let $[\lambda(\mflo-1)-\lambdat\Delta]C_{1}-p\lambda^{2}(\mflo-1)C_{2}=\lambda\lambdat[\lambda(\mflo-1)-\lambdat\Delta]\alpha$ and $V=W\beta$ for some $\alpha,\beta\in\roi^{\times}$ by the inequalities (\ref{D0-fourth}) and (\ref{D0-fifth}). Using the inequalities $\val{C_{1}}>1$ and $\val{C_{2}}>-\val\lambda$, one can readily check that $\Delta[\lambda(\mflo-1)-\lambdat\Delta]V\sequiv\lambdat[\lambda(\mflo-1)-\lambdat\Delta]^{2}[\Delta C_{0}-p\lambda(\mflo-1)C_{2}]+\lambda\lambdat(\mflt+p\lambda)[\lambda(\mflo-1)-\lambdat\Delta]^{2}\Delta\alpha- \lambda\lambdat^{2}(\mflo-1)[\lambda(\mflo-1)-\lambdat\Delta]^{2}\Delta\alpha$ modulo $(p\lambdat[\lambda(\mflo-1)-\lambdat\Delta]^{3})$, by tracking the proof of the inequality (\ref{D0-sixth}). Hence, we have $\Big(-\overline{\frac{C_{0}}{\lambda\lambdat}}\Baset+ \overline{\frac{[(\mflt+p\lambda)-\lambdat(\mflo-1)]C_{0}}{p\lambda\lambdat(\mflo-1)}}\Baseth\Big)+ u\big(-\overline{\alpha}\Baset+\overline{\beta}\Baseth)\overset{\vphi_{2}}\mapsto\overline{\frac{\lambda C_{0}}{p^{2}}}\Baseo+\overline{\frac{p\lambda(\mflo-1)\beta-\lambda[(\mflt+p\lambda)-\lambdat(\mflo-1)] \Delta\alpha}{p\lambda(\mflo-1)}}\Baset- \overline{\frac{\lambda^{2}\lambdat\alpha}{p\lambda(\mflo-1)}}\Baseth$. (Recall that $\Delta\in 1+\maxid$ and $\frac{\lambda(\mflo-1)-\lambdat\Delta}{\lambda(\mflo-1)}\in1+\maxid$.) Thus, $\Baset-\overline{\frac{(\mflt+p\lambda)-\lambdat(\mflo-1)}{p(\mflo-1)}}\Baseth\overset{\vphi_{2}}\mapsto -\overline{\frac{\lambda^{2}\lambdat}{p^{2}}}\Baseo$, $u\Baset\overset{\vphi_{2}}\mapsto\overline{\frac{(\mflt+p\lambda)-\lambdat(\mflo-1)}{p(\mflo-1)}}\Baset$, and $u\Baseth\overset{\vphi_{2}}\mapsto\Baset$. Since $\vphi(\Baseo)\equiv\Baseth$ modulo $\maxid\mfm$, we also get $\vphi_{2}(u^{2}\Baseo)=\Baseth$. $N$ is immediate from the proof of Proposition \ref{prop-D0second}.

We finally show that the Breuil modules in \textbf{(c)} correspond to the mod $p$ reductions of the strongly divisible modules in Proposition \ref{prop-D0third}. We keep the notation as in the proof of Proposition \ref{prop-D0third}. We may let $V=W\alpha$ for some $\alpha\in\roi^{\times}$ by the inequality (\ref{D0-nineth}). Recall that $\val{C_{0}}\geq2-\val\lambda=\val{\lambda\lambdat}$ and $\val{C_{1}}\geq1$. By tracking the proof of the inequality (\ref{D0-eleventh}), one can readily check that $\Delta V\sequiv (\mflt+p\lambda\Delta)[\lambdat\Delta C_{0}+(\mflt+p\lambda)(C_{1}-p\lambda\Delta C_{2})]+p\lambdat(\mflt+p\lambda\Delta)^{2}C_{2}$ modulo $(W)$, which rearranges to $\lambdat(\mflt+p\lambda\Delta)\Delta C_{0}+(\mflt+p\lambda\Delta)^{2}(C_{1}-p\lambda C_{2})+p\lambda(\mflt+p\lambda\Delta)^{2}(1-\Delta)C_{2}+p\lambdat(\mflt+p\lambda\Delta)^{2}C_{2}= \lambdat(\mflt+p\lambda\Delta)\Delta C_{0}+(\mflt+p\lambda\Delta)^{2}(C_{1}-p\lambda C_{2})- p\lambda\lambdat[(\mflo-1)(\mflt+p\lambda\Delta)+p\lambdat\Delta]\Delta C_{2}+ p\lambdat(\mflt+p\lambda\Delta)^{2}C_{2}$. The last equality is due to the equation (\ref{D0-seventh}). This congruence implies that $\Delta V\sequiv(\mflt+p\lambda\Delta)^{2}(C_{1}-p\lambda C_{2})-p\lambda\lambdat(\mflo-1)(\mflt+p\lambda\Delta)C_{2}$ modulo $(\lambdat(\mflt+p\lambdat\Delta)^{2})$, which implies that $\val{C_{1}-p\lambda C_{2}}\geq\Val{p\lambdat(\mflo-1)}-\val{\mflt+p\lambda\Delta}>\val{\lambdat}$. Hence, we have $\overline{\frac{C_{0}}{\lambda\lambdat}}\Baseth+ u\Big(\overline{\lambda C_{2}}\Baseo-\overline{\frac{p\lambda(\mflo-1)C_{2}}{\mflt+p\lambda\Delta}}\Baset+ \overline{\alpha}\Baseth\Big)\overset{\vphi_{2}}\mapsto\overline{\frac{\lambda C_{0}}{p^{2}}}\Baseo-\overline{\lambda C_{2}}\Baset$. (Recall that $\Delta\in 1+\maxi{E}$ and $\frac{\mflt+p\lambda\Delta}{\mflt+p\lambda}\in 1+\maxi{E}$.) Thus, $\Baseth\overset{\vphi_{2}}\mapsto\overline{\frac{\lambda^{2}\lambdat}{p^{2}}}\Baseo$ and $u\Baseo-\overline{\frac{p(\mflo-1)}{\mflt+p\lambda}}u\Baset\overset{\vphi_{2}}\mapsto-\Baset$. Since $\vphi(\Baset)\equiv\Baseth$ modulo $\maxi{E}\mfm$, we also get $u^{2}\Baset\overset{\vphi_{2}}\mapsto\Baseth$. $N$ is immediate from the proof of Proposition \ref{prop-D0third}.

\begin{prop}\label{prop-Breuilfirst}
Assume that $0<\val{\lambda}\leq\frac{1}{2}$ and $2\val\lambda+\val\lambdat=2$. Then the reductions modulo $p$ of the semi-stable representations $\mathrm{V^{*}_{st}}(D_{[0,\frac{1}{2}]})$ of $G_{\QP}$ are absolutely irreducible if and only if either one of the following holds:
\begin{enumerate}
\item $\val{\mflt+p\lambda}<\Val{p(\mflo-1)}$ and $\val{\mflt+p\lambda}<\val{\lambda\lambdat}$;
\item $\Val{(\mflt+p\lambda)-\lambdat(\mflo-1)}>\Val{p(\mflo-1)}$ and $\val{\mflo-1}<1-\val{\lambda}$.
\end{enumerate}
\end{prop}

\begin{proof}
By Proposition \ref{non-simple Breuil module prop1}, the Breuil modules in \textbf{(a)} are non-simple. If $\Val{(\mflt+p\lambda)-\lambdat(\mflo-1)}>\Val{p(\mflo-1)}$, then $\overline{\frac{(\mflt+p\lambda)-\lambdat(\mflo-1)}{p(\mflo-1)}}=0$ in $\FFF$. Hence, the Breuil modules in \textbf{(b)} are simple if and only if $\Val{(\mflt+p\lambda)-\lambdat(\mflo-1)}>\Val{p(\mflo-1)}$ and $\val{\mflo-1}<1-\val\lambda$, by Propositions \ref{simple Breuil module prop} and \ref{non-simple Breuil module prop2}. Similarly, the Breuil modules in \textbf{(c)} are simple if and only if $\val{\mflt+p\lambda}<\Val{p(\mflo-1)}$ and $\val{\mflt+p\lambda}<\val{\lambda\lambdat}$.
\end{proof}

By Proposition~\ref{Prop-ad-filt-module[0,1/2]}, the reductions modulo $p$ of the representations corresponding to $D_{[0,\frac{1}{2}]}$ when $\val{\lambda}=\frac{1}{2}$ and $\mflt=0$ are reducible. This is consistent with the results in Proposition~\ref{prop-Breuilfirst}.

\subsection{Breuil modules of $D_{[\frac{1}{2},1]}$}
The Breuil modules $\mcm_{[\frac{1}{2},1]}:=\mcf(\Baseo,\Baset,\Baseth)$ corresponding to the mod $p$ reductions of the strongly divisible modules $\mfm_{[\frac{1}{2},1]}$ (constructed in Section~\ref{sec-strong div mod1}) are computed as follows:

\noindent\textbf{(a):} the mod $p$ reductions of the modules in Proposition \ref{prop-D1first} correspond to
     \begin{itemize}
     \item $\mcm_{2}:=\mcf\big(u\Baseo,u\Baset,u\Baseth\big)$;
     \item $\vphi_{2}:\mcm_{2}\rightarrow \mcm$ is induced by
     $\left\{
       \begin{array}{ll}
         u\Baseo \mapsto \overline{\frac{(\mflo-\lambdat)-\lambda(1-\mflt)}{p}}\Baseo-\overline{\frac{\lambda\lambdat(1-\mflt)}{p^{2}}}\Baset- \Baseth & \hbox{} \\
         u\Baset\mapsto \Baseo & \hbox{} \\
         u\Baseth\mapsto \overline{\frac{\lambda^{2}}{p}}\Baseo+\overline{\frac{\lambda^{2}\lambdat}{p^{2}}}\Baset; & \hbox{}
       \end{array}
     \right.$
     \item $N:\mcm\rightarrow\mcm$ is induced by $N(\Baseo)=N(\Baset)=N(\Baseth)=0$.
     \end{itemize}
\noindent\textbf{(b):} the mod $p$ reductions of the modules in Proposition \ref{prop-D1second} correspond to
     \begin{itemize}
     \item $\mcm_{2}:=\mcf\big(u\Baseo-\overline{\frac{p[(\mflo-\lambdat)-\lambda(1-\mflt)]}{\lambda\lambdat(1-\mflt)}}u\Baset, u^{2}\Baset, \Baseth\big)$;
     \item $\vphi_{2}:\mcm_{2}\rightarrow \mcm$ is induced by
     $\left\{
       \begin{array}{ll}
         u\Baseo-\overline{\frac{p[(\mflo-\lambdat)-\lambda(1-\mflt)]}{\lambda\lambdat(1-\mflt)}}u\Baset \mapsto-\Baset & \hbox{} \\
         u^{2}\Baset\mapsto \Baseth & \hbox{} \\
         \Baseth\mapsto \overline{\frac{\lambda^{2}\lambdat}{p^{2}}}\Baseo; & \hbox{}
       \end{array}
     \right.$
     \item $N:\mcm\rightarrow\mcm$ is induced by $N(\Baseo)=N(\Baset)=N(\Baseth)=0$.
     \end{itemize}
\noindent\textbf{(c):} the mod $p$ reductions of the modules in Proposition \ref{prop-D1third} correspond to
     \begin{itemize}
     \item $\mcm_{2}:=\mcf\big(u^{2}\Baseo, \Baset-\overline{\frac{\lambda\lambdat(1-\mflt)}{p(\mflo-\lambdat)}}\Baseth,u\Baseth\big)$;
     \item $\vphi_{2}:\mcm_{2}\rightarrow \mcm$ is induced by
     $\left\{
       \begin{array}{ll}
         u^{2}\Baseo\mapsto \Baseth & \hbox{} \\
         \Baset-\overline{\frac{\lambda\lambdat(1-\mflt)}{p(\mflo-\lambdat)}}\Baseth\mapsto -\overline{\frac{\lambda^{2}\lambdat}{p^{2}}}\Baseo & \hbox{} \\
         u\Baseth\mapsto \Baset; & \hbox{}
       \end{array}
     \right.$
     \item $N:\mcm\rightarrow\mcm$ is induced by $N(\Baseo)=N(\Baset)=N(\Baseth)=0$.
     \end{itemize}

We check our computation of the Breuil modules $\mcm_{[\frac{1}{2},1]}$ below. We first prove that the Breuil modules in \textbf{(a)} correspond to the mod $p$ reductions of the strongly divisible modules in Proposition \ref{prop-D1first}. We keep the notation as in the proof of the Proposition \ref{prop-D1first}. We may let $\lambdat(\mflo-\lambdat)C_{1}-p\lambda\mflo C_{1}-p\lambda C_{2}=p^{2}\lambdat\alpha$ and $V=W\beta$ for some $\alpha,\beta\in \roi^{\times}$ by the inequalities (\ref{D1-first}) and (\ref{D1-second}). Since $\val{C_{0}}\geq2-\val\lambda$ and $\val{C_{1}}\geq1$, we have $\lambdat(\mflo-\lambdat)C_{1}-p\lambda C_{2}\sequiv p^{2}\lambdat\alpha$ modulo $(p^{2}\lambdat)$ and $\lambdat C_{0}-pC_{2}+\lambdat(1-\mflt)C_{1}\sequiv p\lambda\lambdat\beta$ modulo $(p\lambda\lambdat)$, both of which imply that $\lambda C_{0}+\lambda(1-\mflt)C_{1}-(\mflo-\lambdat)C_{1}\sequiv p\lambda^{2}\beta-p^{2}\alpha$ modulo $(p^{2})$. By the definition of the functor (\ref{BrMod-fuctor}), we have $u\Big(\overline{\frac{C_{1}}{p}}\Baseo- \overline{\alpha}\Baset+\overline{\beta}\Baseth\Big)\overset{\vphi_{2}}\mapsto \overline{\frac{p\lambda^{2}\beta-p^{2}\alpha-\lambda(1-\mflt)C_{1}+(\mflo-\lambdat)C_{1}}{p^{2}}}\Baseo+ \overline{\frac{p\lambda^{2}\lambdat\beta-\lambda\lambdat(1-\mflt)C_{1}}{p^{3}}}\Baset-\overline{\frac{C_{1}}{p}}\Baseth$. Hence, $u\Baseo\overset{\vphi_{2}}\mapsto\overline{\frac{(\mflo-\lambdat)-\lambda(1-\mflt)}{p}}\Baseo- \overline{\frac{\lambda\lambdat(1-\mflt)}{p^{2}}}\Baset-\Baseth$, $u\Baset\overset{\vphi_{2}}\mapsto\Baseo$, and $u\Baseth\overset{\vphi_{2}}\mapsto\overline{\frac{\lambda^{2}}{p}}\Baseo+ \overline{\frac{\lambda^{2}\lambdat}{p^{2}}}\Baset$. $N$ is immediate from the proof of Proposition \ref{prop-D1first}.

We also check that the Breuil modules in \textbf{(b)} correspond to the mod $p$ reductions of the strongly divisible modules in Proposition \ref{prop-D1second}. We keep the notation as in the proof of Proposition \ref{prop-D1second}.
We may let $V=W\alpha$ for some $\alpha\in\roi^{\times}$ by the inequality~(\ref{D1-fifth}). Recall that $\val{C_{0}}\geq2-\val\lambda$ and $\val{C_{1}}\geq1$. By tracking the proof of the inequality (\ref{D1-sixth}), we have $\Delta V\sequiv[\lambdat(1-\mflt)+p\lambda\Delta]\big(\lambdat^{2}(1-\mflt)\Delta C_{0}+\lambdat(1-\mflt)[\lambdat(1-\mflt)+p\lambda\Delta]C_{1}-p[\lambdat(1-\mflt)+p\lambda\Delta]\Delta C_{2}\big)-p(\mflo-\lambdat)[\lambdat(1-\mflt)+p\lambda\Delta]^{2}\Delta C_{1}$ modulo $(W)$, which implies that
$\lambdat^{2}(1-\mflt)\Delta C_{0}+\lambdat(1-\mflt)[\lambdat(1-\mflt)+p\lambda\Delta]C_{1}-p[\lambdat(1-\mflt)+p\lambda\Delta]\Delta C_{2}\big)-p(\mflo-\lambdat)[\lambdat(1-\mflt)+p\lambda\Delta]\Delta C_{1}\sequiv\lambda\lambdat^{2}(1-\mflt)[\lambdat(1-\mflt)+p\lambda\Delta]\Delta\alpha$ modulo $(\lambda\lambdat^{2}(1-\mflt)[\lambdat(1-\mflt)+p\lambda\Delta])$. Hence, we have $\overline{\frac{C_{0}}{\lambda\lambdat}}\Baseth+ u\Big(\overline{\frac{C_{1}}{p}}\Baseo+ \overline{\frac{[\lambda(1-\mflt)-(\mflo-\lambdat)]C_{1}}{\lambda\lambdat(1-\mflt)}}\Baset +\overline{\alpha}\Baseth\Big)\overset{\vphi_{2}}\mapsto\overline{\frac{\lambda C_{0}}{p^{2}}}\Baseo-\overline{\frac{C_{1}}{p}}\Baset$. (Recall that $\Delta\in 1+\maxi{E}$ and $\frac{\lambdat(1-\mflt)+p\lambda\Delta}{\lambdat(1-\mflt)}\in 1+\maxi{E}$.) Thus, $\Baseth\overset{\vphi_{2}}\mapsto\overline{\frac{\lambda^{2}\lambdat}{p^{2}}}\Baseo$ and $u\Big(\Baseo+\overline{\frac{p\lambda(1-\mflt)-p(\mflo-\lambdat)}{\lambda\lambdat(1-\mflt)}}\Baset\Big) \overset{\vphi_{2}}\mapsto-\Baset$. Since $\vphi(\Baset)\equiv\Baseth$ modulo $\maxi{E}\mfm$, we also have $u^{2}\Baset\overset{\vphi_{2}}\mapsto\Baseth$. $N$ is immediate from the proof of Proposition \ref{prop-D1second}.

We finally show that the Breuil modules in \textbf{(c)} correspond to the mod $p$ reductions of the strongly divisible modules in Proposition \ref{prop-D1third}. We keep the notation as in the proof of Proposition \ref{prop-D1third}. The inequality (\ref{D1-twelfth}) with $\val{C_{0}}\geq2-\val\lambda$ implies that $\val{C_{2}}\geq\min\{\val{\lambdat(\mflo-\lambdat)},\val{\frac{p\lambdat}{\lambda}}\}> \Val{\frac{\lambdat}{\lambda}(\mflo-\lambdat)}$, and so we have $\val{C_{1}}\geq\min\{\val{p\lambda},2-\val{\mflo-\lambdat},\val{\lambda\lambdat}\}>1$ by the inequality (\ref{D1-eleventh}). We may let $p\lambda(\mflo C_{1}+C_{2})-\lambdat(\mflo-\lambdat\Delta)C_{1}=\lambda\lambdat^{2}(\mflo-\lambdat\Delta)\alpha$ and $V=W\beta$ for some $\alpha,\beta\in\roi^{\times}$ by the inequalities (\ref{D1-nineth}) and (\ref{D1-tenth}). Using the inequalities $\val{C_{1}}>1$ and $\val{C_{2}}>\Val{\frac{\lambdat}{\lambda}(\mflo-\lambdat\Delta)}$, one can readily check that $\lambda\lambdat^{2}(\mflo-\lambdat\Delta)\alpha\sequiv p\lambda C_{2}-\lambdat(\mflo-\lambdat\Delta)C_{1}$ modulo $(\lambda\lambdat^{2}(\mflo-\lambdat\Delta))$ and $\Delta(\mflo-\lambdat\Delta)W\beta\sequiv \lambdat(\mflo-\lambdat\Delta)^{2}(\lambdat\Delta C_{0}-pC_{2})-\lambda\lambdat^{3}(1-\mflt)(\mflo-\lambdat\Delta)^{2}\Delta\alpha$ modulo $((\mflo-\lambda\Delta)W)$, by tracking the proof of the inequality (\ref{D1-twelfth}). Hence, we have $\Big(-\overline{\frac{C_{0}}{\lambda\lambdat}}\Baset+ \overline{\frac{(1-\mflt)C_{0}}{p(\mflo-\lambdat\Delta)}}\Baseth\Big)+ u\big(\overline{\alpha}\Baset+\overline{\beta}\Baseth)\overset{\vphi_{2}}\mapsto\overline{\frac{\lambda C_{0}}{p^{2}}}\Baseo+\overline{\frac{p\lambdat(\mflo-\lambdat\Delta)\Delta\beta+ \lambda\lambdat^{2}(1-\mflt)\Delta\alpha}{p\lambdat(\mflo-\lambdat\Delta)}}\Baset- \overline{\frac{\lambda^{2}\lambdat\alpha}{p(\mflo-\lambdat\Delta)}}\Baseth$. (Recall that $\Delta\in 1+\maxid$ and $\frac{\mflo-\lambdat\Delta}{\mflo-\lambdat}\in1+\maxid$.) Thus, $\Baset-\overline{\frac{\lambda\lambdat(1-\mflt)}{p(\mflo-\lambdat)}}\Baseth\overset{\vphi_{2}}\mapsto -\overline{\frac{\lambda^{2}\lambdat}{p^{2}}}\Baseo$, $u\Baset\overset{\vphi_{2}}\mapsto\overline{\frac{\lambda\lambdat(1-\mflt)}{p(\mflo-\lambdat)}}\Baset$, and $u\Baseth\overset{\vphi_{2}}\mapsto\Baset$. Since $\vphi(\Baseo)\equiv\Baseth$ modulo $\maxid\mfm$, $\vphi_{2}(u^{2}\Baseo)=\Baseth$. $N$ is immediate from the proof of Proposition \ref{prop-D1third}.

\begin{prop}\label{prop-Breuilsecond}
Assume that $\frac{1}{2}\leq\val\lambda<1$ and $2\val\lambda+\val\lambdat=2$. Then the reductions modulo $p$ of the semi-stable representations $\mathrm{V^{*}_{st}}(D_{[\frac{1}{2},1]})$ of $G_{\QP}$ are absolutely irreducible if and only if either one of the following holds:
\begin{enumerate}
\item $\Val{\lambda(\mflo-\lambdat)}<\Val{p(1-\mflt)}$ and $\Val{\mflo-\lambdat}<1$;
\item $\Val{p[(\mflo-\lambdat)-\lambda(1-\mflt)]}>\Val{\lambda\lambdat(1-\mflt)}$ and $\val{1-\mflt}<\val\lambda$.
\end{enumerate}
\end{prop}

\begin{proof}
The same argument as in the proof of Proposition~\ref{prop-Breuilfirst} works.
\end{proof}

By Proposition~\ref{Prop-ad-filt-module[1/2,1]}, the reductions modulo $p$ of the representations corresponding to $D_{[\frac{1}{2},1]}$ when $\val{\lambda}=\frac{1}{2}$ and $\mflo=0$ are reducible. This is consistent with the results in Proposition~\ref{prop-Breuilsecond}.

\subsection{Some remarks}
\subsubsection{} Our computation of Breuil modules says that the Breuil modules in Example~\ref{simple Breuil module exam} are exactly the simple Breuil modules that occur as irreducible mod $p$ reductions of the semi-stable and non-crystalline representations of $G_{\QP}$ with Hodge--Tate weights $(0,1,2)$. The irreducible mod $p$ reductions of the strongly divisible modules in Subsections~\ref{subsec-D0second} and \ref{subsec-D1third} (resp., in Subsections~\ref{subsec-D0third} and \ref{subsec-D1second}) correspond to the Breuil modules $\mcm(s,a,b,c)$ (resp., to the Breuil modules $\mcm(s^{2},a,b,c)$) for some $a,b,c\in\FFF^{\times}$. Hence, we conclude, by Proposition~\ref{simple Breuil module prop}, that if $\bar{\rho}:G_{\QP}\rightarrow\mathrm{GL}_{3}(\FPB)$ is an irreducible mod $p$ reduction of a semi-stable and non-crystalline representation with Hodge-Tate weights $(0,1,2)$, then
$\bar{\rho}|_{I_{\QP}}$ is isomorphic to either
\begin{equation*}
\omega_{3}^{2p+1}\oplus\omega_{3}^{2p^{2}+p}\oplus\omega_{3}^{2+p^{2}}\hspace{0.1cm}\mbox{  or  }\hspace{0.1cm} \omega_{3}^{p+2}\oplus\omega_{3}^{p^{2}+2p}\oplus\omega_{3}^{1+2p^{2}}
\end{equation*}
where $\omega_{3}$ is the fundamental character of level $3$.

\subsubsection{}We also claim that if $\val\lambda=\frac{1}{2}$ then the condition (1) (resp., (2)) in Proposition~\ref{prop-Breuilfirst} is equivalent to (2) (resp., (1)) in Proposition~\ref{prop-Breuilsecond} in terms of the identification in Proposition~\ref{Prop-ad-filt-module} (as it should be due to the Proposition~\ref{Prop-ad-filt-module}).

Indeed, we let $D_{[0,\frac{1}{2}]}=D_{[0,\frac{1}{2}]}(\lambda,\lambdat,\mflo,\mflt)$ and $D_{[\frac{1}{2},1]}=D_{[\frac{1}{2},1]}(\lambda,\lambdat,\mflo',\mflt')$ and assume that $\val{\lambda}=\frac{1}{2}$. Then, the condition (2) in Proposition~\ref{prop-Breuilfirst} holds if and only if $\Val{\mflt-\lambdat(\mflo-1)}>\Val{p(\mflo-1)}$ and $\val{\mflo-1}<\val\lambda$, if and only if $\Val{\lambda[\mflt-(\lambdat-p\lambda)(\mflo-1)]}>\Val{p(\lambda-\lambdat)(\mflo-1)}$ and $\Val{(\lambda-\lambdat)(\mflo-1)}<1$, if and only if $\Val{\lambda(\lambdat-p\lambda)(1-\mflt')}> \Val{p[(\lambda-\lambdat)(\mflt'-1)+\mflo']}$ and $\Val{(\lambda-\lambdat)(\mflt'-1)+\mflo'}<1$ (by the identification in Proposition~\ref{Prop-ad-filt-module}), if and only if $\Val{\lambda(1-\mflt')}>\val{\mflo'}$ and $\Val{\mflo'}<1$, if and only if the condition (1) in Proposition~\ref{prop-Breuilsecond} holds. Similarly, the condition (2) in Proposition~\ref{prop-Breuilsecond} holds if and only if $\Val{p[\lambda(1-\mflt')-\mflo']}>\Val{\lambda\lambdat(1-\mflt')}$ and $\val{1-\mflt'}<\val{\lambda}$, if and only if
$\Val{p[(\lambda-\lambdat)(1-\mflt')-\mflo']}>\Val{\lambda(\lambdat-p\lambda)(1-\mflt')}$ and $\Val{(\lambdat-p\lambda)(1-\mflt')}<\val{\lambda\lambdat}$, if and only if $\Val{p(\lambda-\lambdat)(1-\mflo)}>\Val{\lambda[(\lambdat-p\lambda)(1-\mflo)+\mflt]}$ and $\Val{(\lambdat-p\lambda)(1-\mflo)+\mflt}<\val{\lambda\lambdat}$ (by the identification in Proposition~\ref{Prop-ad-filt-module}), if and only if $\Val{p(\mflo-1)}>\val{\mflt}$ and $\val\mflt<\val{\lambda\lambdat}$, if and only if the condition (1) in Proposition~\ref{prop-Breuilfirst} holds.

Moreover, if $\val\lambda=\frac{1}{2}$, then the images of the strongly divisible modules in Subsection~\ref{subsec-D0second} (resp., in Subsection~\ref{subsec-D0third}) under the isomorphism~(\ref{association from 0 to 1}) are homothetic to the strongly divisible modules in Subsection~\ref{subsec-D1third} (resp., in Subsection~\ref{subsec-D1second}), provided that the condition (2) (resp., the condition (1)) in Proposition~\ref{prop-Breuilfirst} and the condition (1) (resp., the condition (2)) in Proposition~\ref{prop-Breuilsecond} hold, in terms of the identification in Proposition~\ref{Prop-ad-filt-module}, since a $p$-adic representation whose mod $p$ reduction is absolutely irreducible has a unique Galois stable lattice up to homothety.

\subsubsection{} Both of the families of strongly divisible modules of $D_{[0,\frac{1}{2}]}$ (resp., of $D_{[\frac{1}{2},1]}$) in Subsections~\ref{subsec-D0second} and~\ref{subsec-D0third} (resp., in Subsections~\ref{subsec-D1second} and ~\ref{subsec-D1third}) are defined when $0<\val\lambda<\frac{1}{2}$, $\val{\mflt+p\lambda}=\val{p(\mflo-1)}$, and $\val{\mflo-1}<\val{\frac{p}{\lambda}}$ (resp., when $\frac{1}{2}<\val\lambda<1$, $\val{\lambda(\mflo-\lambdat)}=\val{p(1-\mflt)}$, and $\val{\mflo-\lambdat}<1$.) They are obviously not homothetic, and so the reduction modulo $p$ of the corresponding representations are reducible, which is consistent with the results in Proposition~\ref{prop-Breuilfirst} (resp., in Proposition~\ref{prop-Breuilsecond}.)

\section*{Acknowledgements}
The author is deeply indebted to David Savitt for kindly suggesting the problem and for his inspiring guidance and encouragement throughout the progress of this work. The author would also like to thank Florian Herzig for numerous helpful comments and suggestions.

\bibliographystyle{alpha}

\end{document}